%% file: RTMjournalv14.tex
\documentclass[oneside,english]{amsart}

\usepackage{setspace} 
\usepackage{amssymb}
\usepackage[]{epsf,epsfig,amsmath,amssymb,amsfonts,latexsym}

\usepackage{amsmath}
\usepackage{amsthm}
\usepackage{mathtools}
\usepackage{lmodern}
\usepackage{mathabx}
\usepackage{mathrsfs}
\usepackage{wasysym}
\usepackage{bbm}		
\usepackage[inline]{enumitem}
\usepackage{nccmath}	
\usepackage{braket}		
\usepackage[normalem]{ulem}		
\usepackage{nicefrac}
\usepackage{cancel}
\usepackage{xcolor}
\usepackage[a4paper, left=2.7cm, right=2.7cm, top=3.5cm, bottom=2cm]{geometry}
\usepackage[hyperindex=true,pdfborder={0 0 0}]{hyperref}
\usepackage{doi}
\usepackage[numbers]{natbib}
\usepackage{cleveref}
\usepackage{leftidx} 
\usepackage{comment}
\usepackage{graphicx,xcolor}
\usepackage{float}
\usepackage{mathrsfs}
\usepackage{epstopdf}
\usepackage{thmtools}
\usepackage{stmaryrd}
\usepackage{microtype}
\setlist[itemize]{noitemsep, topsep=1pt}
\setlist[enumerate]{noitemsep, topsep=1pt}

\usepackage{tikz}
\usetikzlibrary{patterns,positioning,arrows,decorations.markings,calc,decorations.pathmorphing,decorations.pathreplacing}
\usetikzlibrary{lindenmayersystems}
\usetikzlibrary{arrows}
\usetikzlibrary{calc}
\usepackage{tikz-3dplot}

\theoremstyle{plain}
\newtheorem{theorem}{Theorem}[section]
\newtheorem{lemma}[theorem]{Lemma}
\newtheorem{corollary}[theorem]{Corollary}
\newtheorem{proposition}[theorem]{Proposition}

\newtheorem{definition}[theorem]{Definition}

\theoremstyle{definition}
\newtheorem{remark}[theorem]{Remark}

\def\NN{\mathbb{N}}
\def\QQ{\mathbb{Q}}
\def\ZZ{\mathbb{Z}}
\def\ZZd{\mathbb{Z}^d}
\def\RR{\mathbb{R}}

\def\ag{\mathcal{A}}
\def\bg{\mathcal{B}}

\def\FF{\mathcal{F}}

\def\EE{\mathbb{E}}
\newcommand{\End}{\mathrm{End}}
\newcommand{\Aut}{\mathrm{Aut}}
\newcommand{\TM}{\mathrm{TM}}
\newcommand{\RTM}{\mathrm{RTM}}
\newcommand{\OB}{\mathrm{OB}}
\newcommand{\LP}{\mathrm{LP}}
\newcommand{\FG}{\mathrm{RFA}}
\newcommand{\RFA}{\mathrm{RFA}}
\newcommand{\EL}{\mathrm{EL}}
\newcommand{\SHIFT}{\mathrm{SHIFT}}
\newcommand{\SP}{\mathrm{SP}}
\newcommand{\define}[1]{\textbf{#1}}
\newcommand{\AOS}{\mathrm{AOS}}
\newcommand{\OS}{\mathrm{OS}}
\newcommand{\CL}{\mathrm{CL}}

\newcommand{\ID}{\mathrm{id}}
\newcommand{\fix}{\mathrm{fix}}

\newcommand{\Sym}{\mathrm{Sym}}
\newcommand{\Alt}{\mathrm{Alt}}

\title{The group of reversible Turing machines: subgroups, generators and computability}

\author{Sebasti\'an Barbieri, Jarkko Kari \and Ville Salo}

	\date{}

\begin{document}
	
		\begin{abstract}
			We study an abstract group of reversible Turing machines. In our model, each machine is interpreted as a homeomorphism over a space which represents a tape filled with symbols and a head carrying a state. These homeomorphisms can only modify the tape at a bounded distance around the head, change the state and move the head in a bounded way. We study three natural subgroups arising in this model: the group of finite-state automata, which generalizes the topological full groups studied in topological dynamics and the theory of orbit-equivalence; the group of oblivious Turing machines whose movement is independent of tape contents, which generalizes lamplighter groups and has connections to the study of universal reversible logical gates; and the group of elementary Turing machines, which are the machines which are obtained by composing finite-state automata and oblivious Turing machines.
			We show that both the group of oblivious Turing machines and that of elementary Turing machines are finitely generated, while the group of finite-state automata and the group of reversible Turing machines are not. We show that the group of elementary Turing machines has undecidable torsion problem. From this, we also obtain that the group of cellular automata (more generally, the automorphism group of any uncountable one-dimensional sofic subshift) contains a finitely-generated subgroup with undecidable torsion problem. We also show that the torsion problem is undecidable for the topological full group of a full $\ZZ^d$-shift on a non-trivial alphabet if and only if $d \geq 2$.
		\end{abstract}

		\maketitle

		\medskip	
		
        \smallskip
		
		\textit{Secondary:}
		37B10, 
		37B50. 

\section{Introduction}\label{section.introduction}

This article is the extended version of the conference paper~\cite{BaKaSa16}. 
The main new results that did not appear in the conference paper are that elementary Turing machines are finitely generated, and that automorphism groups of uncountable sofic $\ZZ$-shifts have finitely-generated subgroups with undecidable torsion problem.

\subsection{Turing machines and their generalization}

Turing machines have been studied since the 30s as the standard formalization of the abstract concept of computation. However, more recently, Turing machines have also been studied in the context of dynamical systems. In~\cite{Ku97}, two dynamical systems were associated to a Turing machine, one with a ``moving tape'' and one with a ``moving head''. After that, there has been a lot of study of dynamics of Turing machines, see for example~\cite{DeKuBl06,BlCaNi02,KaOl08,GaMa07,GaGu10,Je13,GaOlTo15}. Another connection between Turing machines and dynamics is that they can be used to describe ``effectively closed'' zero-dimensional dynamical systems. A particularly interesting case is that of subshifts whose forbidden patterns are enumerated by a Turing machine. These subshifts are called effectively closed, or $\Pi^0_1$ subshifts, and especially in multiple dimensions, they are central to the topic due to the strong links known between SFTs, sofic shifts and $\Pi^0_1$-subshifts, see for example~\cite{Hochman2009b,DuRoSh10,AuSa13}. An intrinsic notion of Turing machine computation for these subshifts on general groups was proposed in~\cite{AuBaSa14}, and a similar study was performed with finite state machines in~\cite{SaTo15,SaTo15c}.

In all these papers, the definition of a Turing machine is (up to mostly notational differences and switching between the moving tape and moving head model) the following: A Turing machine is a function $T\colon \Sigma^\ZZ \times Q \to \Sigma^\ZZ \times Q$ defined by a local rule $f_T\colon \Sigma \times Q \to \Sigma \times Q \times \{-1,0,1\}$ by the formula
\[ T(x, q) = (\sigma^{-d}(\tilde{x}), q') \mbox{ if } f_T(x_0, q) = (a,q',d), \]
where $\sigma\colon \Sigma^{\ZZ} \to \Sigma^{\ZZ}$ is the shift action given by $\sigma^{d}(x)_{z} = x_{z-d}$, $\tilde{x}_0 = a$ and $\tilde{x}|_{\ZZ \setminus \{0\}} = x|_{\ZZ \setminus \{0\}}$. In this paper, such Turing machines are called \define{classical Turing machines}. This definition (as far as we know) certainly suffices to capture all computational and dynamical properties of interest, but it also has some undesirable properties: The composition of two classical Turing machines -- and even the square of a classical Turing machine -- is typically not a classical Turing machine, and the reverse of a reversible classical Turing machine is not always a classical Turing machine.

In this paper, we give a more general definition of a Turing machine, by allowing it to move the head and modify cells at an arbitrary (but bounded) distance on each timestep. With the new definition, we get rid of both issues: With our definition,
\begin{itemize}
	\item Turing machines are closed under composition, forming a monoid, and 
	\item reversible Turing machines are closed under inversion, forming a group.
\end{itemize}
We also characterize reversibility of classical Turing machines in combinatorial terms, and show what their inverses look like.
Our definition of a Turing machine originated in the yet unpublished work of M. Schraudner and the last author, where the group of such machines was studied on general subshifts (with somewhat different objectives). The same definition was given in 1991 by Moore~\cite{Mo91} under the name ``generalized shifts''.

These benefits of the definition should be compared to the benefits of allowing arbitrary radii in the definition of a cellular automaton: If we define cellular automata as having a fixed radius of, say, $3$, then the inverse map of a reversible cellular automaton is not always a cellular automaton, as the inverse of a cellular automaton may have a much larger radius, see~\cite{CzKa07}. Similarly, with a fixed radius, the composition of two cellular automata is not necessarily a cellular automaton.

We give our Turing machine definitions in two ways, with a moving tape and with a moving head, as done in~\cite{Ku97}. The moving tape point of view is often the more useful one when studying one-step behavior and invariant measures, whereas we find the moving head point of view easier for constructing examples, and when we need to track the movement of multiple heads. The moving head Turing machines are in fact a subset of cellular automata on a particular kind of subshift. The moving tape machine on the other hand is a generalization of the topological full group of a subshift, which is an important concept in particular in topological dynamics and the theory of orbit equivalence. For the study of topological full groups of minimal subshifts and their interesting group-theoretic properties, see for example~\cite{GiPuSk99,GrMe11,JuMo12}. The (one-sided) SFT case is studied in~\cite{Ma12}. We shall show that our two Turing machine models yield isomorphic monoids, and isomorphic groups in the case of reversible Turing machines.

\subsection{Our results and comparisons with other groups}

In Section~\ref{section.preliminaries}, we define our models and prove basic results about them. In Section~\ref{subsection_measure}, we define a natural uniform measure on these spaces and use it to show that injectivity and surjectivity are both equal to reversibility in our model.

Our results have interesting counterparts in the theory of cellular automata: One of the main theorems in the theory of cellular automata is that on a large class of groups (the \define{surjunctive groups}, see for instance Section 3 of~\cite{CeCo10}) injectivity implies surjectivity, and (global) bijectivity is equivalent to having a cellular automaton inverse map. Furthermore, one can attach to a reversible one- or two-dimensional cellular automaton its ``average drift'', that is, the speed at which information moves when the map is applied, and this is a homomorphism from the group of cellular automata to a subgroup of $\QQ^d$ under multiplication (where $d$ is the corresponding dimension), see \cite{Ka96}. In Section~\ref{section.subgroups} we use the uniform measure to define an analog, the ``average movement'' homomorphism for Turing machines.


In Section~\ref{section.subgroups}, we define some interesting subgroups of the group of reversible Turing machines. First, we define the local permutations -- Turing machines that never move the head at all --, and their generalization to oblivious Turing machines where movement is allowed, but is independent of the tape contents. The group of oblivious Turing machines can be seen as a kind of generalization of lamplighter groups. It turns out that the group of oblivious Turing machines is finitely generated. Our proof relies strongly on the existence of universal reversible logical gates, see~\cite{AaGrSc15}. 

We also define the group of (reversible) finite-state machines -- Turing machines that never modify the tape. This group is not finitely generated, but we give a natural infinite generating set for it. Finite-state machines with a single state exactly correspond to the topological full groups of full shifts, and in this sense our definition of a reversible finite-state machine can be seen as a generalization of the topological full group on a full $\ZZ^d$-shift.

Our original motivation for defining these subgroups -- finite-state machines and local permutations -- was to study the question of whether they generate all reversible Turing machines. Namely, a reversible Turing machine changes the tape contents at the position of the head and then moves, in a globally reversible way. Thus, it is a natural question whether every reversible Turing machine can actually be split into reversible tape changes (actions by local permutations) and reversible moves (finite-state automata).

We call the join of finite-state machines and local permutations elementary Turing machines, and show that not all Turing machines are elementary, as Turing machines can have arbitrarily small average movement, but elementary ones have only a discrete sublattice of possible average movements. However, we show that every reversible classical Turing machine (and thus all its iterates) is elementary.

A reason why elementary Turing machines are so interesting is that, while we show that the group of reversible Turing machines is not finitely generated (Theorem~\ref{RTM_is_not_finitely_generated}), we have that the group of elementary Turing machines on a one-dimensional tape is finitely generated, for any fixed alphabet size and number of states.

\begin{theorem}\label{theorem_EL_is_FG}
	The group of elementary Turing machines on $n$ symbols, $k$ states and dimension $1$ is finitely generated for every $n,k\geq 1$.
\end{theorem}


In Section~\ref{section.computability}, we show that the group of Turing machines is recursively presented and has a decidable word problem, but that its torsion problem (the problem of deciding if a given element has finite order) is undecidable in all dimensions. In fact, we show that the finitely-generated group of elementary Turing machines has an undecidable torsion problem in the purely group-theoretic sense. The proof is based on simulating classical Turing machines with elementary ones on a fixed alphabet and state set, with a simulation that preserves finite orbits. The result follows because the periodicity of reversible Turing machines is undecidable \cite{KaOl08}.

\begin{theorem}\label{thm:decidability_torsion_elementaryTM}
	The torsion problem of the group of elementary Turing machines on $n$ symbols, $k$ states and dimension $d$ is undecidable for every $n \geq 2$, $k \geq 1$ and $d \geq 1$.
\end{theorem}

As an application of the former result, we obtain that there is a finitely-generated group of reversible cellular automata on a full shift (thus on any uncountable sofic shift), which has undecidable torsion problem. This follows from an ``almost-embedding'' of the group of Turing machines into the group of cellular automata -- there is no actual embedding for group-theoretic reasons, but we show that the almost-embedding we construct preserves the decidability of the torsion problem.

\begin{corollary}\label{cor:autgroup_undecidabletorsion}
	Let $X$ be an uncountable sofic $\ZZ$-subshift. Then there is a finitely generated subgroup of $\Aut(X)$ which has undecidable torsion problem. 
\end{corollary}

For finite-state machines, we show that the torsion problem is decidable in dimension one, but is undecidable in higher dimensions; again we construct a finitely-generated subgroup where the problem is undecidable. The proof in the one-dimensional case is based on a simple pigeonhole argument, while in the higher-dimensional case undecidability is a corollary of the undecidability of the snake tiling problem~\cite{Ka03}.

\begin{theorem}\label{thm:decidability_torsion_FSA}
	Consider the group $G$ of finite-state machines on $n\geq 2$ symbols, $k \geq 1$ states and dimension $d \geq 1$,
	\begin{enumerate}
		\item The torsion problem of $G$ is decidable if $d = 1$.
		\item $G$ contains a finitely generated subgroup with undecidable torsion problem when $d \geq 2$.
	\end{enumerate}
\end{theorem}

Furthermore, we show that the decidability result holds for finite-state machines running on an arbitrary sofic subshift and that in fact their finiteness problem is decidable meaning we can decide whether a given finitely-generated subgroup is finite or not (which implies the decidability of their torsion problem). As a special case of our results, we obtain the following statement about topological full groups.

\begin{corollary}\label{cor:fullgroup_undecidabletorsion}
	Let $d \geq 2$. The topological full group of a full $\ZZ^d$-shift on at least two symbols contains a finitely generated subgroup with undecidable torsion problem.
\end{corollary}

We note that our group is very similar to the Thompson-Brin group 2V \cite{BeMa14,BeBl14}, the main difference being that elements of 2V can erase and add symbols to the tape; indeed $\mathrm{RTM}(2,1)$ can be seen naturally as a subgroup of 2V, consisting of elements that do not use this addional functionality. Our group is not isomorphic to any of the Thompson-Brin groups nV, as they are finitely generated. Thompson's V has a decidable torsion problem, but that of 2V is not, also due to the undecidability of periodicity of reversible Turing machines.

%

\subsection{Preliminaries}

In this section we present general definitions and fix the notation which is used throughout the article. The review of these concepts will be brief and focused on the dynamical aspects. For a more complete introduction the reader may refer to~\cite{LiMa95}. We restrict our setting to finitely generated and torsion-free abelian groups $\ZZ^d$, although, just like with cellular automata~\cite{CeCo10}, the idea of a Turing machine directly generalizes to general groups. All of the results of this and the following section directly generalize for arbitrary countable groups $G$ in place of $\ZZ^d$.

Let $\ag$ be a finite alphabet. The set $\ag^{\ZZ^d} = \{ x \colon \ZZ^d \to \ag\}$ equipped with the shift action $\sigma \colon \ZZ^d \times \ag^{\ZZ^d} \to \ag^{\ZZ^d}$ defined by $(\sigma^{\vec{v}}(x))_{\vec{u}} = x_{\vec{u}-\vec{v}}$ is a \define{full shift}. Thus $\sigma^{\vec v}$ shifts cell contents in direction $\vec v$, or equivalently moves the origin of $x$ to $-\vec v$.
The elements $a \in \ag$ and $x \in \ag^{\ZZ^d}$ are called \define{symbols} and \define{configurations} respectively.
Configuration $x \in \ag^{\ZZ^d}$ is \define{periodic} if $\sigma^{\vec{v}}(x)=x$ for some non-zero $\vec{v}\in\ZZ^d$, and it is \define{eventually periodic} if there exists a periodic configuration
$y \in \ag^{\ZZ^d}$ that differs from $x$ only in finitely many positions.
With the discrete topology on $\ag$ the set of configurations $\ag^{\ZZ^d}$ is a compact metrizable space, a generating metric is given by $d(x,y) = 2^{-\inf\{|\vec{v}|\  : \ \vec{v} \in \ZZ^d,\ x_{\vec{v}} \neq y_{\vec{v}}\}}$ where $|\vec{v}|$ is the taxicab norm $|\vec{v}| = \sum_{i = 1}^d |\vec{v}_i|$.

This topology has the sets $[a]_{\vec{v}} = \{x \in \ag^{\ZZ^d} : x_{\vec{v}} = a\in \ag\}$ as a subbasis. A \define{support} is a finite subset $F \subset \ZZ^d$. Given a support $F$, a \define{pattern with support $F$} is an element $p$ of $\ag^F$. The \define{cylinder} generated by $p$ in position $\vec{v}$ is $[p]_{\vec{v}} = \bigcap_{\vec{u} \in F} [p_{\vec{u}}]_{\vec{v} + \vec{u}}$. For simplicity, we write $[p]= [p]_{\vec{0}}$.

\begin{definition}
	A subset $X$ of $\ag^{\ZZ^d}$ is a \define{subshift} if it is topologically closed and $\sigma$-invariant, that is, for every $\vec{v} \in \ZZ^d$ we have $\sigma^{\vec{v}}(X)\subset X$. 
\end{definition}

Equivalently, $X$ is a subshift if and only if there exists a set of patterns $\FF$ that defines it.
\[X = \bigcap_{p \in \FF, \vec{v} \in \ZZ^d}{\ag^{\ZZ^d} \setminus [p]_{\vec{v}}}.\]
Any such $\FF$ which defines $X$ is called a set of \define{forbidden patterns} for $X$.

For a subshift $X \subset \ag^{\ZZd}$ and a finite support $F \subset \ZZd$ we define the \define{language $L_F(X)$ of support} $F$ of $X$ as the set of patterns $p \in \ag^F$ such that $[p] \cap X \neq \varnothing$. The language of $X$ is the union $L(X)$ of $L_F(X)$ over all finite $F \subset \ZZd$. We denote $p\sqsubset X$ iff $p\in L(X)$. For an individual configuration $x\in \ag^{\ZZ^d}$ we denote
$p\sqsubset x$ iff $x\in [p]_{\vec{v}}$ for some $\vec{v}\in \ZZd$, and we say that pattern $p$ \define{occurs} in $x$.

Let $X,Y$ be subshifts over alphabets $\ag$ and $\bg$ respectively. A continuous $\ZZ^d$-equivariant (i.e.\ shift-commuting) map $\phi \colon  X \to Y$ between subshifts is called a morphism. A well-known theorem of Curtis, Lyndon and Hedlund which can be found in full generality in~\cite{CeCo10} asserts that morphisms are equivalent to maps defined by local rules as follows: There exists a finite $F \subset \ZZ^d$ and $\Phi\colon \ag^{F} \to \bg$ such that for every $x \in X$, $\phi(x)_{\vec{v}} = \Phi(\sigma^{-\vec{v}}(x)|_F)$. If $\phi$ is an endomorphism (that is, $X = Y$) then we refer to it as a cellular automaton. A cellular automaton is said to be reversible if there exists a cellular automaton $\phi^{-1}$ such that $\phi \circ \phi^{-1} = \phi^{-1} \circ \phi = \ID$. It is well known that reversibility is equivalent to bijectivity, see Section 1.10 of~\cite{CeCo10}.

Throughout this article we use the following notation inspired by Turing machines. We denote by $\Sigma = \{0,\dots,n-1\}$ the set of tape symbols and $Q = \{1,\dots,k\}$ the set of states. We also use the symbols $n = |\Sigma|$ for the size of the alphabet and $k = |Q|$ for the number of states. Given a function of the form $f \colon \Omega \to A_1 \times \dotsc \times A_m$ we denote by $f_i$ the projection of $f$ to the $i$-th coordinate.

\section{Two models for Turing machine groups}\label{section.preliminaries}

In this section we define our generalized Turing machine model, and the group of Turing machines. In fact, we give two definitions for this group, one with a moving head and one with a moving tape as in \cite{Ku97}. We show that -- except in the case of a trivial alphabet -- these groups are isomorphic.\footnote{Note that the \define{dynamics} obtained from these two definitions are in fact quite different, as shown in \cite{Ku97,Ku10}.} Furthermore, both can be defined both by local rules and ``dynamically'', that is, in terms of continuity and the shift action. In the moving tape model we characterize reversibility as preservation of a suitably defined measure.

\subsection{The moving head model}

In the moving head model, we will represent our space as  $\Sigma^{\ZZd} \times Q \times \ZZd$. That is, the product of the set of configurations $\Sigma^{\ZZd}$, a set of states $Q$ and the possible positions of a head $\ZZd$. The objects of this space are therefore $3$-tuples $(x,q,\vec{v})$. In order to write this in a shorter manner, we use the notation $x_q^{\vec{v}}$ instead of $(x,q,\vec{v})$.

Given a function
\[ f \colon \Sigma^{F_{\text{in}}} \times Q \to \Sigma^{F_{\text{out}}} \times Q \times \ZZ^d, \]
where $F_{\text{in}},F_{\text{out}}$ are finite subsets of $\ZZ^d$, we can define a map $T_f\colon \Sigma^{\ZZd} \times Q \times \ZZd \to \Sigma^{\ZZd} \times Q \times \ZZd$ as follows: given $x_q^{\vec{v}} \in \Sigma^{\ZZd} \times Q \times \ZZd$ let $p = \sigma^{-{\vec v}}(x)|_{F_{\text{in}}}$ and $f(p,q) = (\tilde{p},\tilde{q},\tilde{\vec u})$. Then we define $T_f(x_q^{\vec{v}}) := \tilde{x}_{\tilde{q}}^{\vec{v}+\tilde{\vec{u}}}$ where:

$$\tilde{x}_{\vec u} = \begin{cases}
x_{\vec{u}} & \text{ if } {\vec u} - {\vec v}  \notin F_{\text{out}} \\
\tilde{p}_{{\vec u} - {\vec v}}&\text{ if }  {\vec u}-{\vec v} \in F_{\text{out}}
\end{cases}$$

\begin{definition}
	A function $T$ for which there is an $f \colon \Sigma^{F_{\text{in}}} \times Q \to \Sigma^{F_{\text{out}}} \times Q \times \ZZ^d$ such that $T = T_f$ is called a \define{moving head $(\ZZd,n,k)$-Turing machine}, and $f$ is its \define{local rule}. If there exists a $(\ZZd,n,k)$-Turing machine $T^{-1}$ such that $T \circ T^{-1} = T^{-1} \circ T = \ID$, we say $T$ is \define{reversible}.
\end{definition}

This definition corresponds to classical Turing machines with the moving head model when $d = 1$, $F_{\text{in}} = F_{\text{out}} = \{\vec{0}\}$ and $f_3(x,q) \in \{-1,0,1\}$ for all $x, q$.
An illustration of how a moving head Turing machine acts can be seen in Figure~\ref{z_2_machine2}. Note that $\sigma^{{- \vec v}}(x)|_{F}$ is the $F$-shaped pattern ``at $\vec v$''. We do not write $x|_{{\vec v} + F}$ because we want the pattern we read from $x$ to have $F$ as its domain.

Note that one of these machines could be defined by several different local functions $f$, and that the third component of $f$ has finite range. Also, given $f \colon \Sigma^{F_{\text{in}}} \times Q \to \Sigma^{F_{\text{out}}} \times Q \times \ZZ^d$ we can define $F := F_{\text{in}} \cup F_{\text{out}} \cup f_3(\Sigma^{F_{\text{in}}} \times Q)$ and $f'\colon \Sigma^F \times Q \to \Sigma^F \times Q \times F$ such that $T_f = T_{f'}$. This motivates the following definition: The \define{minimal neighborhood} of $T$ is the minimum set for inclusion $F$ for which there is $f\colon \Sigma^F \times Q \to \Sigma^F \times Q \times F$ such that $T = T_f$. This minimum always exists as the set of finite subsets of $\ZZ^d$ which work for this definition is closed under intersections.

As $\ZZ^d$ is finitely generated, we can also use a numerical definition of radius in place of the neighborhood: Let $B(\vec{v},r)$ be the set of $\vec{u}\in \ZZ^d$ such that $|\vec{u}- \vec{v}| \leq r$. By possibly changing the local rule $f$, we can always choose $F_{\text{in}} = B(\vec{0},r_i)$ and $F_{\text{out}} = B(\vec{0},r_o)$ for some $r_i, r_o \in \NN$, without changing the Turing machine $T_f$ it defines. The minimal such $r_i$ is called the \define{in-radius} of $T$, and the minimal $r_o$ is called the \define{out-radius} of $T$. We say the in-radius of a Turing machine is $-1$ if there is no dependence on input, that is, the neighborhood $B(0^d,r_i)$ can be replaced by the empty set. The maximum value of $|\vec v|$ for all $\vec v \in f_3(\Sigma^{F} \times Q)$ is called the \define{move-radius} of $T$. Finally, the maximum of all these three radii is called the \define{radius} of $T$. In this terminology, classical Turing machines are those with in- and out-radius $0$, and move-radius $1$.

\begin{figure}[h!]
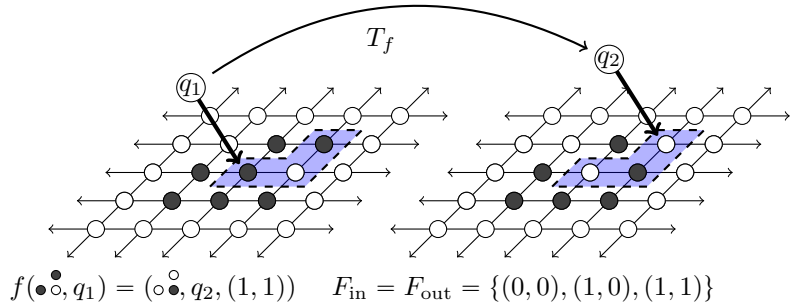

	\centering
	\include{z2_machine2}
	\caption{The action of a moving head machine $T_f$.}
	\label{z_2_machine2}
\end{figure}

\begin{definition}
	We denote by $\TM(\ZZd,n,k)$ the set of $(\ZZd,n,k)$-Turing machines and $\RTM(\ZZd,n,k)$ the set of reversible $(\ZZd,n,k)$-Turing machines.
\end{definition}

In some parts of this article we just consider $d = 1$. In this case we simplify the notation and just write $\RTM(n,k) := \RTM(\ZZ,n,k)$
and $\TM(n,k) := \TM(\ZZ,n,k)$. Of course, we want $\TM(\ZZd,n,k)$ to be a monoid and $\RTM(\ZZd,n,k)$ a group under function composition. This is indeed the case, and one can prove this directly by constructing local rules for the inverse of a reversible Turing machine and composition of two Turing machines. However, it is easier to extract this from the following characterization of Turing machines as a particular kind of cellular automaton.

Let $X_k$ be the subshift with alphabet $Q \cup \{0\}$ such that in each configuration the number of non-zero symbols is at most one. \[ X_k = \{ x \in \{0,1,\ldots,k\}^{\ZZ^d} :  0 \notin \{x_{\vec{u}}, x_{\vec{v}}\} \implies \vec{u} = \vec{v} \}. \] In the case where $k =1$ this subshift is often called the \define{sunny-side up subshift}. Notice that $X_0 = \{0^{\ZZ^d}\}$ consists of a single configuration and for non-negative integers $i < j$ we have $X_i \subsetneq X_j$. Let also $X_{n,k} = \Sigma^{\ZZ^d} \times X_{k}$, where recall that we always set $\Sigma = \{0, \ldots, n-1\}$. For the case $d = 1$, configurations in $X_{n,k}$ represent a bi-infinite tape filled with symbols in $\Sigma$ possibly containing a head that has a state in $Q$. Note that there might be no head in a configuration.

More precisely, we interpret $X_{n,k}$ as a compactification of $\Sigma^{\ZZd} \times Q \times \ZZd$ by identifying $x^{\vec v}_q = (x,q,\vec v) \in \Sigma^{\ZZ^d} \times Q \times \ZZ^d$ with the point $(x,y)$ where $y_{\vec v} = q$ and $y_{\vec u} = 0$ for $\vec u \neq \vec v$. We can now interpret Turing machines as functions on $X_{n,k}$ in the following way: For $(x,y) \in X_{n,k}$, if there is no ${\vec v} \in \ZZ^d$ such that $y_{\vec v} \neq 0$ then $T(x,y) := (x,y)$. Otherwise apply $T$ through the natural bijection.


For a subshift $X$, we denote by $\End(X)$ the monoid of endomorphisms of $X$ and $\Aut(X)$ the group of automorphisms of $X$. From the previous argument it follows that Turing machines are precisely cellular automata on $X_{n,k}$ with the additional property that ``the configurations are only modified near a head''. This is the content of the following proposition.

\begin{restatable}{proposition}{propCACharacterization}
	\label{prop_CACharacterization}
	Let $n,k$ be positive integers and $Y = X_{n,0}$. Then:
	\begin{align*}
	\TM(\ZZ^d,n,k) &= \{ \phi \in \End(X_{n,k}) : \phi|_Y = \ID, \phi^{-1}(Y) = Y \} \\
	\RTM(\ZZ^d,n,k) &= \{ \phi \in \Aut(X_{n,k}) : \phi|_Y = \ID \}.
	\end{align*}
\end{restatable}

\begin{proof}
	Consider a Turing machine $T \in \TM(\ZZ^d,n,k)$ seen as a function on $X_{n,k}$. A direct calculation shows that $T$ is shift-commuting and continuous, therefore $T \in \End(X_{n,k})$. Also, $T$ acts trivially on $X_{n,0}$ so $T|_Y = \ID$ and if a configuration has a head, it can only be shifted but not disappear, thus $T^{-1}(Y) = Y$. Moreover, if $T \in \RTM(\ZZ^d,n,k)$, then $T$ has a Turing machine inverse, thus a cellular automaton inverse, and it follows that $T \in \Aut(X_{n,k})$.
	
	Conversely, let $\phi \in \End(X_{n,k})$, so that $\phi(x,y)_{\vec v} = \Phi(\sigma^{-\vec v}(x,y)|_F)$ for some local rule $\Phi\colon (\Sigma \times \{0,\dots,k\})^{F} \to \Sigma \times \{0,\dots,k\}$ and $F$ a finite subset of $\ZZd$, where we may suppose $\vec 0 \in F$.
	
	As $\phi|_Y = \ID$, we can deduce that $\Phi(u,v) = (u,v)_{\vec{0}}$ if $v = 0^F$. Therefore if $(x,y) \in X_{n,k}$, $y_{\vec v} \neq 0$ and we define $W_{\vec v} = \{\vec u : \vec v \in \vec u + F\} = \vec v - F$ we get that $\phi(x,y)|_{\ZZd \setminus W_{\vec v}} = (x,y)|_{\ZZd \setminus W_{\vec v}}$. We extend $\Phi$ to $\widetilde{\Phi} \colon (\Sigma \times \{0,\dots,k\})^{W_{\vec{0}} + F} \to  (\Sigma \times \{0,\dots,k\})^{W_{\vec{0}}}$ by pointwise application of $\Phi$. Note that $W_{\vec{0}} = -F$.
	
	We can then define $f_{\phi} \colon \Sigma^{F - F} \times Q \to \Sigma^{-F} \times Q \times \ZZ^d$ by using $\widetilde{\Phi}$ as follows:  We set $f_{\phi}(p,q) = (p',q',\vec u)$ if, after defining $r \in \{0,\dots,k\}^{F - F}$ such that $r_{\vec{0}} = q$ and $0$ elsewhere, we have $\widetilde{\Phi}(p,r) = (p',r')$ and $r' \in \{0,\dots,k\}^{-F}$ contains the symbol $q' \neq 0$ in position $\vec u$ (there is always a unique such position $\vec u$ as $\phi^{-1}(Y) = Y$). It can be verified that the Turing machine $T_{f_{\phi}}$ is precisely $\phi$, therefore $\phi \in \TM(G,n,k)$.
	
	If $\phi \in \Aut(X_{n,k})$ then $\phi^{-1}(Y) \supset Y$ is implied by $\phi|_Y = \ID$, and since the inverse automorphism $\phi^{-1}$ satisfies $\phi^{-1}|_Y = \ID$ as well, we also have $\phi^{-1}(Y) \subset Y$. Thus, $\phi$ is a Turing machine. Similarly, the inverse map $\phi^{-1}$ is a Turing machine. Thus, in this case $\phi \in \RTM(\ZZ^d,n,k)$.
\end{proof}

Alternatively, in the previous statement we can just write $\TM(\ZZ^d,n,k) = \{ \phi \in \End(X_{n,k}) \mid \phi|_Y = \ID, \phi^{-1}(Y) \subset Y \}$, since $\phi^{-1}(Y) \supset Y$ is
implied by $\phi|_Y = \ID$. 

\begin{corollary}
	Let $\phi \in \TM(\ZZ^d,n,k)$. We have that $\phi \in \RTM(\ZZ^d,n,k)$ if and only if $\phi$ is bijective.
\end{corollary}

Readers familiar with the theory of cellular automata may wonder if injectivity is enough, since injective cellular automata on full shifts are surjective. This is not a priori clear since cellular automata on nontransitive sofic shifts (such as $X_{n,k}$) can be injective without being surjective. We will, however, later prove the stronger result that both injectivity and surjectivity are equivalent to bijectivity.

Clearly, the conditions of Proposition~\ref{prop_CACharacterization} are preserved under function composition and inversion. Thus:

\begin{corollary}
	Under function composition, $(\TM(\ZZ^d,n,k), \circ)$ is a monoid and $(\RTM(\ZZ^d,n,k), \circ)$ is a group.
\end{corollary}

We usually omit the function composition symbol, and use the notations $\TM(\ZZ^d,n,k)$ and $\RTM(\ZZ^d,n,k)$ to refer to the corresponding monoids and groups.

An important corollary of Proposition~\ref{prop_CACharacterization} is that every result we prove about Turing machine groups says something about cellular automata. In particular, if a group $H$ embeds into $\RTM(\ZZ,n,k)$, then $H$ also embeds into $\Aut(X_{n,k})$, which implies that there exists a one-dimensional sofic subshift whose automorphism group contains $H$.\footnote{In Section~\ref{sec:CATorsion}, we show that $\RTM(\ZZ,n,k)$ also ``almost'' embeds into $\Aut(\Sigma^\ZZ)$.}

Before defining the second model for Turing machines, we introduce an extended model which will be occasionally used in what follows. Given a subshift $X \subset \Sigma^{\ZZ^d}$ we denote the set of moving-head $(\ZZ^d,n,k)$-Turing machines $f$ which satisfy $f(X \times X_k) \subset X \times X_k$ by $\TM(X,k)$ (where we omit the group from the notation, since it is determined by $X$). And the set of reversible ones by $\RTM(X,k)$. This change basically amounts to replacing the full shift $\Sigma^{\ZZ^d}$ in the definition by an arbitrary subshift $X$. This will be used to make explicit which properties from $X$ are required for our results, though our focus is on the case where $X = \Sigma^{\ZZ^d}$ in which $\RTM(\Sigma^{\ZZ^d},k) = \RTM(\ZZ^d,n,k)$ for $|\Sigma| = n$. Proposition~\ref{prop_CACharacterization} is still valid in this extended context.

One use of this generalization is that it allows us to study Turing machines on a more robust class of subshifts than just full shifts. For example, when studying Turing machines on transitive SFTs rather than full shifts only, states can be eliminated, due to the following lemma.

\begin{definition}\label{def:sparseshift}
	Let $X \subset \Sigma^\ZZ$ be a subshift and $\# \notin \Sigma$. Write $\sqrt[n]{X}$ for the shift space with points
	\begin{align*}
	\{ x \in ( \Sigma \cup \{\#\})^\ZZ : &\mbox{ there are } k \in \{0,\dots,n-1\} \mbox{ and } y \in X \mbox{ such that for every } m \in \ZZ, \\
	& (x_{mn+k} = y_m \mbox{ and for every } j \in \{0,\dots,n-1\} \setminus \{k\} \mbox{ we have } x_{mn+j}= \#) \}.
	\end{align*}
\end{definition}

The subshift $\sqrt[n]{X}$ is the space of configurations on $\ZZ$ where some configuration of $X$ occurs on a coset of $n\ZZ$ and the rest of the positions are filled with a special symbol $\#$.

\begin{lemma}
	\label{lem:OneState} Let $X \subset A^\ZZ$ be a subshift. Then $\TM(X,k)$ is isomorphic to $\TM(\sqrt[k]{X},1)$ as a monoid and $\RTM(X,k)$ is isomorphic to  $\RTM(\sqrt[k]{X},1)$ as a group.
\end{lemma}

\begin{proof}
	As the monoid of Turing machines can be seen as a submonoid of endomorphisms of a subshift, it follows that they commute with the shift. More precisely, for each machine $T$ and $\vec{m} \in \ZZ$ if we interpret $T$ acting on $X \times X_k$ by endomorphisms, then we have that $\sigma^{\vec{m}} \circ T \circ \sigma^{-\vec{m}} = T$. In particular, as non-$\#$ symbols appear at a bounded distance in each configuration of $\TM(\sqrt[k]{X},1)$, it suffices to define a Turing machine in $(\sqrt[k]{X} \setminus [\#]_0) \times \{1\} \times \ZZ$ to completely determine its action over $\sqrt[k]{X}\times \{1\}\times \ZZ$.
	
	Let $\psi \colon X \to \sqrt[k]{X}$ be defined by
	\[ \psi(x)_i = \left\{ \begin{array}{ll}
	\#, & \mbox{if } i \not\equiv 0 \bmod k, \mbox{ and} \\
	x_{i/k}, & \mbox{otherwise.} \end{array}\right.
	\]
	Note that $\psi$ defines a $1$-to-$1$ map onto $\sqrt[k]{X} \setminus [\#]_0$.
	
	Also, define $\eta \colon \ZZ \times Q \to \ZZ$ by $\eta(\vec{v},q) = k \vec{v}+q-1$ which is clearly bijective. We can thus finally define a bijection $\delta \colon X \times Q \times \ZZ \to (\sqrt[k]{X} \setminus [\#]_0) \times \{1\} \times \ZZ$ by $\delta(x_q^{\vec{v}}) = \psi(x)^{\eta(\vec{v},q)}_1$.
	
	Now, given a machine $T \in \TM(\sqrt[k]{X},1)$ we define $\varphi(T) \in \TM(X,k)$ by $\varphi(T) = \delta^{-1} \circ T \circ \delta$. Note that this is well-defined since $T \in \TM(\sqrt[k]{X},1)$ implies that $T(\sqrt[k]{X} \setminus [\#]_0 \times X_1) = \sqrt[k]{X} \setminus [\#]_0 \times X_1$ again because non-$\#$-symbols appear with bounded gaps, and since $T$ is the identity map on points where the head does not appear. By definition it is then clear that $\varphi(T \circ T') = \varphi(T) \circ \varphi(T')$ and that $\varphi$ is $1$-to-$1$. We only need to show that $\varphi(T) \in \TM(X,k)$ and that $\varphi$ is onto.
	
	Firstly, it is clear that $\delta^{-1} \circ T \circ \delta$ maps $X \times Q \times \ZZ$ to itself, therefore the head cannot disappear. We have that $\delta \circ \sigma^{\vec{m}} = \sigma^{\vec{km}} \circ \delta$ and thus $\varphi(T)$ is shift commuting.
	\begin{align*}
	\varphi(T) \circ \sigma^{\vec{m}} &= \delta^{-1} \circ T \circ \delta \circ \sigma^{\vec{m}}\\
	&= \delta^{-1} \circ T \circ \sigma^{\vec{km}} \circ\delta\\
	&= \delta^{-1} \circ \sigma^{\vec{km}} \circ T \circ \delta\\
	&= \sigma^{\vec{m}} \circ\delta^{-1} \circ T \circ \delta\\
	&= \sigma^{\vec{m}} \circ \varphi(T).
	\end{align*}
	Since $\delta$ and $T$ are continuous, and $\delta^{-1}$ is continuous on the image of $\delta$, we have that $\varphi(T)$ is continuous and shift invariant and therefore defines an endomorphism of $X \times X_k$ which is an element of $\TM(X,k)$. Conversely, an analogous argument shows that for each $T \in \TM(X,k)$ then the map $T'$ defined as $T' = \delta \circ T \circ \delta^{-1}$ on $\sqrt[k]{X} \setminus [\#]_0 \times \{1\} \times \ZZ$ (and by conjugating with a suitable power of $\sigma$ on other points) is in $\TM(\sqrt[k]{X},1)$ and thus $\varphi(T') = T$, showing that $\varphi$ is onto.	
	
	Finally, if $T \in \RTM(\sqrt[k]{X},1)$ then $\varphi(T)\circ \varphi(T^{-1}) = \ID$ and thus $\varphi(T) \in \RTM(X,k)$. \end{proof}

\subsection{The moving tape model}

Even though the moving head model is helpful when building examples, it has a fundamental disadvantage: the space on which the machines act is not compact. From an intuitive point of view, it means that a sequence of Turing machines could potentially move the head to infinity and make it disappear. Or alternatively, in the point of view of seeing Turing machines as endomorphisms as in Proposition~\ref{prop_CACharacterization} (which usually, e.g.\ in K\r{u}rka \cite{Ku97}, is directly what the moving head model refers to), the space is compact, but there are uncountably many points that do not quite represent Turing machine configurations.

It's also possible to consider the position of the Turing machine as fixed at $\vec{0}$, and move the tape instead, to obtain the moving tape Turing machine model. In \cite{Ku97}, where Turing machines are studied as dynamical systems, the moving head model and moving tape model give non-conjugate dynamical systems. However, the abstract monoids defined by the two points of view turn out to be equal, and we obtain an equivalent definition of the group of Turing machines. 


As in the previous section, we begin with a definition using local rules.

Given a function
$f \colon \Sigma^{F_{\text{in}}} \times Q \to \Sigma^{F_{\text{out}}} \times Q \times \ZZ^d$,
where $F_{\text{in}},F_{\text{out}}$ are finite subsets of $\ZZ^d$, we can define a map $T_{f} \colon \Sigma^{\ZZ^d} \times Q \to \Sigma^{\ZZ^d} \times Q$ as follows: If $f(x|_{F_{\text{in}}}, q) = (p, q', \vec v)$, then $T_f(x, q) = (\sigma^{-\vec v}(y), q')$
where
\[ y_{\vec u} = \left\{\begin{array}{cc}
x_{\vec u}, &\mbox{if } \vec u \notin F_{\text{out}} \\
p_{\vec u}, &\mbox{if } \vec u \in F_{\text{out}},
\end{array}\right. \]

\begin{definition}
	Any function $T \colon \Sigma^{\ZZd} \times Q \to \Sigma^{\ZZd} \times Q$ for which there is an $f$ as above such that $T = T_f$ is called a \define{moving tape $(\ZZ^d,n,k)$-Turing machine} and $f$ is its \define{local rule}. If there exists a $(\ZZ^d,n,k)$-Turing machine $T^{-1}$ such that $T \circ T^{-1} = T^{-1}\circ T = \ID$ we say $T$ is reversible.
\end{definition}

\begin{figure}[h!]
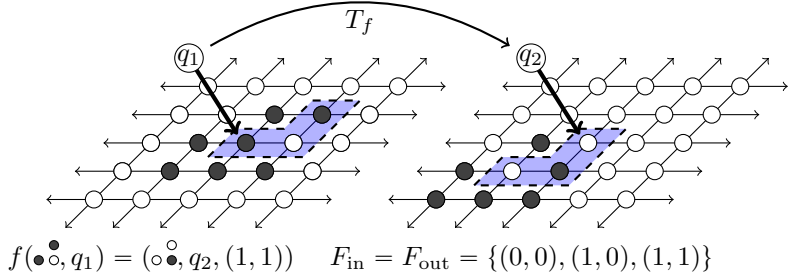

	\centering
	\include{z2_machine}
	\caption{The action of a moving tape machine $T_f$.}
	\label{z_2_machine}
\end{figure}

These machines also have the following characterization with a slightly more dynamical feel to it. Say that two configurations $x$ and $y$ in $\Sigma^{\ZZd}$ are \define{asymptotic}, and write $x \sim y$, if there exists a finite $F \subset \ZZ^d$ such that $x|_{\ZZ^d \setminus F} = y|_{\ZZ^d \setminus F}$. In order to be more specific, we write $x \sim_F y$ to claim that a particular choice of $F$ satisfies the property.

\begin{restatable}{lemma}{MovingTapeDynamicalDef}
	\label{lem:MovingTapeDynamicalDef}
	Let $T \colon\Sigma^{\ZZd} \times Q \to \Sigma^{\ZZ^d} \times Q$ be a function. Then $T$ is a moving tape Turing machine if and only if it is continuous, and for a continuous function $s \colon \Sigma^{\ZZ^d} \times Q \to \ZZ^d$ and a finite $F \subset \ZZ^d$ we have $T_1(x, q) \sim_F \sigma^{s(x, q)}(x)$ for all $(x, q) \in \Sigma^{\ZZ^d} \times Q$.
\end{restatable}

\begin{proof}
	It is easy to see that $T_f$ for any local rule $f \colon \Sigma^{F_{\text{in}}} \times Q \to \Sigma^{F_{\text{out}}} \times Q \times \ZZ^d$ is continuous. The projection to the third component of $f$ gives the function $-s$, and one can take $F$ as the minimal neighborhood of $T_f$.
	
	For the converse, since $s$ is a continuous function from a compact space to a discrete one we conclude that the image of $s$ is bounded. Furthermore it only depends on a finite set $F_0$ of coordinates of $x$. Since $T$ is continuous, $T(x,q)|_{F + s(\Sigma^{\ZZ^d} \times Q)}$ depends only on a finite set of coordinates $F_1$ of $x$. It is then easy to extract a local rule
	\[ f \colon \Sigma^{F_0 \cup F_1} \times Q \to \Sigma^{F - s(\Sigma^{\ZZ^d} \times Q)} \times Q \times {\ZZ^d}, \]
	for $T$. 
\end{proof}

We call the function $s$ in the definition of these machines the \define{shift indicator} of $T$, as it indicates how much the tape is shifted depending on the local configuration around $\vec{0}$. In the theory of orbit equivalence and topological full groups, the analogs of $s$ are usually called \define{cocycles}. Note that in the definition using local functions, the third coordinate of the image indicates how much the head moves, while the shift indicator shows how the configuration shifts, hence the minus sign next to $s$ in the above proof.

\begin{remark}
	In the previous lemma it is not enough that $T_1(x, q) \sim \sigma^{s(x, q)}(x)$ for all $(x, q) \in \Sigma^{\ZZ^d} \times Q$, we need the configurations to be uniformly asymptotic to each other (with a fixed $F\subset \ZZd$). Indeed, let $Q = \{1\}$ and consider the function $T \colon \Sigma^{\ZZ} \times Q \to \Sigma^{\ZZ} \times Q$ defined by $(T_1(x, 1))_i = x_{-i}$ if $x_{[-|i|+1,|i|-1]} = 0^{2i-1}$ and $\{x_i,x_{-i}\} \neq \{0\}$, and $(T_1(x, 1))_i = x_i$ otherwise. Clearly this map is continuous, the constant map $s(x,q) = \vec 0$ gives a shift-indicator for it and $T_1(x, q) \sim x$ for every $x \in \Sigma^{\ZZd}$. However, $T$ is not defined by any local rule since it can modify the tape arbitrarily far from the origin.
\end{remark}

As for moving head machines, it is easy to see (either by constructing local rules or by applying the dynamical definition) that the composition of two moving tape Turing machines is again a moving tape Turing machine. This allows us to proceed as before and define their monoid and group.

\begin{definition}
	We denote by $\TM_{\fix}(\ZZ^d,n,k)$ and $\RTM_{\fix}(\ZZ^d,n,k)$ the monoid of moving tape $(\ZZ^d,n,k)$-Turing machines and the group of reversible moving tape $(\ZZ^d,n,k)$-Turing machines respectively.
\end{definition}

Now, let us show that both models are equivalent in group-theoretical terms. First, we define the natural monoid epimorphism $\Psi\colon \TM(\ZZ^d,n,k) \to \TM_{\fix}(\ZZ^d,n,k)$ that shifts the configurations so that the
head remains at the origin:

\begin{definition}
	\label{def:morphism}
	Let $T \in \TM(\ZZ^d,n,k)$. We define $\Psi(T)\in\TM_{\fix}(\ZZ^d,n,k)$ as the moving tape Turing machine such that \[\Psi(T)(x,q) = (\sigma^{-\vec{v}}(y),r) \mbox{ when } T(x_{q}^{\vec{0}})= y_{r}^{\vec{v}}.   \]
\end{definition}

In other terms, define the shift equivalence relation $\equiv$ on the moving head space
$\Sigma^{\ZZd} \times Q \times \ZZd$ by $x_{q}^{\vec{u}}\equiv \sigma^{\vec{v}}(x)_{q}^{\vec{u}+\vec{v}}$
for all $x\in\Sigma^{\ZZd}$, $q\in Q$ and $\vec{u}, \vec{v}\in \ZZd$. Then for $T \in \TM(\ZZ^d,n,k)$, $\Psi(T)$ is the action induced by $T$ on the quotient $\Sigma^{\ZZd} \times Q \times \ZZd/{\equiv}$. Note that 
for any two $\equiv$-classes $U$ and $V$, if $T(U)\subset V$ then $T|_U \colon U\longrightarrow V$ is a bijection. As we shall see, this shows that $\Psi$ preserves the injectivity and
the surjectivity of Turing machines.

\begin{restatable}{proposition}{InjectiveSurjectiveTapeVsHead}
	\label{injective_surjective_tape_vs_head}
	$T$ is injective (surjective) if and only if $\Psi(T)$ is injective (surjective, respectively).
\end{restatable}
\begin{proof}
	If $T$ is not injective then $T(c_1)=T(c_2)$ for some $c_1\neq c_2$. Because $T$ is injective in each $\equiv$-class, $c_1\not\equiv c_2$. Hence there are two
	distinct $\equiv$-classes with the same image so that $\Psi(T)$ is not injective. Conversely, if $\Psi(T)$ maps two distinct $\equiv$-classes into the same class,
	there are $c_1\not\equiv c_2$ such that $T(c_1)\equiv T(c_2)$. Because  $T$ is surjective in $\equiv$-classes, there is $c'_1\equiv c_1$ such that $T(c'_1)=T(c_2)$,
	so $T$ is not injective.
	
	If $T$ is surjective then clearly every $\equiv$-class has a pre-image so that $\Psi(T)$ is surjective. And if $\Psi(T)$ is surjective then every  $\equiv$-class
	has a pre-image, and because $T$ is surjective in $\equiv$-classes, every element of every $\equiv$-class has a pre-image, that is, $T$ is surjective.\end{proof}

The function $\Psi$ is clearly always monoid epimorphism. It is not injective in the trivial case $n = 1$: Indeed, we have that $\RTM_{\fix}(\ZZ^d,1,k)$ is isomorphic to the symmetric group on $k$ symbols, and $\TM_{\fix}(\ZZ^d,1,k)$ is isomorphic to the monoid of all functions from $\{1,\dots,k\}$ to itself. Therefore both of these groups are finite when $n=1$. On the other hand, clearly $\ZZ^d$ embeds into $\RTM(\ZZ^d,1,k)$ and $\TM(\ZZ^d,1,k)$ as the shifts are non-trivial elements of these groups. Next, we show that in most other cases $\Psi$ is injective.

Intuitively, in order to also have $\ZZ^d$ embed into the monoid of moving tape Turing machines, we need the configuration space to admit configurations with a certain degree of aperiodicity. We shall see that this is indeed the only obstruction and obtain as a corollary that for every $n \geq 2$ the map $\Psi$ is an isomorphism.

\begin{definition}
	\label{def:locallyaperiodic}
	A subshift $X \subset \Sigma^{\ZZ^d}$ is said to be \define{locally aperiodic} if for every $x\in X$, every finite $F \subset \ZZ^d$ and every non-zero $\vec{v}\in\ZZ^d$ there exists $y\in [x|_F]$ such that $\sigma^{\vec{v}}(y)\neq y$.
\end{definition}

For example, the sunny-side up subshift $X_1\subset \{0,1\}^\ZZ$ is
locally aperiodic because every non-empty cylinder contains a  configuration with exactly one occurrence of $1$, and such a configuration is non-periodic. But this means that the morphism $\Psi$ of Definition~\ref{def:morphism} is not necessarily injective even on reversible Turing machines over locally aperiodic subshifts. Consider, for example, the single state machine $T\in \TM(X_1,1)$ that maps, for some fixed $\vec{v}\in\ZZ$,
\[
x_{1}^{\vec{0}}\mapsto
\left\{
\begin{array}{ll}
x_{1}^{\vec{0}}, & \mbox{  if $x_0=0$},\\
\sigma^{\vec{v}}(x)_{1}^{\vec{v}}, & \mbox{ if $x_0=1$}.
\end{array}
\right.
\]
This machine shifts the configuration and the position by $\vec{v}$ if the cell under the head contains the unique $1$ of the configuration,
and does not do anything if the cell contains symbol $0$. The Turing machine is clearly reversible, and $\Psi(T)$ is the identity regardless of
the choice of vector $\vec{v}$.

To guarantee injectivity of $\Psi$ we define an even more restrictive variant of local aperiodicity:

\begin{definition}
	A subshift $X \subset \Sigma^{\ZZ^d}$ is \define{strongly locally aperiodic} if for every $x\in X$, every finite $F \subset \ZZ^d$ and every non-zero $\vec{v}\in\ZZ^d$ there exists $y\in [x|_F]$ such that $y_{\vec{u}}\neq y_{\vec{u}+\vec{v}}$  for some $\vec{u}$ such that
	$\vec{u},\vec{u}+\vec{v}\not\in F$.
\end{definition}

Obviously $X = \{0,\dots,n-1\}^{\ZZd}$ is strongly locally aperiodic whenever $n\geq 2$. With this definition we have the following.

\begin{restatable}{lemma}{GroupIsomorphismLemmabis}
	\label{group_isomorphism_lemmabis}
	Let $X \subset \Sigma^{\ZZ^d}$ be strongly locally aperiodic. We have that
	\begin{align*}
	\TM_{\fix}(X,k) &\cong \TM(X,k) \\
	\RTM_{\fix}(X,k) &\cong \RTM(X,k).
	\end{align*}
\end{restatable}

\begin{proof}
	Consider again the epimorphism $\Psi\colon \TM(\ZZ^d,n,k) \to \TM_{\fix}(\ZZ^d,n,k)$ and suppose there exist a pair $T \neq T'$ in $\TM(X,k)$ such that $\Psi(T) = \Psi(T')$. Let $x_q^{\vec{0}}$ be such that $T(x_q^{\vec{0}}) \neq T'(x_q^{\vec{0}})$. Denoting by $W = F \cup F'$ the union of the neighborhoods $F$ of $T$ and $F'$ of $T'$ we get that $T$ and $T'$ can be described by rules of the form $f_T(x|_{W},q) = (p,r,\vec{v})$ and $f_{T'}(x|_{W},q) = (p',r',\vec{v}')$. Denote by $x_{[p]}$ and $x_{[p']}$ the configuration where the symbols of $x$ in the support $W$ have been replaced by $p$ and $p'$ respectively. Clearly $\Psi(T) = \Psi(T')$ implies that $r = r'$ and $\sigma^{-\vec{v}}(x_{[p]}) = \sigma^{-{\vec{v}'}}(x_{[p']})$.
	
	If $\vec{v}=\vec{v}'$ then also $x_{[p]}=x_{[p']}$, which contradicts $T(x_q^{\vec{0}}) \neq T'(x_q^{\vec{0}})$. So we must have  $\vec{v}\neq \vec{v}'$. As $X$ is strongly locally aperiodic, there exists $y\in [x|_{W}]$ such that
	$y_{\vec{u}}\neq y_{\vec{u}+\vec{v'}-\vec{v}}$ for some position $\vec{u}$ that satisfies $\vec{u},\vec{u}+\vec{v'}-\vec{v}\not\in W$.
	As $y|_{W}=x|_{W}$, we have that $\Psi(T)(y,q)=(\sigma^{-\vec{v}}(y_{[p]}),r)$ and $\Psi(T')(y,q)=(\sigma^{-\vec{v}'}(y_{[p']}),r')$.
	Then $\Psi(T) = \Psi(T')$ implies that $y_{[p]}=\sigma^{\vec{v}-\vec{v}'}(y_{[p']})$, which is not true in position $\vec{u}$.\end{proof}

%

A non-trivial full shift $\Sigma^{\ZZ^d}$ is strongly locally aperiodic, and thus Lemma~\ref{group_isomorphism_lemmabis} gives the following corollary.

\begin{restatable}{corollary}{GroupIsomorphismLemma}
	\label{group_isomorphism_lemma}
	If $n \geq 2$ then:
	\begin{align*}
	\TM_{\fix}(\ZZ^d,n,k) &\cong \TM(\ZZ^d,n,k) \\
	\RTM_{\fix}(\ZZ^d,n,k) &\cong \RTM(\ZZ^d,n,k).
	\end{align*}
\end{restatable}

The previous result means that apart from the trivial case $n = 1$ where the tape plays no role, we can study the properties of these groups using any model. 



\subsection{The uniform measure and reversibility.}\label{subsection_measure}

Consider the space $\Sigma^{\ZZ^d} \times Q$. Let $\mu$ be the product of the uniform Bernoulli measure on $\Sigma^{\ZZ^d}$ and the uniform discrete measure on $Q$. That is, $\mu$ is the measure such that for every finite $F\subset \ZZ^d$ and $p \in \Sigma^F$, we have \[\mu([p]\times \{q\}) = \frac{1}{kn^{|F|}}.\]

\begin{theorem}
	\label{theorem_injective_surjective_reversible}
	Let $T \in \TM_{\fix}(\ZZ^d,n,k)$. Then the following are equivalent:
	\begin{enumerate}
		\item $T$ is injective.
		\item $T$ is surjective.
		\item $T \in \RTM_{\fix}(\ZZ^d,n,k)$.
		\item $T$ preserves the uniform measure ($\mu(T^{-1}(A)) = \mu(A)$ for all Borel sets $A$).
		\item $\mu(T(A)) = \mu(A)$ for all Borel sets $A$.
	\end{enumerate}
\end{theorem}

\begin{proof}
	Let $T$ be arbitrary, and let $F$ be its minimal neighborhood. Consider the cylinders $C_i = [p_i] \times \{q\}$ where $p_i \in \Sigma^F, q \in Q$. These cylinders form a clopen partition of $\Sigma^{\ZZ^d} \times Q$ into $kn^{|F|}$ cylinders of measure $\frac{1}{kn^{|F|}}$.
	
	Now, because $F$ is the minimal neighborhood of $T$, $T$ is a homeomorphism from $C_i$ onto $D_i = T(C_i)$, and $D_i$ is a cylinder set of the form $[p'] \times \{q'\}$ for some $p' \in \Sigma^{\vec{v} + F}$, $q' \in Q$, which must be of the same measure as $C_i$ as the domain $\vec{v} + F$ of $p'$ has as many coordinates as the domain $F$ of $p$. Note that $D_i$ is not necessarily a cylinder centered at the origin, and the offset $\vec v$ is given by the shift-indicator. Now, observe that injectivity is equivalent to the cylinders $D_i = T(C_i)$ being disjoint. Namely, they must be disjoint if $T$ is injective, and if they are disjoint then $T$ is injective because $T|_{C_i} : C_i \to D_i$ is a homeomorphism. Surjectivity on the other hand is equal to $\Sigma^{\ZZ^d} \times Q = \bigcup_i D_i$, since $\bigcup_i D_i = \bigcup_i T(C_i) = T(\Sigma^{\ZZ^d} \times Q)$.
	
	Now, it is easy to show that injectivity and surjectivity are equivalent: If $T$ is injective, then the $D_i$ are disjoint, and $\mu(\bigcup_i D_i) = \sum_i \mu(D_i) = 1$, so we must have $\bigcup_i D_i = \Sigma^{\ZZd} \times Q$ because $\Sigma^{\ZZ^d} \times Q$ is the only clopen set of full measure. If $T$ is not injective, then for some $i \neq j$ we have $D_i \cap D_j \neq \varnothing$. Then $D = D_i \cap D_j$ is a nonempty clopen set, and thus has positive measure. It follows that $\mu(\bigcup_i D_i) \leq \sum_i \mu(D_i) - \mu(D) < 1$, so $\bigcup_i D_i \subsetneq \Sigma^{\ZZ^d} \times Q$. Of course, since injectivity and surjectivity are equivalent, they are both equivalent to bijectivity, and thus to reversibility of $T$.
	
	The argument given above in fact shows that reversibility is equivalent to preserving the uniform Bernoulli measure in the forward sense -- if $T$ is reversible, then $\mu(T(A)) = \mu(A)$ for all clopen sets $A$, and thus for all Borel sets, while if $T$ is not reversible, then there is a disjoint union of cylinders $C \cup D$ such that $\mu(T(C \cup D)) < \mu(C \cup D)$.
	
	For measure-preservation in the usual (backward) sense, observe that the reverse of a reversible Turing machine is reversible and thus measure-preserving in the forward sense, so a reversible Turing machine must itself be measure-preserving in the traditional sense. If $T$ is not reversible, then $\mu(T(C \cup D)) < \mu(C \cup D)$ for some disjoint cylinders $C$ and $D$ large enough that $T|_C$ and $T|_D$ are measure-preserving homeomorphisms. Then for $E = T(C) \cap T(D)$ we have $\mu(T^{-1}(E)) \geq \mu((T^{-1}(E) \cap C) \cup (T^{-1}(E) \cap D)) = 2\mu(E)$. 
\end{proof}

\begin{remark}
	The proof is based on showing that every Turing machine is a \define{local homeomorphism} and preserves the measure of all large-radius cylinders $C$ in the forward sense $\mu(T(C)) = \mu(C)$. Note that preserving the measure of large-radius cylinders in the forward sense does not imply preserving the measure of all Borel sets (or even all cylinders), in general. For example, the machine with $n \geq 2, k = 1$ which turns the symbol in $F =\{\vec{0}\}$ to $0$ without moving the head satisfies $\mu([p]) = \mu(T[p])$ for any $p \in \Sigma^S$ with $S \supset F$. But $\mu(\Sigma^{\ZZ^d} \times Q) = 1$ and $\mu(T(\Sigma^{\ZZ^d} \times Q)) = \mu([0]) = 1/2$.
\end{remark}

\begin{remark}
	Using Proposition~\ref{injective_surjective_tape_vs_head} and Theorem~\ref{theorem_injective_surjective_reversible} we see that also under
	the moving head model injectivity and surjectivity are equivalent.
\end{remark}

We can also use the uniform measure to define the average movement of a Turing machine.

\begin{definition}
	Let $T \in \TM_{\fix}(\ZZd,n,k)$ with shift indicator function $s \colon \Sigma^{\ZZd} \times Q \to \ZZd$. 
	We define the \define{average movement} $\alpha(T) \in \RR^d$ as
	\[ \alpha(T) := \EE_{\mu}(s) = \int_{\Sigma^{\ZZd}\times Q} s(x,q) d\mu, \]
	where $\mu$ is the uniform measure on $\Sigma^{\ZZd}\times Q$.
	For $T$ in $\TM(\ZZd,n,k)$ we define $\alpha$ as the application to its image under the canonical epimorphism $\Psi$ from Definition~~\ref{def:morphism}, that is, $\alpha(T) := \alpha(\Psi(T))$.
\end{definition}

Of course, as $s(x,q)$ depends only on finitely many coordinates of $x$, this integral is actually a finite sum over the cylinders $p \in \Sigma^F$ for some finite $F \subset \ZZd$, and thus we have $\alpha(T) \in \QQ^d$ for all $T$. The following lemma shows that in fact $\alpha$ is an homomorphism.

\begin{lemma}
	The map $\alpha \colon \RTM_{\fix}(\ZZd,n,k) \to \QQ^d$ is a group homomorphism.
\end{lemma}

\begin{proof}
	If $T_1,T_2 \in \RTM_{\fix}(\ZZd,n,k)$ then since reversibility implies measure-preservation, we have
	\begin{align*}
	\alpha(T_1 \circ T_2) &= \EE_{\mu}(s_{T_1 \circ T_2}) \\
	&= \EE_{\mu}(s_{T_1}\circ T_2+s_{T_2}) \\
	&= \EE_{\mu}(s_{T_1}\circ T_2) + \EE_{\mu}(s_{T_2}) \\
	&= \EE_{\mu}(s_{T_1})+ \EE_{\mu}(s_{T_2}) \\
	&= \alpha(T_1) + \alpha(T_2).
	\end{align*}
	where $\EE_{\mu}(s_{T_1 \circ T_2}) = \EE_{\mu}(s_{T_1}\circ T_2+s_{T_2})$ holds because $s_{T_1 \circ T_2}(x, q) = s_{T_1}(T_2(x, q)) + s_{T_2}(x, q)$ for all $(x,q)\in \Sigma^{\ZZd}\times Q$.
\end{proof}

\section{Subgroups and generators}
\label{section.subgroups}	
In this section we study several subgroups of $\RTM(\ZZd,n,k)$. The main result of this section is that there is a finitely-generated subgroup of ``elementary Turing machines''. In the following sections, we show that in pure computability terms, these are able to simulate general Turing machines.

The group of elementary Turing machines $\EL(\ZZd,n,k)$ is the subgroup of reversible Turing machines which is generated by the union of two natural subgroups: $\LP(\ZZd,n,k)$, the group of local permutations and $\FG(\ZZd,n,k)$, the group of reversible finite-state automata. These two groups separately capture the dynamics of changing the tape and moving the head. We also define the group of oblivious Turing machines $\OB(\ZZd,n,k)$ as an extension of $\LP(\ZZd,n,k)$ where arbitrary tape-independent moves are allowed.

The main results we prove about these subgroups are the following, when $n \geq 2$:
\begin{itemize}
	\item $\RFA(\ZZ^d,n,k)$ is not finitely generated (Theorem~\ref{thm:FGFinitelyGenerated}), 
	\item $\RTM(\ZZd,n,k)$ is not finitely generated (Theorem~\ref{RTM_is_not_finitely_generated}). 
	\item $\RFA(X,1)$ is generated by ``orbitwise shifts'' and ``controlled position swaps'' for any one-dimensional subshift $X$ (Theorem~\ref{thm:RFAGenerators}),
	\item $\OB(\ZZ^d,n,k)$ is finitely generated (Theorem~\ref{thm:OBFG}), 
	\item $\EL(\ZZ,n,k)$ is finitely generated (Theorem~\ref{theorem_EL_is_FG}). 
\end{itemize}

For the definitions of ``orbitwise shifts'' and ``controlled position swap', see Section~\ref{sec:FG_generators}. For any class $\CL(\ZZd,n,k)$ of Turing machines with moving head we denote by $\CL_{\fix}(\ZZd,n,k)$
the corresponding class of moving tape machines, that is, the image of the class under the morphism $\Psi$ from Definition~\ref{def:morphism}.



\subsection{Definitions of subgroups}

\subsubsection{Oblivious Turing machines}

For $\vec v \in \ZZ^d$, define the machine $T_{\vec v}$ which does not modify the state or the tape, and moves the head by the vector $\vec v$ on each step. Denote $\SHIFT(\ZZ^d,n,k) = \langle \{T_{\vec{v}} \}_{\vec{v} \in \ZZ^d} \rangle$. Clearly $\SHIFT(\ZZ^d,n,k) \cong \ZZ^d$. Define also $\SP(\ZZ^d,n,k)$ as the \define{state-permutations}: Turing machines that never move and only permute their state as a function of the tape.

\begin{definition}
	We define the group $\LP(\ZZ^d,n,k)$ of \define{local permutations} as the subgroup of Turing machines in $\RTM(\ZZ^d,n,k)$ whose shift-indicator is trivial, that is, the constant function $x \mapsto \vec{0}$.
	
	We define also $\OB(\ZZ^d,n,k) = \langle \SHIFT(\ZZ^d,n,k), \LP(\ZZ^d,n,k) \rangle$, the group of \define{oblivious Turing machines}. 
\end{definition}

In other words, $\LP(\ZZ^d,n,k)$ is the group of reversible machines that do not move the head, and $\OB(\ZZ^d,n,k)$ is the group of reversible Turing machines whose head movement is independent of the state and the tape contents. 
Note that in the definition of both groups, we allow changing the state as a function of the tape, and vice versa, thus $\SP(\ZZ^d,n,k) \leqslant \LP(\ZZ^d,n,k)$.

\begin{remark}
	The restricted wreath products $H \wr \ZZ^d$ where $H$ is a finite group are sometimes called \define{generalized lamplighter groups}, the original lamplighter group being $\ZZ/2\ZZ \wr \ZZ$. Thus, $\OB(\ZZd,n,k)$ can be seen as a doubly generalized lamplighter group, since the subgroup of $\OB(\ZZd,n,k)$ generated by the local permutations $\LP(\ZZd,n,1)$ with radius $0$ and $\SHIFT(\ZZd,n,1)$ is isomorphic to the wreath product $S_n \wr \ZZd$ of the symmetric group $S_n$. Further on we show that, similar to lamplighter groups, $\OB(\ZZd,n,k)$ is also finitely generated.
\end{remark}

\subsubsection{Finite-state automata}

\begin{definition}
	We define the group $\FG(\ZZ^d,n,k)$ of \define{reversible finite-state automata} as the group of reversible $(\ZZ^d,n,k)$-Turing machines that do not change the tape. That is, the local rules are of the form $f(p,q) = (p,q',\vec{v})$ for all entries $p \in \Sigma^F, q \in Q$.
\end{definition}

Similarly, for a subshift $X\subset \Sigma^{\ZZd}$ we let $\FG(X,k)$ be the subgroup machines in $\RTM(X,1)$ that do not change the tape.

%
%

This group is ``orthogonal'' to $\OB(\ZZ^d,n,k)$ in the following sense,

	\begin{align*}
	\FG(\ZZ^d,n,k) \cap \LP(\ZZ^d,n,k) &= \SP(\ZZ^d,n,k) \\
	\FG(\ZZ^d,n,k) \cap \OB(\ZZ^d,n,k) &= \langle \SP(\ZZ^d,n,k), \SHIFT(\ZZ^d,n,k) \rangle
	\end{align*}

\begin{remark}
	It follows directly from the definitions that the group $\FG(\ZZ^d,n,1)$ is isomorphic to the topological full group of the full $\ZZ^d$-shift on $n$ symbols as defined in~\cite{GiPuSk99}. Similarly, if we fix a subshift $X$, then $\FG_{\text{fix}}(X,1)$ is isomorphic to the topological full group of the shift action on $X$. The subscript ``fix'' is only needed when $X$ is not strongly locally aperiodic, see~Lemma~\ref{group_isomorphism_lemmabis}.
\end{remark}


As usual, the case $n = 1$ is not particularly interesting, and we have that $\FG(\ZZ^d,1,k) = \RTM(\ZZ^d,1,k)$. In the general case the group is more interesting.

\begin{restatable}{theorem}{FGNotFinitelyGenerated}
	\label{thm:FGFinitelyGenerated}
	Let $n \geq 2$. Then $\FG(\Sigma^{\ZZ^d},n,k)$ is not finitely generated.
\end{restatable}


\begin{proof}
	We prove this in the moving-tape model. For $\vec{v} \in \ZZd$, let $\operatorname{Per}_{\vec v}(\Sigma^{\ZZd})$ be the set of configurations of $\Sigma^{\ZZd}$ whose stabilizer under the shift action contains $\vec{v}\ZZd$. Let $\ZZ_{\geq 2}$ be the set of integers $t \geq 2$. Let $\phi \colon \FG_{\text{fix}}(\Sigma^{\ZZ^d},n,k) \to (\ZZ/2\ZZ)^{\ZZ_{\geq 2}}$ be the parity homomorphism where $\phi(T)_{t}$ is the sign of the permutation $T$ performs on the finite set $\operatorname{Per}_{(t,t,\dots,t)}(\Sigma^{\ZZd}) \times Q$.
	
	As the image of a finitely generated group under a homomorphism is also finitely generated, it suffices to show that $\phi(\FG_{\text{fix}}(\Sigma^{\ZZ^d},n,k))$ is not finitely generated. It suffices thus to prove that for any finite $m\geq 2$ the restriction of $\phi(\FG_{\text{fix}}(\Sigma^{\ZZ^d},n,k))$ to $(\ZZ/2\ZZ)^{\{2,\dots,m\}}$ is surjective. From here it clearly follows that $\phi(\FG_{\text{fix}}(\Sigma^{\ZZ^d},n,k))$ is not finitely generated.
	
	Let $t \geq 2$ and $\vec v = (t,t,\dots,t)\in \ZZd$. Let $T_{t}$ be the machine which in state $q \neq 1$ acts trivially, and if $q = 1$ does the following: Let $\vec{e}=\vec{e}_1=(1,0,\dots,0,0)$
	be the first canonical basis vector. For $\vec u \in \ZZd$ let $A_{\vec u} = \vec u + \{0,\dots,t-1\}\times\cdots\times \{0,\dots,t-1\}$. Then, on configuration $x\in \Sigma^{\ZZd}$,
	\begin{itemize}
		\item if the restriction of $x$ to $A_{\vec{0}}$ contains a unique $1$ which is at $\vec 0$, and is otherwise identically zero, shift the configuration by $-\vec{e}$,
		\item if the restriction of $x$ to $A_{\vec{e}}$ contains a unique $1$ which is at $\vec{e}$, and it otherwise identically zero, shift the configuration by $\vec{e}$,
		\item otherwise do nothing.
	\end{itemize}
	
	The machine $T_{t}$ is an involution and thus is reversible. From this construction it follows that $\phi(T_{t})_{t} = 1$ and $\phi(T_{t})_{t'} = 0$ for all $2 \leq t' < t$.
	
	Now, let $y \in (\ZZ/2\ZZ)^{\{2,\dots,m\}}$. Let $M_1 = \ID$. Iteratively for $2 \leq j \leq m$ construct \[ M_j = \begin{cases}
	M_{j-1} & \mbox{ if } y_j = \phi(M_{j-1})_j\\
	T_j \circ M_{j-1} & \mbox{ if } y_j \neq \phi(M_{j-1})_j
	\end{cases}  \]
	
	As $\phi$ is a homomorphism it follows that $\phi(M_m) = y$ and therefore the restriction of $\phi(\FG_{\text{fix}}(\Sigma^{\ZZ^d},n,k))$ to $(\ZZ/2\ZZ)^{\{2,\dots,m\}}$ is surjective.\end{proof}

We shall introduce a new point of view on finite-state machines, which we call the \define{permutation model}. This model be helpful in the upcoming proof that elementary Turing machines are finitely generated. In this model, we associate to every Turing machine in $\RFA_{\text{fix}}(\ZZd,n,k)$ an automorphism of $(\Sigma\times (\ZZ/2\ZZ)^Q)^{\ZZd}$. The main idea behind this correspondence is that given $T \in \RFA_{\text{fix}}(\ZZd,n,k)$, every configuration $x\in \Sigma^{\ZZd}$ induces an action over $\ZZd \times Q$, namely, a head pointing at some position in $\ZZd$ in state $q$ would move to a new pair (position,state) in $\ZZd \times Q$ after applying $T$. As $T$ is reversible, this induces a permutation over $\ZZd \times Q$ which can be applied simultaneously to an arbitrary number of heads, which can in turn be represented as a configuration in $((\ZZ/2\ZZ)^Q)^{\ZZd}$. We make this embedding precise in the proof of the following proposition.

\begin{proposition}
	\label{prop:RFAInAut}
	The group $\RFA_{\text{fix}}({\ZZ^d},n,k)$ embeds into $\Aut(A^{\ZZ^d})$ for $|A| = n 2^k$.
\end{proposition}

\begin{proof}
	Note that any pair $(T,x) \in \RFA_{\text{fix}}({\ZZ^d},n,k) \times \Sigma^{\ZZ^d}$ induces a permutation $\sigma_{T,x}$ over ${\ZZ^d} \times Q$. Namely, let $s$ be the shift-indicator of $T$, and let
	\[ \sigma_{T,x}(\vec v, q) = (\vec v - s(\sigma^{-\vec v}(x), q), T_2(\sigma^{-\vec v}(x), q)) \]
	
	We identify $A$ as the alphabet $\Sigma \times (\ZZ/2\ZZ)^Q$ and thus the configurations of $A^{\ZZ^d}$ can be seen as pairs $(x,y) \in \Sigma^{\ZZ^d} \times ((\ZZ/2\ZZ)^{Q})^{\ZZ^d}$.
	
	Finally, let $\varphi\colon \RFA_{\text{fix}}({\ZZ^d},n,k)  \to \Aut(\ag^{\ZZ^d})$ be the map defined by $\varphi(T) = \phi_T$ where:
	
	\[(\phi_T)_1(x, y)_{\vec{v}}= x_{\vec{v}},\]
	\[(\phi_T)_2(x, y)_{\sigma_{x,T}(\vec{v},q)} = 1 \mbox{ if and only if } y_{(\vec{v},q)} = 1.\]
	
	That is, $\phi_T$ does not modify the $\Sigma^{\ZZ^d}$ tape, and for every position $(\vec{v},q) \in {\ZZ^d} \times Q$ marked with a one, we interpret it as a Turing machine head in state $q$ in position $\vec{v}$ and mark in the image the state and position it would end up after applying $T$. This is clearly a cellular automaton on $\Sigma^{\ZZ^d} \times ((\ZZ/2\ZZ)^{Q})^{\ZZ^d}$ as the shift indicator of $T$ has a finite radius.
	
	We claim $\varphi$ is an embedding. A direct computation shows that the permutation induced by $T_1 \circ T_2$ is just $\sigma_{T_1,x} \circ \sigma_{T_2,x}$ thus showing that $\varphi$ is a homomorphism. Now, if $T_1 \neq T_2$ there is some pair $(x,q)$ where they act differently. If we consider the configuration $(x,y)$ where $y_{0^d,q} =1$ and $0$ elsewhere. Clearly  $\phi_{T_1}(x,y) \neq \phi_{T_2}(x,y)$. Therefore $\varphi$ is injective. 
\end{proof}

\begin{definition}
	For any $T \in \RFA_{\text{fix}}(\ZZd,n,k)$ the automorphism $\phi_T \in \Sigma^{\ZZ^d} \times (\ZZ/2\ZZ)^{{\ZZ^d} \times Q}$ obtained by appling the embedding of the previous proof is called the \define{permutation model} of $T$.
\end{definition}

The permutation model has the remarkable property of being linear in the second component. Namely, given $T \in \RFA_{\text{fix}}(\ZZd,n,k)$ and its permutation model $\phi(T)$, we have that for any $x \in \Sigma^{\ZZ^d}$ and $y,z \in (\ZZ/2\ZZ)^{{\ZZ^d} \times Q}$ we have \[\phi_T(x,y+z) = (x,(\phi_T)_2(x,y)+(\phi_T)_2(x,z)),\]
where the sum is computed coordinate-wise.

It is known that the automorphism group of any nontrivial full shift embeds in the automorphism group of any uncountable sofic shift \cite{Sa15}. Thus we have the following corollary:

\begin{corollary}
	If $X$ is an uncountable sofic $\ZZ$-subshift, then $\RFA(\ZZ,n,k)$ embeds into $\Aut(X)$. In particular $\RFA(\ZZ,n,k)$ embeds into $\Aut(\ag^{\ZZ})$ for any $|\ag| \geq 2$.
\end{corollary}

\subsubsection{Elementary Turing machines}

\begin{definition}
	The group of \define{elementary Turing machines} $\EL(\ZZ^d,n,k)$ is the group generated by finite-state machines and local permutation, that is\[\EL(\ZZ^d,n,k) := \langle \FG(\ZZ^d,n,k),\LP(\ZZ^d,n,k)\rangle.\]
\end{definition}

The group of elementary Turing machines is generated by machines which either do not change the tape or do not move the head, and its aim is to approximate the group $\RTM(\ZZ^d,n,k)$ using simple building blocks. In other words, it can be understood as a sufficiently rich class of Turing machines which can be constructed from simple atoms. For instance, Langdon's ant~\cite{La86} is an example of a machine in $\EL(\ZZ^2,2,4)$.

Clearly the group $\LP(\ZZ^d,n,k)$ is not finitely generated because it is locally finite and infinite. We have also shown that $\RFA(\ZZ^d,n,k)$ is not finitely generated and will soon show that $\RTM(\ZZd,n,k)$ is not finitely generated. However, we are later going to show that both $\OB(\ZZ^d,n,k)$ and $\EL(\ZZ,n,k)$ are finitely generated.

Before studying these machines, we show that $\RTM(\ZZd,n,k)$ is not finitely generated. This proof uses the average movement homomorphism $\alpha$ defined in Section~\ref{subsection_measure}.

\begin{lemma}
	\[ \alpha(\RTM(\ZZd,n,k)) = \left\langle \frac{e_i}{kn^j} : j \in \NN, i \in \{1,\dots,d\} \right\rangle \leqslant (\QQ^d, +) \]
\end{lemma}

\begin{proof}
	Consider the $(\ZZ,n,k)$-Turing machine $T_{\texttt{SURF},m}$ given by the local function $f \colon \Sigma^{\{0,\dots,m\}} \times Q \to \Sigma^{\{0,\dots,m\}} \times Q \times \ZZ$ which is defined as follows: For $a \in \Sigma$ and $q < k$ let $f(0^ma,q) = (0^ma,q+1,0)$ and $f(0^ma,k) = (a0^m,1,1)$. Otherwise $f(u,q)= (u,q,0)$. This machine is reversible, and satisfies that $\alpha(T_{\texttt{SURF},m}) = 1/{kn^{m}}$. This machine can easily be extended to a $(\ZZd,n,k)$-Turing machine with average movement $(1/{kn^m},0,\dots,0)$, and an analogous construction yields a Turing machine with the exact same movement in other coordinates. Thus we obtain that \[\left\langle \frac{e_i}{kn^j} : j \in \NN, i \in \{1,\dots,d\} \right\rangle \leqslant \alpha(\RTM(\ZZd,n,k)).\]
	
	To obtain the equality, observe that the integral $\int_{\Sigma^{\ZZd}\times Q} s(x,q) d\mu$ defining the average movement is a finite sum over cylinders, the contribution of each cylinder is an integer vector, and the measure of a cylinder is in the group generated by $\frac{1}{kn^j}$. Thus every element in $\alpha(\RTM(\ZZd,n,k))$ can be written as a finite sum of $\frac{e_i}{kn^j}$. \end{proof}

The image under the homomorphism $\alpha$ of $(\RTM(\ZZd,n,k))$ computed above is not finitely generated whenever $n \geq 2$, therefore we conclude that $\RTM(\ZZd,n,k)$ cannot be finitely generated for $n \geq 2$.

\begin{restatable}{theorem}{RTMIsNotFinitelyGenerated}
	\label{RTM_is_not_finitely_generated}
	For $n \geq 2$, the group $\RTM(\ZZd,n,k)$ is not finitely generated.
\end{restatable}

Although $\alpha$ is not a homomorphism on $\TM(\ZZd,n,k)$, using Theorem~\ref{theorem_injective_surjective_reversible} we obtain that $\TM(\ZZd,n,k)$ cannot be finitely generated either.

\begin{theorem}
	For $n \geq 2$ the monoid $\TM(\ZZd,n,k)$ is not finitely generated.
\end{theorem}

\begin{proof}
	Let $T \in \RTM(\ZZd,n,k)$. As $T$ is injective, we conclude that if $T = T_n \circ \dots \circ T_1$ then $T_1$ is also injective. By Theorem~\ref{theorem_injective_surjective_reversible} this means that $T_1$ is reversible, which in turn implies $T_2$ is injective and so on. Therefore, if $\TM(\ZZd,n,k)$ is generated by $S$, then there exists $S' \subset S$ such that $\langle S' \rangle = \RTM(\ZZd,n,k)$. But every such $S'$ is infinite by Theorem~\ref{RTM_is_not_finitely_generated} and thus $S$ must also be infinite. 
\end{proof}

In the remainder of the section we study how close $\EL(\ZZ^d,n,k)$ is to $\RTM(\ZZ^d,n,k)$.  We also give a natural infinite generating set for $\RFA(X,k)$ for every $\ZZ$-subshift $X$. Later on we show that $\EL(\ZZ,n,k)$ is finitely generated which implies that $\EL(\ZZ,n,k)$ is strictly contained in $\RTM(\ZZ,n,k)$. However, we will first provide a simple proof of that fact which works in any dimension.

\begin{restatable}{proposition}{Qp}
	Let $\frac{1}{k}\ZZ = \{\frac{n}{k} \mid n \in \ZZ\}$. We have that \[\alpha(\FG(\ZZd,n,k)) = \alpha(\EL(\ZZd,n,k)) = \left(\frac{1}{k}\ZZ \right)^d.\] In particular, $\EL(\ZZd,n,k)$ is strictly contained in $\RTM(\ZZd,n,k)$.
\end{restatable}

\begin{proof}
	Clearly $\alpha(T)=0$ for every $T\in \LP(\ZZ,n,k)$, so the first equality holds. Let us then consider average movement values of finite-state automata.
	The machine $T_j$ that increments the state by $1$ on each step (modulo the number of states) and walks one step along the $j$th axis whenever it enters the state $1$, has $\alpha(T_j) = (0,\ldots,0,1/k,0,\ldots,0)$. We obtain
	\[ \left(\frac{1}{k}\ZZ \right)^d = \langle \alpha(T_j) L j \leq d \rangle \leqslant \alpha(\FG(\ZZd,n,k)). \]
	
	Next, let us show that for every finite-state machine $T$, we have $\alpha(T) \in (\frac{1}{k}\ZZ )^d$. For this, consider the behavior of $T$ on the all-zero configuration. Given a fixed state $q$, $T$ moves by an integer vector $\vec v_q$, thus contributing $\frac{1}{k}\vec v_q$ to the average movement. Let $\vec{v} = \sum_{q \in Q}\frac{1}{k}\vec v_q$ be the average movement of $T$ on the all-zero configuration.

	We claim that $\alpha(T) = \vec v$. Note that by composing $T$ with a suitable combination of the machines $T_{j}$ and their inverses, it is enough to prove this in the case $\vec v = \vec{0}$. Now, for a large $m$, Let $p \in \Sigma^{\{-m,\dots,m\}^d}$ be a pattern, $\vec{u} \in \{-m,\dots,m\}^d$ a position and $q \in Q$ a state. Complete $p$ to a configuration $x_p \in \Sigma^{\ZZd}$ by writing $0$ in every cell outside $\{-m,\dots,m\}^d$. Write $\alpha_m(T)$ for the average movement of $T$ for the finitely many choices of $p, u, q$. Formally, if $s_T$ is the shift indicator of $T$: \[\alpha_m(T) = \frac{1}{k (2m+1)^d n^{(2m+1)^d} } \sum_{p,u,q} s_T(\sigma^{-\vec{u}}(x_p),q).\] As $m \rightarrow \infty$, it is easy to show that $\alpha_m(T) \rightarrow \alpha(T)$, as the movement vector of $T$ is distributed correctly in all positions except at the border of $\{-m,\dots,m\}^d$ which grows as $o(m^d)$.
	
	On the other hand, for any fixed $p \in \Sigma^{\{-m,\dots,m\}^d}$ the average movement of $T$ on $x_p$ started from a random state and a random position is $\vec 0$, that is, \[\sum_{u,q} s_T(\sigma^{-\vec{u}}(x_p),q) = \vec{0}.\] This follows from the fact that $T \in \FG(\ZZ^d,n,k)$ and thus the action is simply a permutation of the set of position-state pairs and the fact that $\vec{v}= \vec{0}$. From here we conclude that the sum restricted to $u \in \{-m,\dots,m\}^d$ is $o(m^d)$. It follows that $\alpha(T) = \lim \alpha_m(T) = \vec 0$. 
\end{proof}

\subsection{The oblivious Turing machines are finitely generated}
\label{sec:ObliviousFG}

In this section, we show that $\OB(\ZZd,n,k)$ is finitely generated. Our proof is based on the existence of strongly universal reversible gates for permutations of $A^m$, which can be found in~\cite{AaGrSc15,Xu15} for the binary alphabet case, and generalized to other alphabets in~\cite{Boykett16}. We need a finite generating set for permutations of $Q\times \Sigma^m$, and hence the proof in~\cite{Boykett16} has to be adjusted to account for non-homogeneous alphabet sizes (that is, due to possibly having $n\neq k$).

Let us remark that the case $n=1$ is trivial: The group $\LP(\ZZd,1,k)$ is finite and
$\SHIFT(\ZZd,1,k)$ is generated by the single step moves.
We hence assume that $n\geq 2$.

The following result was proved in~\cite{Boykett16}
(Lemmas 3 and 5):

\begin{lemma}
	\label{lem:graphlemma}
	Let $H=(V,E)$ be a connected undirected graph.
	\begin{enumerate}
		\item[(a)] The transpositions $(s\; t)$ for $\{s,t\}\in E$ generate $\Sym(V)$, the set of permutations of the vertex
		set.
		\item[(b)] Let $\Delta\subset \Sym(V)$ be a set of permutations of $V$ that
		contains for each edge  $\{s,t\}\in E$ a 3-cycle $(x\; y\; z)$ where $\{s,t\}\subset \{x,y,z\}$. Then $\Delta$ generates
		$\Alt(V)$, the set of even permutations of the vertex set.
	\end{enumerate}
\end{lemma}

Let $m\geq 1$, and consider permutations
of $Q\times\Sigma^m$. \define{Controlled swaps} are transpositions $(s\; t)$ where $s,t\in Q\times \Sigma^m$
have Hamming distance one. \define{Controlled 3-cycles\/} are permutations $(s\; t\; u)$ where the Hamming distances between the three
pairs are 1, 1 and 2.


Let us denote by $C_m^{(2)}$ and  $C_m^{(3)}$ the sets of controlled swaps and 3-cycles in
$\Sym(Q\times\Sigma^m)$ respectively. Let $H=(V,E)$ be the graph with vertices $V=Q\times\Sigma^m$ and
edges $\{s,t\}$ that connect elements $s$ and $t$ having Hamming distance one. This is a connected undirected graph, so we get from
Lemma~\ref{lem:graphlemma}(b) that
the controlled 3-cycles generate all its even permutations:

\begin{lemma}
	\label{lem:controlled}
	Let $n\geq2$ and $m\geq 1$. The group $\Alt(Q\times \Sigma^m)$  is generated by $C_m^{(3)}$.
\end{lemma}

Let $\ell \leq m$, and let $f$ be a permutation of  $Q\times \Sigma^{\ell}$. We can apply $f$ on
$1+\ell$ coordinates of $Q\times \Sigma^{m}$ (including the first), while leaving the other
$m-\ell$ coordinates untouched. More precisely, the \define{prefix application} $\hat{f}$ of $f$ on $Q\times \Sigma^{m}$, defined by \[\hat{f}(q,s_1,\dots ,s_\ell,\dots ,s_m) =
(f_1(q,s_1,\dots ,s_\ell), \dots ,f_{\ell+1}(q,s_1,\dots ,s_\ell), s_{\ell+1},\dots ,s_m),\]
applies $f$ on the first $1+\ell$ coordinates. To apply it on other choices of coordinates we conjugate $\hat{f}$
using rewirings of symbols.
For any permutation $\pi\in \Sym(\{1,\dots,m\})$ we define the
\define{rewiring} permutation of $Q\times \Sigma^m$ by
\[ r_\pi : (q,s_1,\dots ,s_m) \mapsto (q,s_{\pi(1)},\dots ,s_{\pi(m)}).\]
It permutes the positions of the $m$ tape symbols according to $\pi$.
Now we can conjugate the prefix application $\hat{f}$ using a rewiring to get
$\hat{f}_\pi = r_\pi^{-1} \circ \hat{f}\circ  r_\pi$, we call $\hat{f}_\pi$ an
\define{application} of $f$ in the coordinates $\pi(1),\dots,\pi(\ell)$. Let us denote by
\[[f]_m = \{ \hat{f}_\pi\ |\ \pi\in \Sym(m)\},\]
the set of permutations of $Q\times \Sigma^{m}$ that are applications of $f$.
For a set $P$ of permutations we denote by $[P]_m$ the union of $[f]_m$ over all $f\in P$.

Note that if $n$ is even and $f\in\Sym(Q\times \Sigma^\ell)$ for $\ell < m$ then $[f]_m$ only contains even permutations.
The reason is that
the coordinates not participating in the application of $f$ carry a symbol of the even
alphabet $\Sigma$. The application $[f]_m$ then consists of an even number of
disjoint permutations of equal parity -- hence the result is even. In contrast,
for the analogous reason, if $n$ is odd then $[f]_m$ only contains odd permutations
whenever $f$ is itself is an odd permutation.

\begin{lemma}
	\label{lem:iteration}
	Let $m\geq 6$, and let $G_m =\langle [C_4^{(2)}]_m\rangle$ be the group generated by the
	applications of controlled swaps of $Q\times \Sigma^4$ on $Q\times \Sigma^m$.
	If $n=|\Sigma|$ is odd then $G_m=\Sym(Q\times \Sigma^{m})$. If $n$ is even
	then $G_m=\Alt(Q\times \Sigma^{m})$.
\end{lemma}

\begin{proof}
	For even $n$, by the note above,  $[C_4^{(2)}]_m \subset \Alt(Q\times
	\Sigma^{m})$, and for odd $n$ there are odd permutations in $[C_4^{(2)}]_m$.
	So in both cases
	it is enough to show $\Alt(Q\times \Sigma^{m}) \subset G_m$. We
	also note that, obviously, $[G_{m-1}]_m\subset G_m$.
	
	Based on the
	decomposition in Figure~\ref{fig:gates}, we first conclude that
	any controlled 3-cycle $f$ of $Q\times \Sigma^m$ is a composition
	of four applications of controlled swaps of $Q\times \Sigma^{m-2}$.
	In the figure, the components of $Q\times \Sigma^m$ have been
	ordered in parallel horizontal wires, the $Q$-component being among the topmost three
	wires. Referring to the symbols in the illustration, the gate on the left is a
	generic 3-cycle $(pszabcdw\; ptzabcdw\;
	qszabcdw)$ where one of the first three wires is the $Q$-component,
	$a,b,c,d\in \Sigma$ and $w\in \Sigma^{m-6}$.
	The proposed decomposition consists of two different controlled swaps $p_1$ and $p_2$ applied twice in
	the order $f=p_1p_2p_1p_2$. Because $p_1$ and $p_2$ are involutions,
	the decomposition amounts to identity unless the input is of the form $xyz abcd w$ where
	$x\in\{p,q\}$ and $y\in\{s,t\}$. When the input is of this form, it
	is easy to very that the circuit on the right indeed amounts to the required 3-cycle. We
	conclude that $C_m^{(3)}\subset  \langle [C_{m-2}^{(2)}]_m \rangle$, for all
	$m\geq 6$. By Lemma~\ref{lem:controlled},
	\begin{equation}
	\label{eq:gateq}
	\Alt(Q\times \Sigma^m) = \langle C_m^{(3)} \rangle \subset \langle
	[C_{m-2}^{(2)}]_m \rangle.
	\end{equation}
	
	\begin{figure}[ht]
		\begin{center}
			\includegraphics[width=11.5cm]{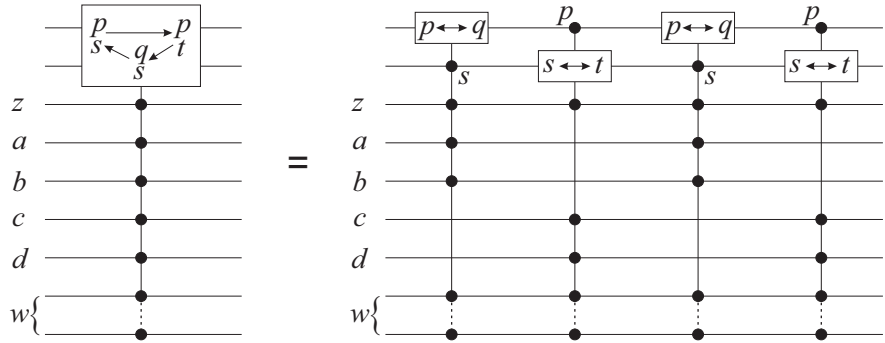}
		\end{center}
		\caption{A decomposition of a controlled 3-cycle of $Q\times \Sigma^m$ on the left into a sequence of
			four applications of
			controlled swaps of $Q\times \Sigma^{m-2}$ on the right. The ordering of the wires is such that
			topmost three wires contain the $Q$-component and the two wires
			changed by the 3-cycle (one of which may or may not be the $Q$-component). Black circles are control points: the gate computes the identity unless the
			wire carries the symbol indicated at the left of the wire or next
			to the control point.
		}
		\label{fig:gates}
	\end{figure}

	We proceed by induction on $m$. The base case $m=6$ is clear: By
	(\ref{eq:gateq}),
	\[\Alt(Q\times \Sigma^6) \subset \langle [C_{4}^{(2)}]_6 \rangle =
	G_6.\]
	
	Consider then $m>6$ and suppose
	that $G_{m-1}$ is as claimed. If $n$ is odd then, by the inductive hypothesis,
	$$
	[C_{m-2}^{(2)}]_m  \subset  [\Sym(Q\times \Sigma^{m-1})]_m \subset
	[G_{m-1}]_m \subset G_m.
	$$
	By (\ref{eq:gateq}) then $\Alt(Q\times \Sigma^m)\subset
	\langle [C_{m-2}^{(2)}]_m \rangle \subset  G_m$.
	As pointed out above, $G_m$ contains odd
	permutations (all elements of $[C_4^{(2)}]_m$ are odd), so $G_m=\Sym(Q\times \Sigma^{m})$
	as claimed.

	If $n$ is even then
	an application of a permutation of $Q\times \Sigma^{m-2}$ on
	$Q\times \Sigma^m$ is also an application of an even permutation of $Q\times \Sigma^{m-1}$
	on $Q\times \Sigma^m$. (For this reason we left two non-controlling wires
	for the gates on the
	right side of Figure~\ref{fig:gates}.) By this and the inductive hypotheses,\[
	[C_{m-2}^{(2)}]_m \subset [\Alt(Q\times \Sigma^{m-1})]_m\subset
	[G_{m-1}]_m \subset G_m,\]
	so, by (\ref{eq:gateq}), we have the required $\Alt(Q\times \Sigma^m)\subset
	G_{m}$.\end{proof}

\begin{corollary}
	\label{cor:borrowedbit}
	$[\Sym(Q\times \Sigma^{m})]_{m+1} \subset \langle [\Sym(Q\times \Sigma^4)]_{m+1}\rangle$
	for all $m\geq 5$.
\end{corollary}

\begin{proof}
	If $n$ is even then $[\Sym(Q\times \Sigma^{m})]_{m+1} \subset \Alt(Q\times
	\Sigma^{m+1})$ and if $n$ is odd then
	$[\Sym(Q\times \Sigma^{m})]_{m+1} \subset \Sym(Q\times
	\Sigma^{m+1})$. In either case, the claim follows from
	Lemma~\ref{lem:iteration} and $C_4^{(2)}\subset \Sym(Q\times
	\Sigma^4)$.\end{proof}

In Corollary~\ref{cor:borrowedbit}, arbitrary permutations of $Q\times\Sigma^{m}$
are obtained as projections of permutations of
$Q\times\Sigma^{m+1}$.
The extra symbol is an ancilla that can have an arbitrary initial
value and is returned back to this value in the end.
Such an ancilla is called a ``borrowed bit'' in~\cite{Xu15}.
It is needed in the case of even $n$
to facilitate implementing odd permutations of $Q\times \Sigma^{m}$.

Now we are ready to prove the following theorem.

\begin{restatable}{theorem}{OBFG}
	\label{thm:OBFG}
	$\OB(\ZZd,n,k)$ is finitely generated.
\end{restatable}


\begin{proof}
	We construct a finite generating set $A_1\cup A_2\cup A_3$. Let $A_1$ contain the one step moves $T_{e_i}$ for $i=1,\dots ,d$. These clearly generate
	$\SHIFT(\ZZd,n,k)$.
	
	Each $T\in \LP(\ZZd,n,k)$
	is defined by a local rule $f\colon \Sigma^F\times Q \rightarrow  \Sigma^F\times Q\times\{\vec{0}\}$
	with a finite $F\subset \ZZd$. To have injectivity, we clearly need that $\pi\colon(p,q)\mapsto (f_1(p,q), f_2(p,q))$
	is a permutation of $\Sigma^F\times Q$. We denote $T=P_\pi$.
	Let us fix an arbitrary $E\subset  \ZZd$ of size $4$, and let
	$A_2$ be the set of all $P_\pi\in \LP(\ZZd,n,k)$ determined by $\pi\in \Sym(\Sigma^E\times Q)$.
	
	For any  permutation $\alpha$ of $\ZZd$ with finite support, we define
	the cell permutation machine $C_\alpha\colon (p,q) \mapsto (p',q)$, where $p'_{\vec v}=p_{\alpha({\vec v})}$
	for all ${\vec v}\in\ZZd$. These are clearly in $\LP(\ZZd,n,k)$. We take $A_3$ to consist of the cell
	permutation machines $C_i=C_{(0\; e_i)}$ that, for each $i=1,\dots ,d$, swaps the contents of the currently scanned cell and its
	neighbor with offset $e_i$.
	
	Observe that $A_1$ and $A_3$ generate all cell permutation
	machines $C_\alpha$. First, conjugating $C_i$ with
	$T_{\vec v}\in\SHIFT(\ZZd,n,k)$ gives the cell permutation machine $C_\alpha=T_{\vec v}^{-1}C_iT_{\vec
		v}$ for the transposition $\alpha=({\vec v}\; {\vec v} + e_i)$.
	Such transpositions generate all permutations of $\ZZd$ with finite
	support: This follows from Lemma~\ref{lem:graphlemma}(a) by
	considering a finite connected grid graph containing the support of
	the permutation.
	
	Consider then an arbitrary
	$P_\pi\in \LP(\ZZd,n,k)$, where $\pi\in \Sym(\Sigma^F\times Q)$.
	We can safely assume $|F|\geq 5$.
	Let us pick one ancilla $v\in\ZZd\setminus F$ and denote $F'=F\cup \{v\}$.
	By Corollary~\ref{cor:borrowedbit}, $P_\pi$ is a composition of
	machines of type $P_\rho$ for $\rho\in \Sym(\Sigma^H\times Q)$
	where $H\subset F'$ has size $|H|=4$. It is enough to be able to generate these.
	Let $\alpha$ be a permutation of $\ZZd$ that exchanges $E$ and $H$,
	two sets of cardinality four. Then $C_\alpha^{-1} P_\rho C_\alpha
	\in A_2$, which implies that $P_\rho$ is generated by $A_1\cup A_2\cup
	A_3$.\end{proof}

\subsection{Generators for finite-state automata on any one-dimensional subshift}
\label{sec:FG_generators}

In this section, we show that while finite-state automata are not finitely generated on the one-dimensional full shift, the group of finite-state automata has a natural generating set on every one-dimensional subshift. 


We shall show that a generating set for this group is composed of two types of objects: there is a (possibly infinite rank) abelian group that translates orbits and is an abstracted notion of average movement, and a collection of elements with zero average-movement which is generated by ``controlled position swaps'' which are similar in spirit to the controlled swaps of the previous section. We show this result for the topological full group (one state) and will extend it to the case of multiple states using Lemma~\ref{lem:OneState}.

Similar results are known for topological full groups of minimal systems (see for example~\cite{Ma06,GrMe11}). The main additional issue is with average movement, which does not actually apply in our main application of the full shift. Usually aperiodicity is assumed when studying topological full groups, but periodic points do not pose any problems, except for the small issue that without strong local aperiodicity we might not have $\RFA_{\text{fix}}(X,k) \cong \RFA(X,k)$. We study $\RFA(X,k)$, understanding that the groups are the same in all cases we are actually interested in.



As suggested, we begin by showing that on any $\ZZ$-subshift, there is a natural generalization of the average-movement homomorphism which measures the average movement separately on every orbit. This homomorphism coincides with the average-movement $\alpha$ when $X = \Sigma^{\ZZ}$, and on minimal subshifts it corresponds to the index map defined for topological full groups~\cite{GiPuSk99}. 

The results of this section will rely strongly on our subshifts being one-dimensional. Our analogue of the index map will be based on counting heads to the left and right of some interval, which is not straightforward to generalize to the multidimensional case.

Suppose for the rest of this section that $X$ denotes a one-dimensional subshift. In this section, we write $\RFA(X) = \RFA(X,1)$. We can see the shift $\sigma$ as an element of $\RFA(X)$ defined by $\sigma(x^i_1) = x^{i-1}_1$. 

Let $\sim_n'$ be the relation on the language $L_n(X)$ of words of length $n$ appearing on $X$ defined by $u \sim_n' v$ if there exists $x \in X$ such that $u$ and $v$ occur in $x$. Let $\sim_n$ be the transitive closure of $\sim_n'$. We write $\sim$ for $\sim_n$ the equivalence relation induced by the collection of $\sim_n$ for every $n \in \NN$. 

\begin{definition}
A subshift $X\subset \Sigma^{\ZZ}$ is \define{weakly chain-transitive} if for every $n \in \NN$ and $u,v \in L_n(X)$ we have $u \sim v$. 
\end{definition}

We use the prefix ``weak'' to distinguish this notion from other (stronger) notions of chain-transitivity in the literature such as the one found in~\cite{Kaz08}.

For any subshift $X$, let $\widehat X \subset \prod_n L_n(X)/{\sim_n}$ be the image of $X$ under the factor map $\pi(x)_n = x_{[0,n-1]}/\!\sim_n$. Considering $\widehat X$ as a dynamical system with a trivial $\ZZ$-action, it is easy to see that $\pi\colon X \to \widehat X$ is a factor map: $x_{[0,n-1]} \sim x_{[-1,n-2]}$ for all $x \in X$ and $n \in \NN$, so $\pi(\sigma(x)) = \pi(x) = \ID(\pi(x))$. Notice that $X$ is weakly chain-transitive if and only if $\widehat X$ is a singleton.

Recall from Definition~\ref{def:sparseshift} that given a subshift $X\subset \Sigma^{\ZZ}$, we define a subshift $\sqrt[k]{X} \subset (\Sigma \cup \{\#\})^{\ZZ}$ where configurations of $X$ occur in a coset of $k\ZZ$ and all other positions are filled with the symbol $\#$.

\begin{lemma}
Let $X$ be a weakly chain-transitive subshift. Then $\sqrt[k]{X}$ is weakly chain-transitive.
\end{lemma}

\begin{proof}
Let $u', v'$ be two words of the same length in $\sqrt[k]{X}$. Then possibly by extending $u'$ and $v'$ we may write them as subwords of words $u''$ and $v''$ of the form \[u' \sqsubset u'' = \#^{n-1} u_1 \#^{k-1} u_2 \#^{k-1} u_3 \cdots u_{n} \#^{k-1},\]\[v' \sqsubset v'' = \#^{k-1} v_1 \#^{k-1} v_2 \#^{k-1} v_3 \cdots v_{n} \#^{k-1},\] for some $n \in \NN$ and words $u,v$ in $L_n(X)$. Since $X$ is weakly chain-transitive, there is a finite chain $u = w_0 \sim w_1 \sim \cdots \sim w_k = v$ such that $w_i$ and $w_{i+1}$ occur in the same point of $X$ for all $i$. Then the corresponding interspersed versions of the $w_i$ give a chain between $u''$ and $v''$, and thus between $u'$ and $v'$. \end{proof}

\begin{definition}
We say that $T \in \RFA(X)$ is an \define{orbitwise shift} if for every $x \in X$ there exists $k \in \ZZ$ such that\[T(x^i_1) = x^{i-k}_1 \mbox{ for every } i \in \ZZ.\] 
An \define{abstract orbitwise shift} is a continuous function $f\colon X \to \ZZ$ such that $f(x) = f(\sigma(x))$ for all $x \in X$. Write $\OS(X)$ and $\AOS(X)$ for the group of orbitwise shifts and the group of abstract orbitwise shifts respectively.
\end{definition}

Note that orbitwise shifts form a subgroup of $\RFA(X)$, and abstract orbitwise shifts form a group under pointwise addition. Orbitwise shifts are a much smaller group than $\RFA(X)$ in general. For example $\OS(X)$ is always abelian, while one can show that $\RFA(X)$ may contain a free group on two generators, see for example~\cite{ElMo12} or~\cite{BaKaSa16}.

\begin{lemma}\label{lemma_evidente}
The abstract orbitwise shifts are precisely the continuous functions $f\colon X \to \ZZ$ that factor through $\widehat X$ in the sense that $f = g \circ \pi$ for some continuous map $g\colon \widehat X \to \ZZ$.
\end{lemma}

\begin{proof}
Suppose that $f = g \circ \pi$ for some $g \colon \widehat X \to \ZZ$. Then $f$ is continuous as the composition of two continuous functions, and
\[ f(x) = g(\pi(x)) = g(\pi(\sigma(x))) = f(\sigma(x)), \]
so $f$ is an abstract orbitwise shift.

On the other hand, if $f \colon X \to \ZZ$ is continuous and $f(x) = f(\sigma(x))$, then the map $g \colon \widehat X \to \ZZ$ given by the formula $g(\pi(x)) = f(x)$ for all $x \in X$ is well-defined: suppose not, and that $\pi(x) = \pi(y)$ and $f(x) \neq f(y)$. Since $f$ is continuous and $X$ is compact, $f(z)$ depends only on finitely many coordinates of $z$, that is, there exists $n$ such that $f(z) = f(z')$ whenever $z_{\{-n,\dots,n\}} = z_{\{-n,\dots,n\}}$. Since $\pi(x) = \pi(y)$, $u = x_{\{-n,\dots,n\}}$ and $v = y_{\{-n,\dots,n\}}$ are connected by a finite chain $u = w_0 \sim w_1 \sim \cdots \sim w_k = v$ such that $w_i$ and $w_{i+1}$ occur in the same configuration of $X$ for all $0\leq i<k$.

Since the image of $f$ is determined by the restriction to the central $2n+1$ coordinates, we can write $n_i = f(w_i)$ for this unique value. Now, since $f(z) = f(\sigma(z))$ for all $z \in X$, it is easy to see that in fact $n_i = n_{i+1}$ for all $i$, by moving along orbits of points connecting the words $w_i$ and $w_{i+1}$. This shows that $g$ is well-defined. Furthermore, as the image of $f$ is determined by the restriction to the central $2n+1$ coordinates, it follows easily that $g$ is continuous. 
\end{proof}

We shall now define our analogue of the index map. Let $T \in \RFA(X)$ and let $\phi := \phi_T$ be its permutation model, so that $\phi$ is an automorphism of $X \times (\ZZ/2\ZZ)^\ZZ$. Let $r \in \NN$ such that the biradius of $\phi$ is bounded by $r$ (i.e., the maximum of radii of $\phi$ and $\phi^{-1}$ is bounded by $r$ as an automorphism). Let $u \in L(X)$ such that $|u|\geq 4r$ and let $(x, y) \in X \times (\ZZ/2\ZZ)^\ZZ$ be a configuration where 
\begin{enumerate}
	\item $x|_{\{0,\dots,|u|-1\}} = u$
	\item $y_m = 1$ if and only if $r \leq m < |u|-r$.
\end{enumerate}

Recall that the automorphism $\phi$ preserves the number $|y^{-1}(1)|$, therefore $\phi_2(x,y)$ also has $|u|-2r$ heads. We define \[ L_{\phi}(u) = |\phi(x,y)^{-1}(1) \cap \{0,\dots, 2r-1\}| -r \mbox{ and } R_{\phi}(u) = |\phi_2(x,y)^{-1}(1) \cap \{|u|-2r, \dots, |u|-1\}| -r.  \]

This definition is best explained informally. Count the number of heads on the left side of the coordinate $2r$ (exclusive) of $u$ after applying $\phi$ and call this $L^*_{\phi}(u)$. Let $R^*_{\phi}(u)$ be the number of heads on the right side of coordinate $|u|-2r$ (inclusive). We clearly have $L^*_{\phi}(u) + R^*_{\phi}(u) = 2r$, as ${\phi}$ permutes the heads on any configuration and since its biradius is $r$, coordinates in $[2r,|u|-2r-1]$ all contain heads. Then $L_{\phi}(u) = L^*_{\phi}(u) - r$ and $R_{\phi}(u) = R^*_{\phi}(u) - r$ satisfy that $L_{\phi}(u)+R_{\phi}(u)=0$. For an illustration in the case where $\phi$ is the permutation model associated to the square of the shift, see Figure~\ref{ville_head_counter}.

\begin{figure}[h!]
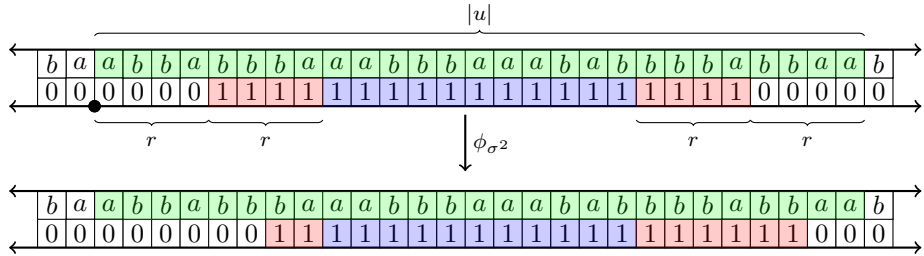

	\centering
	\include{ville_head_counter}
	\caption{For $T = \sigma^2$ its permutation model $\phi_{\sigma^2}$ moves the heads to the right twice. If we choose $r = 4$ we observe that the sum of the red zones is $2r$, the blue zone remains unchanged, $L_{\phi}(u) = -2$ and $R_{\phi}(u)=2$.}
	\label{ville_head_counter}
\end{figure}

\begin{definition}
	Let $X$ be a $\ZZ$-subshift, $T \in \RFA(X)$ and let $\phi$ be its permutation model. The \define{head index map} of $X$ and $T$ is the map $H_T \colon X \to \ZZ$ given by \[H_T(x) = \lim\limits_{n \to \infty} L_{\phi}(x|_{\{0,\dots,n-1\}}).\]
\end{definition}

Notice that the function $H_T$ is well defined, as the values of $L_{\phi}$ stabilize after $n \geq 4r$, where $r$ is the biradius of $\phi$. It is also noteworthy that choosing any value $r'$ greater than the biradius $r$ does not change the value. Indeed, as $\phi$ only permutes the heads in a finite radius, any head which is at the left of $2r'$ at a distance greater than the biradius will necessarily stay left of $2r'$ , and thus will not change the value $L_{\phi}(u) = L^*_{\phi}(u)-r'$ as the difference $r'-r$ is accounted for on $L^*_{\phi}(u)$.

\begin{lemma}
Let $X$ be a $\ZZ$-subshift. For every $T \in \RFA(X)$ we have that $H_T \in \AOS(X)$ and the map $\gamma \colon \RFA(X) \to \AOS(X)$ defined by $\gamma({\phi}_T) = H_T$ is a homomorphism. This homomorphism splits with section $\beta \colon \AOS(X) \to \RFA(X)$ defined by $\beta(f)(x) = \sigma^{f(x)}(x)$, and $\beta(\AOS(X)) = \OS(X)$, so $\beta$ gives an isomorphism between $\AOS(X)$ and $\OS(X)$.
\end{lemma}

\begin{proof}
Consider an arbitrary $\phi$ in the permutation model acting on $X \times (\ZZ/2\ZZ)^\ZZ$ with biradius $r$. We first claim that if $|u| \geq 4r$ and $au \sqsubset X$, then $L_{\phi}(au) = L_{\phi}(u)$. To see this, note that all heads except the leftmost new one map exactly as before, and thus $R^*_{\phi}(au) = R^*_{\phi}(u)$. As $R^*_{\phi}(v) + L^*_{\phi}(v) = 2r$ for every $v \in L(X)$ we conclude that $L^*_{\phi}(au)=L^*_{\phi}(u)$ and thus $L_{\phi}(au) = L_{\phi}(u)$. Symmetrically, one shows that $L_{\phi}(ua) = L_{\phi}(u)$.

Suppose now that $u, v \sqsubset x \in X$. Let $w \sqsubset x$ be any word containing both $u$ and $v$, and apply the observation of the previous paragraph repeatedly to get $L_{\phi}(u) = L_{\phi}(w) = L_{\phi}(v)$. It follows that $H_T(x) = H_T(\sigma(x))$ for all $x$. The function $H_T$ is continuous because the limit stabilizes after $n > 4r$. We conclude that $H_T \in \AOS(X)$.

To see that $\gamma(T) = H_T$ is a homomorphism, we need the following stronger fact about $L_{\phi}$: fix $m \in \NN$ and let $(x, y) \in X \times (\ZZ/2\ZZ)^\ZZ$ be a configuration where $x_{[0,|u|-1]} = u$ where $|u| \geq m+4r$, and let $y = ...000.0^rv1^{|u|-|v|-r}000...$ where $v\in (\ZZ/2\ZZ)^m$ is any word of length $m$. Then the number of heads in ${\phi}_2(x, y)|_{\{0,\dots,m+2r-1\}}$ is $|v^{-1}(1)| + r + H_{T}(x)$. To see this, simply fill in the gaps of $v$ (turn its $0$-symbols to $1$-symbols). At each step, the number of heads in the image increases by one, and the head is always added among the coordinates in $\{0,\dots,m+2r-1\}$ since ${\phi}$ has biradius $r$. After filling the gaps in $v$, we have $m+r+H_T(x)$ heads in ${\phi}_2(x, y)|_{\{0,\dots,m+2r-1\}}$ by the definition of $L_{\phi}$.

Now, computation of $\gamma(T \circ T')(x)$ can be done by first applying $\phi'=\phi_{T'}$ on the left side of a long word $u$ in $x$, with heads positioned suitably on coordinates of $u$, and then applying $\phi=\phi_T$ to the resulting scattered set of heads. By the previous paragraph, we see that $\gamma$ is a homomorphism, that is, $\gamma(T \circ {T'})(x) = \gamma(T)(x) + \gamma(T')(x)$ for all $x \in X$. The map $\beta$ is well-defined essentially by the definition of $\AOS(X)$, and it is clearly a section for $\gamma$. \end{proof}

As mentioned, the homomorphism $\gamma$ generalizes the index map defined on elements of the topological full group. It is well-known that the index map is $\ZZ$-valued in the minimal case. This is true more generally for all weakly chain-transitive subshifts, because for a weakly chain-transitive subshift the maximal invariant symbolic factor is trivial.

\begin{lemma}
If $X$ is weakly chain-transitive, then its orbitwise shifts are precisely the shifts.
\end{lemma}

\begin{proof}
Since $\widehat{X}$ is a singleton, Lemma~\ref{lemma_evidente} implies that $\AOS(X) \cong \ZZ$ and therefore $\beta \circ \gamma$ must map every element of $\RFA(X)$ to a power of $\sigma$ by the definition of $\beta$. 
\end{proof}

A clear example where the previous lemma holds is the full $\ZZ$-shift $\Sigma^{\ZZ}$. Given any $T \in \FG_{\text{fix}}(X)$, we can compose $T$ with orbitwise shifts in order to force the average movement to be zero in every orbit, that if, to obtain $T \in \FG_{\text{fix}}(X)$ such that $\gamma(T)$ is identically zero. Therefore, it only remains to find a way to generate all elements of $\RFA_{\fix}(X)$ with have no average movement on any orbit.

\begin{definition}
Let $u, v \in \Sigma^*, a \in \Sigma$. A reversible finite-state automaton $T = T_{u,a,v} \in \RFA(X)$ is a \define{controlled position swap} if $T((xu.avy)^j_1) = (xu.avy)^{1-j}_1$ for $j \in \{0,1\}$, and $T(z) = z$ for all $z \in X \times X_1$ whose image these rules do not determine. 
\end{definition}

\begin{lemma}
\label{lem:Unary}
Let $u, v \in \Sigma^*, a \in \Sigma$. Then the controlled position swap $T_{u,a,v}$ is well-defined and in $\RFA(X)$ whenever $uav$ is non-unary (i.e., when $uav \notin \{a\}^*$). 
\end{lemma}

\begin{proof}
The cylinders $[u.av]$ and $[ua.v]$ have empty intersection when $uav$ is non-unary. Thus, the conditions under which we move the head do not overlap, and the finite-state machine $T_{u,a,v}$ is well-defined. It is clearly an involution, thus invertible.
\end{proof}

More generally, given a clopen set $C \subseteq X$, we define the map $T_C \in \RFA(X)$, which we call a \define{clopen-controlled position swap}.
\[ T_C(x^i_1) = \left\{
\begin{array}{ll}
x^{i+1}_1 & \mbox{if } \sigma^{-i}(x) \in C \\
x^{i-1}_1 & \mbox{if } \sigma^{-i+1}(x) \in C \\
x & \mbox{otherwise}
\end{array}
\right. \]

As above, these maps are well-defined for any clopen set $C$ such that $C \cap \sigma(C) = \varnothing$. Any such $T_C$ can be obtained as a composition of finitely many controlled position swaps, since clopen sets can be written as a finite union of cylinder sets.

\begin{theorem}
\label{thm:RFAGenerators}
Let $X$ be a subshift. Then $\RFA(X)$ is generated by orbitwise shifts and controlled position swaps.
\end{theorem}

\begin{proof}
For $T \in \RFA(X)$, let $P_T \subset \{(0,0)\} \cup \ZZ_+^2$ be the set of all pairs $(m, k)$ such that there exists $x \in X$ and $i \in \ZZ$ such that $T(x^{i+j}_1) = x^{i+j+m}_1$ for all $0 \leq j < k$. Order these pairs lexicographically ($m$ is more significant). The order type is a suborder of $1 + \omega^2 \cong \omega^2$, thus well-founded. It is now enough to prove that, whenever $T$ has zero average movement,
\begin{enumerate}
\item $P_T$ has a maximal element $M_T$,
\item if $M_T > (0,0)$, there exists controlled position swap $T'$ such that $M_{T \circ T'} < M_T$, and
\item $T$ is the identity map if and only if $M_T = (0,0)$.
\end{enumerate}
Namely, by well-foundedness we must reach $(0,0)$ in finitely many steps by iterating the second item, and by the third we have reached the identity map.

Fix now $T \in \RFA(X)$. To see that a maximum $M_T = (m, k)$ exists, first observe that $P_T \neq \emptyset$ since $(0,0) \in P_T$. 
Next, for $(m, k) \in P_T$ the local rule of $T$ clearly gives a finite upper bound on $m$. Finally, if $m$ is maximal and $(m,k) \in P_T$ for arbitrarily large $k$, then $m > 0$ (since $P_T \subset \{(0,0)\} \cup \ZZ_+^2$ by choice), and by compactness there exists a configuration $x \in X$ satisfying $T(x^i_1) = x^{i+m}_1$ for all $i \in \ZZ$, thus average movement on $x$ is clearly nonzero.

Suppose now that $M_T = (m,k) > (0,0)$. Now, let $C \subset X$ be the set of configurations $x$ such that $\forall i \in [-k+1, 0]: T(x^i_1) = x^{i+m}_1$. Observe that $x \in C \implies T(x^1_1) = x^{1+j}_1$ where $j \leq m-2$, namely $j \leq m-1$ by the maximality of $k$, and $j = m-1$ would contradict bijectivity of $T$. Clearly $C$ is a clopen set, and $C \cap \sigma(C) = \varnothing$ because $k > 0$. We let $T' = T_C$.

Consider now an arbitrary point $x \in X$ and $i \in \ZZ$.
\begin{enumerate}
\item If $\sigma^{-i}(x) \notin C$ and $\sigma^{-i+1}(x) \notin C$, then $(T \circ T') (x^i_1) = T(x^i_1)$.
\item If $\sigma^{-i}(x) \in C$, then $(T \circ T')(x^i_1) = T(x^{i+1}_1)$.
\item If $\sigma^{-i+1}(x) \in C$, then $(T \circ T')(x^i_1) = T(x^{i-1}_1)$.
\end{enumerate}

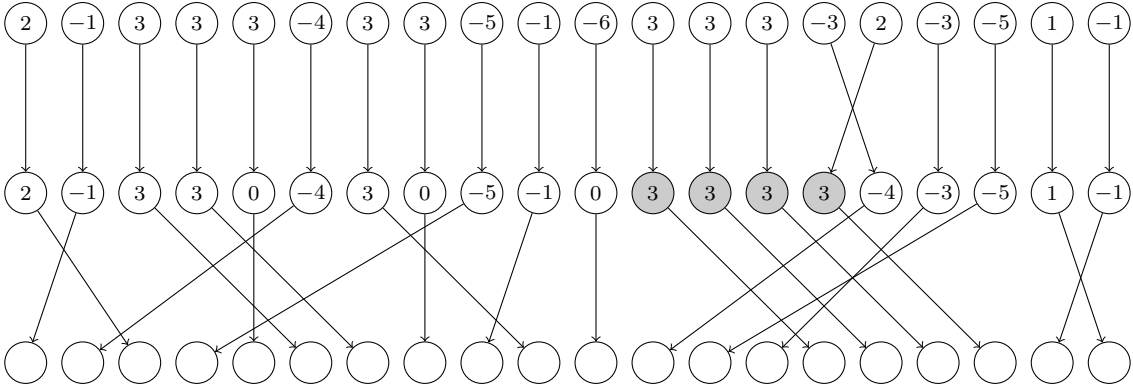
\begin{figure}
\begin{center}
\begin{tikzpicture}[scale=0.75,inner sep = 0,circle,draw,minimum width = 0.55cm]
\node[circle,draw] (c1) at (3,0) {};
\node[circle, draw] (c2) at (4,0) {};
\node[circle, draw] (c3) at (5,0) {};
\node[circle, draw] (c4) at (6,0) {};
\node[circle, draw] (c5) at (7,0) {};
\node[circle, draw] (c6) at (8,0) {};
\node[circle, draw] (c7) at (9,0) {};
\node[circle, draw] (c8) at (10,0) {};
\node[circle, draw] (c9) at (11,0) {};
\node[circle, draw] (c10) at (12,0) {};
\node[circle, draw] (c11) at (13,0) {};
\node[circle, draw] (c12) at (14,0) {};
\node[circle, draw] (c13) at (15,0) {};
\node[circle, draw] (c14) at (16,0) {};
\node[circle, draw] (c15) at (17,0) {};
\node[circle, draw] (c16) at (18,0) {};
\node[circle, draw] (c17) at (19,0) {};
\node[circle, draw] (c18) at (20,0) {};
\node[circle, draw] (c19) at (21,0) {};
\node[circle, draw] (c20) at (22,0) {};
\node[circle, draw] (b1) at (3,3) {\footnotesize $2$};
\node[circle, draw] (b2) at (4,3) {\footnotesize $-1$};
\node[circle, draw] (b3) at (5,3) {\footnotesize$3$};
\node[circle, draw] (b4) at (6,3) {\footnotesize$3$};
\node[circle, draw] (b5) at (7,3) {\footnotesize$0$};
\node[circle, draw] (b6) at (8,3) {\footnotesize$-4$};
\node[circle, draw] (b7) at (9,3) {\footnotesize$3$};
\node[circle, draw] (b8) at (10,3) {\footnotesize$0$};
\node[circle, draw] (b9) at (11,3) {\footnotesize$-5$};
\node[circle, draw] (b10) at (12,3) {\footnotesize$-1$};
\node[circle, draw] (b11) at (13,3) {\footnotesize$0$};
\node[circle, draw, fill=black!20!white] (b12) at (14,3) {\footnotesize$3$};
\node[circle, draw, fill=black!20!white] (b13) at (15,3) {\footnotesize$3$};
\node[circle, draw, fill=black!20!white] (b14) at (16,3) {\footnotesize$3$};
\node[circle, draw, fill=black!20!white] (b15) at (17,3) {\footnotesize$3$};
\node[circle, draw] (b16) at (18,3) {\footnotesize$-4$};
\node[circle, draw] (b17) at (19,3) {\footnotesize$-3$};
\node[circle, draw] (b18) at (20,3) {\footnotesize$-5$};
\node[circle, draw] (b19) at (21,3) {\footnotesize$1$};
\node[circle, draw] (b20) at (22,3) {\footnotesize$-1$};
\node[circle, draw] (a1) at (3,6) {\footnotesize$2$};
\node[circle, draw] (a2) at (4,6) {\footnotesize$-1$};
\node[circle, draw] (a3) at (5,6) {\footnotesize$3$};
\node[circle, draw] (a4) at (6,6) {\footnotesize$3$};
\node[circle, draw] (a5) at (7,6) {\footnotesize$3$};
\node[circle, draw] (a6) at (8,6) {\footnotesize$-4$};
\node[circle, draw] (a7) at (9,6) {\footnotesize$3$};
\node[circle, draw] (a8) at (10,6) {\footnotesize$3$};
\node[circle, draw] (a9) at (11,6) {\footnotesize$-5$};
\node[circle, draw] (a10) at (12,6) {\footnotesize$-1$};
\node[circle, draw] (a11) at (13,6) {\footnotesize$-6$};
\node[circle, draw] (a12) at (14,6) {\footnotesize$3$};
\node[circle, draw] (a13) at (15,6) {\footnotesize$3$};
\node[circle, draw] (a14) at (16,6) {\footnotesize$3$};
\node[circle, draw] (a15) at (17,6) {\footnotesize$-3$};
\node[circle, draw] (a16) at (18,6) {\footnotesize$2$};
\node[circle, draw] (a17) at (19,6) {\footnotesize$-3$};
\node[circle, draw] (a18) at (20,6) {\footnotesize$-5$};
\node[circle, draw] (a19) at (21,6) {\footnotesize$1$};
\node[circle, draw] (a20) at (22,6) {\footnotesize$-1$};
\draw (b1) edge[->] (c3);
\draw (b2) edge[->] (c1);
\draw (b3) edge[->] (c6);
\draw (b4) edge[->] (c7);
\draw (b5) edge[->] (c5);
\draw (b6) edge[->] (c2);
\draw (b7) edge[->] (c10);
\draw (b8) edge[->] (c8);
\draw (b9) edge[->] (c4);
\draw (b10) edge[->] (c9);
\draw (b11) edge[->] (c11);
\draw (b12) edge[->] (c15);
\draw (b13) edge[->] (c16);
\draw (b14) edge[->] (c17);
\draw (b15) edge[->] (c18);
\draw (b16) edge[->] (c12);
\draw (b17) edge[->] (c14);
\draw (b18) edge[->] (c13);
\draw (b19) edge[->] (c20);
\draw (b20) edge[->] (c19);
\draw (a1) edge[->] (b1);
\draw (a2) edge[->] (b2);
\draw (a3) edge[->] (b3);
\draw (a4) edge[->] (b4);
\draw (a5) edge[->] (b5);
\draw (a6) edge[->] (b6);
\draw (a7) edge[->] (b7);
\draw (a8) edge[->] (b8);
\draw (a9) edge[->] (b9);
\draw (a10) edge[->] (b10);
\draw (a11) edge[->] (b11);
\draw (a12) edge[->] (b12);
\draw (a13) edge[->] (b13);
\draw (a14) edge[->] (b14);
\draw (a15) edge[->] (b16);
\draw (a16) edge[->] (b15);
\draw (a17) edge[->] (b17);
\draw (a18) edge[->] (b18);
\draw (a19) edge[->] (b19);
\draw (a20) edge[->] (b20);
\end{tikzpicture}
\end{center}
\caption{The composition $T \circ T'$ illustrated with $m = 3$, $k = 4$. The movement of heads in the permutation model on a configuration $x$ is shown, and values at the nodes are those of $y$ (central row) and $z$ (top row). The four nodes forming $m^k$ are highlighted. The configuration $x$ determines these moves, but its contents are not shown.}
\label{fig:Swabadab}
\end{figure}

Defining $y, z \in \ZZ^\ZZ$ by $y_i = j$ where $T(x^i_1) = x^{i+j}_1$ and $z_i = j$ where $(T \circ T')(x^i_1) = x^{i+j}_1$, the difference between $y$ to $z$ is precisely that all subsequences of the form $m^k j \in \ZZ^{k+1}$ are replaced by $m^{k-1} (j+1) (m-1)$. Since $j \leq m-2$, the run of $m$s become strictly shorter, and no symbols larger than $m$ are introduced, thus $M_{T \circ T'} < M_T$ are required. See Figure~\ref{fig:Swabadab} for an illustration.
\end{proof}

\begin{remark}
	In the proof above, only the part about non-zero average movement and orbitwise shifts crucially depends on the properties of $X$. The other part does not: if $T \in \RFA(Y)$, $T$ has average movement zero (on $Y$), and $X \subset Y$, then one can perform all choices in the proof canonically so that decomposing $T$ into controlled position swaps on $Y$ and then restricting to $X$ is equivalent to first restricting to $X$ and then applying our construction on $X$.
\end{remark}

Let us use Lemma~\ref{lem:OneState} to extract corollaries in the case where we have more than one state. If $k \geq 2$, then recall that the isomorphism between $\RTM(X,k)$ and $\RTM(\sqrt[k]{X},1)$ simply uses the different positions on $\#$-segments to encode the state. Thus, moving a head one step to the right along a contiguous segment of $\#$-symbols corresponds to increasing the state by one. Moving a head from the rightmost $\#$-symbol to non-$\#$ symbol to the right of it means corresponds to changing the state to $1$ and stepping to the right.

Thus, translating the previous result amounts to the following: Let $X \subset \Sigma^\ZZ$ be a locally aperiodic subshift and $k \geq 2$. For $u, v \in \Sigma^*$ and $q \in \{1,\dots,k-1\}$ define the \define{controlled state swap} $f_{u,q,v} \in \RFA_{\text{fix}}(X,k)$ as
\[ f_{u,q,v}(x, q') = (x, q'') \]
where $q'' = q'$ if $q' \notin \{q, q+1\}$ or if $x_{\{-|u|,\dots,|v|-1\}} \neq uv$, and otherwise let $q'' \in \{q, q+1\} \setminus \{q'\}$. For $u, v \in \Sigma^*, a \in \Sigma$, define a \define{stateful controlled position swap} as $f(xu.avy, k) = (xua.vy, 1)$ and $f(xua.vy, 1) = (xy.avy, k)$ for all tails $x, y$, and $f(z) = z$ for all points not of this form. This is a well-defined involution for all $u, v$ (since $k \geq 2$).

\begin{corollary}
Let $X$ be a weakly chain-transitive locally aperiodic subshift and $k \geq 2$. Then $\RFA(X,k)$ is generated by the shift map, controlled state swaps, and stateful controlled position swaps.
\end{corollary}

\subsection{Elementary Turing machines are finitely generated}

In this section, we prove~Theorem~\ref{theorem_EL_is_FG}. That is, we show that for any $n, k \geq 1$ the group $\EL(\ZZ,n,k)$ is finitely-generated.

\begin{proof}[Proof of Theorem~\ref{theorem_EL_is_FG}]
The case $n = 1$ is easy. Suppose thus that $n \geq 2$. Let $X = \Sigma^\ZZ$. Since $\OB(X)$ is finitely generated and $\OS(X) = \langle \sigma \rangle$, we only need to show that the controlled position swaps can be implemented when $k = 1$, and that the stateful controlled position swaps and controlled state swaps can be implemented when $k \geq 2$. We work in the fixed-head model.

Suppose first that $k = 1$. 
Let $s_\ell = T_{\epsilon, 1, 0^{\ell}}$. 
By conjugating with an element of $\OB(X)$, these swaps generate all controlled swaps where the left control is $\epsilon$. Namely, to generate $T_{\epsilon, a, v}$ where $av$ is not unary (see Lemma~\ref{lem:Unary}), conjugate $s_\ell$ where $\ell = |v|$ with any local permutation mapping
\[ f(x . av y) = x . 10^\ell y \mbox{ and } f(x a . v y) = x 1 . 0^\ell y. \]
Such $f$ exists because $av$ is not unary and thus does not overlap with its $1$-shift -- simply map these patterns correctly, and fill the permutation arbitrarily. 
Swaps at other positions are obtained by conjugating with $\sigma$. Thus, we only need to show that $s_{\ell}$ can be decomposed into elements of $\OB(X)$, shifts, and finitely many controlled swaps. We do this by induction on $\ell$.

We begin with a construction in the case when $n$ is even. We begin with the intuitive description. Since whether the head is on the coordinate $0$ or $1$ in the word $10^\ell$ can be checked locally, we can consider the position to be just a binary digit $c$ (so $c = 0$ means we are on the zero-coordinate of a word $10$, $c = 1$ means we are on the one-coordinate of such pattern, and otherwise we fix the point anyway). Note that using elements of $\OB(X)$, we can swap positions to the right of the pattern $10$ independently of $c$. Thus, our controlled swaps may depend on positions arbitrarily far to the right from the central $10$, as long as they do not depend on more than a constant number of coordinates.

We begin with state $cwb$ where $c$ is the head position (either $0$ or $1$, marking whether we are on the first or second coordinate of $w$), and $w = 10u$ where $u \in \Sigma^{\ell-1}$ is the control word that we are checking for value $0^{\ell-1}$. The permutation is the following: First, flip $c$ if $b = 0$ (by moving $b$ next to the initial $10$ of $w$ performing a small-radius controlled swap, and moving $b$ back). Then add $1$ (modulo $n$) to $b$ (by an element of $\OB(X)$) if $u = 0^{\ell-1}$. Then repeat these two steps $n$ times. The end result is that $c$ is flipped once if $u = 0^{\ell-1}$, since $b = 0$ at exactly one of the $c$-flipping steps, and otherwise $c$ is flipped either $0$ or $n$ times and thus is not changed.

More precisely, define first $g \in \RFA(X)$ as the controlled swap $T_{\epsilon,1,00}$, that is,
\[ g(x1.00y) = x.100y,\;\; g(x.100y) = x1.00y \]
and $g(z) = z$ for points not of this form. This is clearly in $\RFA(X)$, and we take it in our set of generators. Also define $h \in \OB(X)$ as any local permutation such that
\[ h(x1.0uby) = x1.0buy,\;\; h(x.10uby) = x.10buy \]
for all $u \in \Sigma^{\ell-1}, b \in \Sigma$ (and $h(z)$ may be arbitrary for points not of this form),
which is again well-defined since no point is of both forms.

Now, defining $f' = h^{-1} \circ g \circ h$ we have
\[ f'(x1.0u0y) = x.10u0y,\;\; f'(x.10u0y) = x1.0u0y \]
for all $u \in \Sigma^{\ell-1}$ and $f'(z) = z$ for points not of this form.

Next, define $f'' \in \OB(X)$ by
\[ f''(x.10^\ell by) = x.10^\ell(b+1)y,\;\; f''(x1.0^\ell by) = x1.0^\ell(b+1)y \]
for all $b \in \{0,1,\ldots,n-1\}$ (where increment is modulo $n$) and $f''(z) = z$ for all $z \in X$ not of this form.

Then $(f'' \circ f')^n$ is our desired map. 

Suppose then that $n$ is odd. We will use the proof structure of \cite{Se16}. Again, consider the input to be $c10ub$ where $c \in \{0,1\}$ indicates the position of the head (either on the symbol $1$ or the symbol $0$ of the word $10$ after $c$), and $u \in \Sigma^{\ell-1}$ and $b \in \Sigma$. We construct the machine by induction. Suppose we have already constructed a machine that flips $c$ conditioned on $u = 0^{\ell-1}$, and let us show how to check $b = 0$. First flip $b$ between $0$ and $1$ if $u = 0^{\ell-1}$ (and fix the value of $b$ if $b \in \Sigma \setminus \{0,1\}$), using a machine in $\OB(X)$. Then flip the bit $c$ if $b = 0$ using a machine in $\RFA(X)$ conjugated by a machine in $\OB(X)$, as above. Repeat these steps. The resulting machine $g^2$ is in $\EL(X)$ and will flip $c$ if and only if $b \in \{0,1\}$ and $u \in 0^{\ell-1}$.

Now, conjugate the machine by affine translations: let $h \in \OB(X)$ add $2$ to the value of $b$. Then $h^{-\frac{n-1}{2}} \circ (g^2 \circ h)^{\frac{n-1}{2}}$ will flip the value of $c$ if and only if $u = 0^{\ell-1}$ and $b \neq 0$. Finally, flip $c$ if $u = 0^{\ell-1}$, so that the resulting machine flips $c$ if and only if $ub = 0^{\ell}$.

The details of translation to machines in $\RFA(X)$ and $\OB(X)$ are omitted, as they are similar as in the case of an even alphabet.

In the case $k \geq 2$, we need to perform controlled state swaps and stateful controlled position swaps using finitely many elements of $\RFA(X)$ and elements of $\OB(X)$. The proof of controlled state swaps is the same as the proof above, but now the bit $c$ (which indicated the position above) indicates which state we are in, out of the two we are swapping (and nothing happens if we are in neither state). The proof of stateful controlled position swaps is also similar; and now the bit $c$ is $0$ if we are in state $k$, and $1$ if we are in state $k$, and in state $1$ the control word is seen through an offset of one.\end{proof}


\section{Computability aspects}
\label{section.computability}
In this section we study the computability aspects of $\RTM(\ZZ^d,n,k)$. We begin the section by briefly showing which properties of Turing machines are computable. In particular we prove that injectivity and thus the reversibility of a machine in $\TM(\ZZd,n,k)$ is decidable. This property, along with the possibility to compute the rule of a composition and an inverse gives a recursive presentation for $\RTM(\ZZd,n,k)$ which has a decidable word problem.

In this context, we proceed to study the torsion problem of $\RTM(\ZZd,n,k)$ and its subgroups, that is, whether there exists an algorithm which given a description of $T \in \RTM(\ZZd,n,k)$ always halts and accepts if and only if there exists $n \geq 1$ such that $T^n = \ID$. In this context we show that $\EL(\ZZ,n,k)$ has an undecidable torsion problem (Theorem~\ref{thm:decidability_torsion_elementaryTM}). Furthermore, we use this result to show that the automorphism group of any uncountable $\ZZ$-subshift contains a finitely generated subgroup with undecidable torsion problem (Corollary~\ref{cor:autgroup_undecidabletorsion}).

Finally, we study the torsion problem for $\FG(\ZZd,n,k)$. We show by a simple argument that $\FG(\ZZ,n,k)$ has a decidable torsion problem. Interestingly, the torsion problem in $\FG(\ZZd,n,k)$ for $d \geq 2$ is undecidable. We present a detailed proof of this result which draws upon the undecidability of the snake tiling problem~\cite{Ka03}. These two results add up to the Theorem~\ref{thm:decidability_torsion_FSA} we discussed in the introduction.

\subsection{Basic decidability results}

First, we observe that basic management of local rules is decidable.

\begin{restatable}{lemma}{Decisions}
\label{lemma_computability_obvious_properties}
Given two local rules $f, g$ in the moving head model,
\begin{itemize}
	\item It is decidable whether $T_f = T_g$,
	\item We can effectively compute a local rule for $T_f \circ T_g$,
	\item It is decidable whether $T_f$ is injective,
	\item It is decidable whether $T_f$ is reversible, and
	\item We can effectively compute a local rule for $T_f^{-1}$ when $T_f$ is reversible.
\end{itemize}
\end{restatable}

\begin{proof}
For the first claim, let $f\colon \Sigma^{F_1} \times Q \to \Sigma^{F_2} \times Q \times \ZZd$ and $g\colon \Sigma^{F_3} \times Q \to \Sigma^{F_4} \times Q \times \ZZd$ and define $F = F_1 \cup F_2 \cup F_3 \cup F_4$. Extend the rules $f,g$ so that they are defined as $f',g'\colon \Sigma^F \times Q \to \Sigma^F \times Q \times \ZZd$ satisfying $T_f = T_{f'}$ and $T_g = T_{g'}$. If for some $(p,q)\in \Sigma^F \times Q$ $f'(p,q) \neq g'(p,q)$, then clearly $T_{f '} \neq T_{g'}$. Otherwise, we have $T_{f'} = T_{g'}$.

Finding a local rule for the composition of two Turing machines is a straightforward, if somewhat tedious, exercise. 

For the decidability of reversibility, we give a semi-algorithm for both directions. First, if $T_f$ is reversible, then it has a reverse $T_g$. We thus only need to enumerate local rules $g$, and check whether $T_f \circ T_g = \ID$, which is decidable by using the two previously described procedures.

If $T_f$ is not injective, then $T_f(x,y) = T_f(x',y')$ for some $(x,y), (x',y') \in \Sigma^{\ZZd} \times X_k$ with $y_{\vec 0} \neq 0$. If $r$ is the move-radius of $f$, then necessarily the nonzero position of $y'$ is at distance at most $r$ from the origin and $x_{\vec v} = x'_{\vec v}$ for $|\vec v|$ larger than the radius of $T_f$. Then we can assume $x_{\vec v} = x'_{\vec v} = 0$ for all such $\vec v$. It follows that if $T_f$ is not injective, it is not injective on the finite set of configurations $(x,y) \in \Sigma^{\ZZd} \times X_k$ where $(x,y)_{\vec v} = 0$ for all $|\vec v|$ larger than the radius of $T_f$, which we can check algorithmically.

By Theorem~\ref{theorem_injective_surjective_reversible}, injectivity is equal to reversibility, so the decidability of the reversibility of a Turing machine is a direct consequence of the previous argument.

Finally, if $T_f$ is reversible, we can effectively construct a reverse ${T_f}^{-1}$ for it by enumerating all Turing machines and outputting the first $T'$ such that $T_f \circ T' = T' \circ T_f = \ID$. \end{proof}	

From the results above, we obtain that the set of all possible local rules $f$ which generate reversible actions $T_f$ gives a recursive countable presentation of $\RTM(\ZZd,n,k)$. Furthermore, that presentation has a decidable word problem.

\begin{proposition}
$\RTM(\ZZd,n,k)$ admits a recursive presentation with decidable word problem.
\end{proposition}

\begin{proof}
The set of local rules giving moving head Turing machines can be recursively enumerated. Indeed, consider the sequence of sets $\Lambda_n = \{-n,\dots,n\}^d$ indexed by $n \in \NN$ and list in some lexicographical order all local rules $f\colon \Sigma^{\Lambda_n} \times Q \to \Sigma^{\Lambda_n} \times Q \times \Lambda_n$. For each $n \in \NN$ this is a finite set and thus all local rules can be recursively listed as $(f_{i})_{i \in \NN}$. By Lemma~\ref{lemma_computability_obvious_properties} it is decidable which local rules define reversible Turing machines, and thus one can run that algorithm on each $f_i$ to obtain a recursive enumeration $(f_{\varphi(i)})_{i \in \NN}$ of all reversible moving head Turing machines. Also, using Lemma~\ref{lemma_computability_obvious_properties} one can reduce every word $f_{\varphi(i_1)}f_{\varphi(i_2)}\cdots f_{\varphi(i_k)}$ to some equivalent rule $f_{\text{eq}}$. It suffices to test the equality of $f_{\text{eq}}$ with the identity machine to decide the word problem of this presentation. 
\end{proof}

\subsection{The torsion problem of elementary Turing machines}

\begin{definition}
	Let $G$ be a group which is generated by $S \subset G$. The \define{torsion problem} of $G$ is the set of words $w \in S^*$ for which there is $n \in \ZZ_+$ such that the element represented by $w^n$ is the identity of $G$.
\end{definition}

If a group $G$ is recursively presented, then the torsion problem is recursively enumerable. However, the torsion problem may not be decidable even when $G$ has decidable word problem. Many such examples are known, and the main result of this section provides a new such example.

As discussed in the introduction, we say a moving head Turing machine is \define{classical} if its in- and out-radii are $0$, and its move-radius is $1$. Here we characterize reversibility in classical Turing machines. If $T_0$ has in-, out- and move-radius $0$, that is, $T_0$ only performs a permutation of the set of pairs $(s,q) \in \Sigma \times Q$ at the position of the head, then we say $T_0 \in \LP(\ZZ,n,k)$ is a \define{state-symbol permutation}. If $T_1$ has in-radius $-1$, never modifies the tape, and only makes movements by vectors in $\{-1,0,1\}$, then $T_1 \in \FG(\ZZ,n,k)$ is called a \define{state-dependent shift}.

If we consider the class of all classical Turing machines on some finite alphabet and number of states then the torsion problem is undecidable. This result was shown by Kari and Ollinger in~\cite{KaOl08} --they call it the periodicity problem in their setting-- using a reduction from the mortality problem which in turn they also prove to be undecidable following a reduction from the mortality problem of reversible 2-counter machines. In this section, we show that the torsion problem is also undecidable for elementary Turing machines with a fixed alphabet and number of states. We do this by reducing to the classical machines. We begin by describing them in our setting.

\begin{restatable}{proposition}{ClassicalReversibility}
\label{prop:ClassicalReversibility}
A classical Turing machine $T \in \RTM(\ZZ,n,k)$ is reversible if and only if it can be expressed in the form $T_1 \circ T_0$ where $T_0$ is a state-symbol permutation and $T_1$ is a state-dependent shift. In particular, classical reversible Turing machines are in $\EL(\ZZ,n,k)$.
\end{restatable}

\begin{proof}
We only need to show that if $T$ is reversible then it is of the stated form. Let $f_T \colon \Sigma \times Q \to \Sigma \times Q \times \{-1,0,1\}$ be a local rule for $T$ in the moving tape model. We claim that if $f_T(a,q) = (b,r,d)$ and $f_T(a',q') = (b',r,d')$ then $d = d'$. Namely, if it is not the case, one can easily find two configurations with the same image. There are multiple cases to consider which can all be treated similarly, for instance, let $x,y$ be a left and right infinite sequence of symbols of $\Sigma$ respectively and suppose $d = 0$ and $d' = 1$. In this case:
\[ T((xb'.ay), q) = (xb'.by, r) = T((x.a'by), q'). \]

Therefore contradicting the reversibility of $T$. Repeating this argument over all pairs $d \neq d'$ we obtain that the direction of movement is entirely determined by the output state. Of course, for $T$ to be injective, also $f_T$ must be injective, so the map $g \colon \Sigma \times Q \to \Sigma \times Q$ defined by $g(a,q) = (b,r)$ if $f_T(a,q) = (b,r,d)$ is injective, thus a bijection. From here, the only remaining possibility is that $T_0$ is the permutation $g$, and $T_1$ as the finite-state automaton with local rule $f_{T_1}(a,q) = (a,q,d)$ if $f_T(b,q') = (b',q,d)$ for some $(b, q) \in \Sigma \times Q$. 
\end{proof}

It follows that the inverse of a reversible classical Turing machine is always of the form $T_0 \circ T_1$ where $T_0$ is a state-symbol permutation and $T_1$ is a state-dependent shift. In the terminology of Section~\ref{section.subgroups}, the theorem implies that all reversible classical Turing machines are elementary.

%

Now we are ready to prove Theorem~\ref{thm:decidability_torsion_elementaryTM}, namely, that the torsion problem of $\EL(\ZZd,n,k)$ is undecidable for all $n \geq 2$, $k\geq 1$ and $d\geq 1$.

\begin{proof}[Proof of Theorem~\ref{thm:decidability_torsion_elementaryTM}]
As $\EL(\ZZ,n,1)$ embeds into $\EL(\ZZd,n,k)$ for every $k,d \geq 1$, It is enough to prove this for $d=1$ and $k=1$. 

The machines constructed in~\cite{KaOl08} are already elementary, but their alphabets are not
bounded (note that these are classical Turing machines so there are only finitely many machines
of any given state-alphabet pair). However, one can simulate them with elementary Turing machines with fixed alphabets.
That is, given a classical reversible Turing machine $T$ (on some finite alphabet and number of states), one can construct a machine in $\EL(\ZZ,n,1)$ which
is periodic if and only if $T$ is. Our result then follows from the main result of~\cite{KaOl08}.

Let $T$ be a given classical Turing machine with state set $Q$ and tape alphabet $\Sigma$, and let $f_T$ be its local rule. Note that in this proof we use $\Sigma$ and $Q$ for the alphabet and state set of $T$, rather than those of the machines in $\EL(\ZZ,n,1)$.

By Proposition~\ref{prop:ClassicalReversibility} we have $T=T_1 \circ T_0$ where $T_0$ is a state-symbol permutation and $T_1$ is a state-dependent shift. We can further assume that $T_1$ does not
use movement by $0$: For each state $q$ that would move by $0$ we introduce a new state $q'$, and modify shift $T_1$ to move by $-1$ in state $q$ and by $+1$ in state $q'$. We also modify
state-symbol permutation $T_0$ to enter always state instead $q'$, whenever it would enter $q$, and to map $(a,q')\mapsto (a,q)$  for all tape symbols $a$. In this way a single non-moving step gets replaced by two
steps moving to the right and to the left, respectively. Clearly the modified machine is periodic if and only if the original one is.
Let us denote by $L\subset Q$ and $R\subset Q$ the sets of left and right moving states in $T_1$.
Then $Q$ is the disjoint union of $L$ and $R$

Let $\perp$ be a symbol which is not in $Q$ and $m$ be such that $n^m\geq |\Sigma\times (Q\cup\{\perp\})|$ so that we can encode tape symbols (represented by pairs in $\Sigma \times \{\perp\})$ and state-symbol pairs as unique blocks of length $m$ over the $n$ letter alphabet. Pair
$(a,q)$ represents a tape cell containing $a\in\Sigma$ and being read by the machine in state $q\in Q$, while $(a,\perp)$ is a cell containing $a\in\Sigma$ and not currently read by the
machine.
Take an arbitrary one-to-one encoding function
$\Sigma\times (Q\cup\{\perp\}) \to \{1,\dots, n\}^m$. Let us denote by $[a,r]$ the encoding of $(a,r)\in (Q\cup\{\perp\})$, a word of length $m$ in the alphabet $\{1,\dots, n\}$.

The idea is that a configuration of $T$ with tape content
$x\in \Sigma^\ZZ$, state $q\in Q$ and position $p\in \ZZ$ is represented as the encoded configuration with tape content $x'\in\{1\dots, n\}^\ZZ$, where $x'_{[im,(i+1)m)}=[x_i,\perp]$ for
$i\neq p$ and $x'_{[pm,(p+1)m)}=[x_p,q]$, and
with the stateless head in position $pm$. Notice that there might be many configurations in $\{1,\dots, n\}^\ZZ$ that are not valid encodings.

We construct an element of $\EL(\ZZ,n,1)$ that is a composition $T_1' \circ T_0'$ where  $T_0'\in \LP(\ZZ,n,1)$ and $T_1'\in \RFA(\ZZ,n,1)$ are both involutions. Informally, the local permutation $T_0'$
implements the state-symbol permutation $T_0$ on the encoded tape, including writing the new state in the correct position, while the finite-state automaton $T_1'$ simply moves the machine head $m$
positions to the right or to the left to scan the block containing the new state. We complete the local rules of $T_0'$ and $T_1'$ into involutions that also perform the inverse steps, and act
as identities in all other situations. This automatically causes the machine to reverse its time direction and start retracing it steps backwards
if it encounters a non-valid pattern.

More precisely, the local permutation $T_0'$ is defined with neighborhood $F_0=\{-m,\dots ,2m-1\}$ and so that the local rule $f_0$ is a permutation of $\{1,\dots, n\}^{F_0}$. The machine sees three $m$-blocks: the
current block in the center and neighboring blocks on both sides.
For every $a,b,c,d\in\Sigma$ and $q,p\in Q$,

\begin{enumerate}
	\item if $f_T(a,q)=(b,p,-1)$ then $f_0$ swaps
	\[[c,\perp] [a,q] [d,\perp] \longleftrightarrow [c,p] [b,\perp] [d,\perp],\]
	\item if $f_T(a,q)=(b,p,+1)$ then $f_0$ swaps
	\[[c,\perp] [a,q] [d,\perp] \longleftrightarrow [c,\perp] [b,\perp] [d,p].\]
\end{enumerate}

All other patterns in $\{1\dots, n\}^{F_0}$ are mapped to themselves. By the reversibility of $T$ this $f_0$ is a well-defined involution.

The reversible single-state finite automaton $T_1'$ uses the neighborhood $F_1=\{-2m,\dots ,3m-1\}$, so the machine sees two blocks to the left and to the right from its current block.  Its
local rule is a function
$f_1\colon \{1,\dots, n\}^{F_1}\longrightarrow\ZZ$ where $f_1(u)$ gives the head movement on tape pattern $u$.
\begin{enumerate}
	\item Let $p\in L$. Then  $f_1$ maps, for all $b,c,d\in \Sigma$,
	\[\begin{array}{ccccccccl}
	* & [c,p] & [b,\perp] & [d,\perp] & * & \ & \mapsto & \ & -m,\\
	* & * & [c,p] & [b,\perp] & [d,\perp] & & \mapsto & & +m,
	\end{array}\]
	where $*$ represents any $m$-block in $\{1\dots, n\}^m$. In other words, if the machine sees a state $p\in L$ on the block to its left and non-states $\perp$ at its current block and the
	block on its right, the machine moves left. If it sees a state $p\in L$ at its current block and $\perp$ on the  two blocks to its right, it moves right. These moves are inverses of each
	other: if one is applicable now, the other one is applicable on the next time step and it makes the machine return to its original position.
	\item Symmetrically, let $p\in R$. Then $f_1$ maps, for all $c,b,d\in \Sigma$,
	$$
	\begin{array}{ccccccccl}
	* & [c,\perp] & [b,\perp] & [d,p] & * & \  &\mapsto & \ & +m,\\
	\lbrack c,\perp] & [b,\perp] & [d,p] & * & * &  &\mapsto & & -m.
	\end{array}
	$$
	Also these moves are inverses of each other.
\end{enumerate}
On all other patterns in $\{1\dots, n\}^{F_1}$ the movement is by $0$. Because $L$ and $R$ are disjoint, $T_1'$ is well-defined. It is also an involution.

Machines $T_0'$ and $T_1'$  are so designed that $T_1' \circ T_0'$ simulates one step of $T$ on valid encodings of its configurations. It follows that if $T$ is not periodic neither is $T_1'
\circ T_0'$.

Suppose then that $T_1' \circ T_0'$ is not periodic. When iterating $T_0'$ and $T_1'$ alternatingly, if at any moment either
machine does not change the configuration then the iteration starts to retrace its steps back in time. This follows from the fact that $T_0'$ and $T_1'$ are involutions. Any orbit containing
two such changes of time direction is periodic. By this and compactness, if $T_1' \circ T_0'$ is not periodic, there exists a bi-infinite orbit where each application of $T_0'$ and $T_1'$
changes the configuration. By the construction, the machine locally only sees valid encodings. The orbit -- in either forward or backward time direction -- is then a valid simulation of $T$ on
the portion of the tape that $T$ sees. It follows that $T$ has a non-periodic orbit as well.        
\end{proof}

Combining Theorem~\ref{thm:decidability_torsion_elementaryTM} and Theorem~\ref{theorem_EL_is_FG}, we immediately obtain the following corollary.

\begin{corollary}
There is a finitely generated subgroup of $\RTM(\ZZ,n,k)$ whose torsion problem is undecidable.
\end{corollary}

\subsection{The torsion problem of cellular automata}
\label{sec:CATorsion}

In this section, we show that there is a finitely generated group of cellular automata whose torsion problem is undecidable. It is tempting to try to prove this by constructing an embedding from $\EL(\ZZ,n,k)$ to $\Aut(A^\ZZ)$, but this is impossible: it is well-known that $\Aut(A^\ZZ)$ is residually finite, however, the group of finitely supported permutations of $\NN$ embeds into $\LP(\ZZ,n,k)$ and therefore $\EL(\ZZ,n,k)$ can not be residually finite. Nevertheless, there are standard ways to construct cellular automata from Turing machines, and we show that while one cannot preserve the complete group structure, one can preserve torsion. 

One could do this similarly as in the previous section, by showing directly that an automorphism simulating a Turing machine will be periodic if and only if the Turing machine is. We will prove a slightly stronger abstract result, namely we construct a map from $\EL(\ZZ,n,k)$ to $\Aut(A^\ZZ)$ that preserves some of the group structure of $\EL(\ZZ,n,k)$ while adding only some ``local'' identities. To achieve this, it seems we cannot use the same construction as in the previous section, as discussed below in Remark~\ref{rem:NotTheOnly} discusses.

In the following, the free monoid generated by the elements of a group $G$ is written $G^*$, and consists of formal words $w$ where $w_i \in G$ for all $i = 1, \ldots, |w|$. For $w \in G^*$, write $ \underline{w} = w_1 \cdot w_2 \cdots w_{|w|} \in G$ for the corresponding element of $G$.

\begin{definition}
\label{def:PPreserving}
Let $G$ and $H$ be groups and $\mathcal{P}$ be a group property. We say a function $\phi \colon G \to H$ is $\mathcal{P}$-preserving if the following holds: For every finite set $F \subset G^*$ the group $\langle \underline{w} \;|\; w \in F \rangle \leq G$ has property $\mathcal{P}$ if and only if the group $\langle \underline{\phi(w_1) \phi(w_2) \cdots \phi(w_{|w|})} \;|\; w \in F \rangle$ has property $\mathcal{P}$
\end{definition}

We remark that $\mathcal{P}$-preserving functions $\phi$ need not be morphisms as we do not ask that $\phi(w_1w_2) = \phi(w_1)\phi(w_2)$. We only demand that property $\mathcal{P}$ is preserved when applying $\phi$ to the symbols appearing in the words in $F$.

In what follows we are going to use the property $\mathcal{P}$ of being finite. We use this property to extend computability invariants such as the torsion problem of a group onto another group even if no embedding from the first group to the second exists. This kind of extension obviously demands that the function $\phi$ is in some way computable. We will say a function $\phi \colon G \to H$ is \define{computable} if both $G$ and $H$ have decidable word problem for some fixed presentation and there is a Turing machine which turns any word $w$ in the presentation of $G$ such that $\underline{w}=g \in G$ into a word $u$ in the presentation of $H$ such that $\underline{u}=\phi(g)\in H$.

\begin{lemma}\label{lemma_blurphism}
Let $G$ be a finitely generated group with undecidable torsion problem and generating set $S$, and suppose there exists a computable finiteness-preserving function $\phi \colon G \to H$. Then the subgroup $H' =  \langle \{\phi(s)\}_{s \in S} \rangle \leq H$ has undecidable torsion problem.
\end{lemma}

\begin{proof}
Suppose the group $H'$ generated by the $\phi(s)$ has decidable torsion problem and let $w \in S^*$. As $\phi$ is finiteness-preserving we have that $\langle \bar w \rangle$ is finite if and only if $\langle \phi(w_1) \phi(w_2) \cdots \phi(w_{|w|}) \rangle$ is finite. This means $\underline{w}$ has finite order in $G$ if and only if $\phi(w_1) \phi(w_2) \cdots \phi(w_{|w|})$ has finite order in $H'$. We can compute $\phi(w_1) \phi(w_2) \cdots \phi(w_{|w|})$ from $w$ and from the finite set $\{\phi(s) | s \in S\} \subset H$. Thus, the algorithm to decide the torsion problem in $H'$ can be used to decide the torsion problem in $G$, raising a contradiction.  
\end{proof}

The previous lemma indicates that in order to prove that $\Aut(A^{\ZZ})$ contains a finitely generated subgroup with undecidable torsion problem, we now simply need to provide a computable finiteness-preserving map from $\EL(\ZZ,n,k)$ into $\Aut(A^\ZZ)$.

\begin{lemma}
Let $ A = \Sigma^2 \times (\{\leftarrow, \rightarrow\} \cup Q \times \{\uparrow,\downarrow\})$.
There is a computable finiteness-preserving map $\phi \colon  \RTM(\ZZ,n,k) \to \Aut(A^\ZZ)$.
\end{lemma}

\begin{proof}
The alphabet $A$ consists of triples and thus $A^{\ZZ}$ can be thought of as consisting of three tapes. The two first tapes carry a configuration in $\Sigma^{\ZZ}$ while the third tape has symbols in $(\{\leftarrow, \rightarrow\} \cup (Q \times \{\uparrow,\downarrow\}))$ and deals with the heads and calculation zones. A head is represented by a tuple in $Q \times \{\uparrow,\downarrow\}$, where $Q$ is the set of states of the Turing machine and $\{\uparrow,\downarrow\}$ is the track the machine is in. If this value is $\uparrow$, the head is on the first (``topmost'') tape and if it is $\downarrow$, on the second (``bottom'') tape. $\leftarrow$ means the head is to the left of the current cell on the current zone (if the current zone contains a head), while $\rightarrow$ means the head is to the right.

A configuration in $A^\ZZ$ is split into zones by the contents of the third tape. Namely, every finite portion of the second track can be split uniquely into pieces of the forms $\rightarrow^* (q, a) \leftarrow^*$ and $\rightarrow^* \leftarrow^*$ where $(q,a) \in Q \times \{\uparrow,\downarrow\}$. We call these pieces \define{zones}, see Figure~\ref{conveyor_zones_pic}. To define the finiteness-preserving map $\phi \colon \RTM(\ZZ,n,k) \to A^\ZZ$ it is enough to do so in every piece of this form.

\begin{figure}[h!]
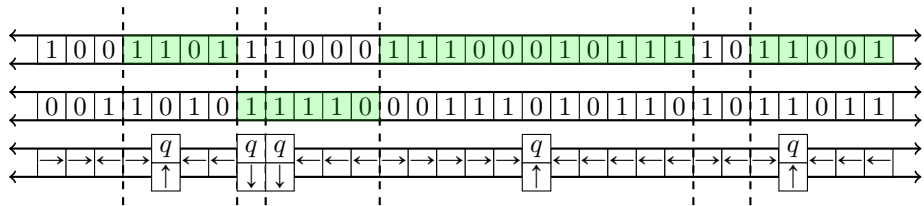

	\centering
	\include{conveyor_zones}
	\caption{A finite word in $A^*$ is divided into zones by the third layer. The dashed lines separate each zone and the colours indicate which tape is being pointed at by the arrow next to the state.}
	\label{conveyor_zones_pic}
\end{figure}

Let $T \in \RTM(\ZZ,n,k)$ be a moving head Turing machine of radius $r$. We define $\phi(T) \in \Aut(A^{\ZZ})$ by defining its action over each zone as follows: If the zone has no head or the zone is of size less than $2r+1$, do nothing. Otherwise let $u_0,\dots,u_{m-1}$ and $v_0,\dots,v_{m-1}$ be the words in the first and second track respectively, $(q,a) \in Q \times \{\uparrow,\downarrow\}$ be the head and $\ell \in \{0,\dots,m-1 \}$ the position of the head in the third track. Using this information we can construct the configuration $x \in \Sigma^{\ZZ}$ given by:

$$x_i = \begin{cases}  u_{j} & \mbox{ if    } j = (i\mod{2m}) \in \{0,\dots,m-1\} \\ v_{2m-j-1} & \mbox{ if    } j = (i\mod{2m})\in \{m,\dots,2m-1\} \end{cases}$$

\begin{figure}[h!]
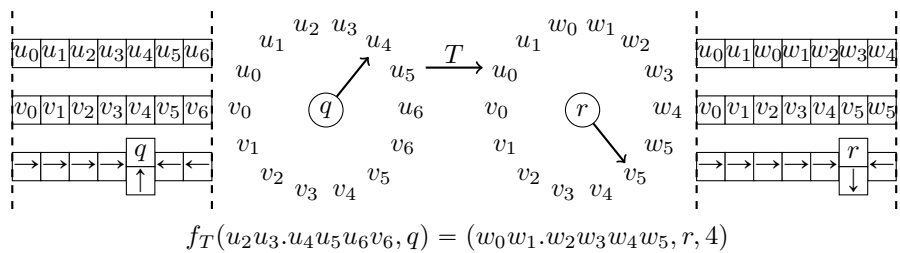

	\centering
	\include{conveyor_belt}
	\caption{Every zone is wrapped around as a conveyor belt, where $\phi(T)$ acts as if it were $T$ seeing a periodic word.}
	\label{conveyor_belt_pic}
\end{figure}

And apply $T$ to $x$, where the position of the head is on $\ell$ if $a = {\uparrow}$ and on $2m-\ell-1$ otherwise. After applying the Turing machine's local rule, recode the result again updating the left and right arrows so that the zone does not change its shape as shown in Figure~\ref{conveyor_belt_pic}. With this definition, in each bounded zone $\phi(T)$ induces a permutation of all possible heads and tape contents while on unbounded zones it acts as $T$ on an infinite configuration.

The map $\phi(T)$ is clearly continuous and shift invariant and therefore $\phi(T) \in \Aut(A^{\ZZ})$. Also, $\phi$ is clearly a computable map. It suffices to show that $\phi$ preserves the property of being finite.

Consider $F \subset \RTM(\ZZ,n,k)^*$, and $H = \langle \phi(w_1) \phi(w_2) \cdots \phi(w_{|w|}) \;|\; w \in F \rangle$ generated in $\Aut(A^\ZZ)$. If $\langle F \rangle$ is infinite, just note that the action of $\phi(T)$ over any configuration such that the third tape is a single unbounded zone with a head, (that is $\ldots \leftarrow\leftarrow\leftarrow (q,\uparrow) \rightarrow\rightarrow\rightarrow\ldots$) replicates exactly the behavior of $T$ on the first tape. Therefore each element of $\langle F \rangle$ will act differently over at least one configuration of this form, implying that $H$ is infinite. Conversely, if $\langle F \rangle$ is finite, 
consider the action over any zone which is unbounded or larger than the maximum movement from the origin attained by a machine in $\langle F \rangle$. Obviously such a machine acts as the original machine, so the action on such zones satisfies the same relations as the action of $\langle F \rangle$. Consider then the zones whose length is bounded by a fixed length $h$. The number of different machine actions on these zones is bounded as each action is a permutation over a finite set, thus again the movement of the Turing machine head is bounded. Therefore $H$ is finite and a rough bound is $|H| \leq |\langle F \rangle | \cdot (\prod_{m \leq h} (2kmn^{2m})!)$ 
\end{proof}

\begin{remark}
\label{rem:NotTheOnly}
The above is not the only possible construction for making cellular automata out of Turing machines. For example in \cite{KaOl08}, one instead uses a direction bit, and flips the running direction of the Turing machine (from forward to backward) if it hits a border of a computation zone. This roughly corresponds to the construction in the previous section. This construction does not, at least without some modifications, give a finiteness-preserving map in the sense of the definition above: Suppose $f$ and $g$ are Turing machines satisfying no relations. Pick $F = \{ (f^{-1}g^{-1}) \cdot g \cdot f \}$ (a formal product of length 3 in $G^*$) in Definition~\ref{def:PPreserving}. Since $\ID = (f^{-1}g^{-1}) \cdot g \cdot f$ is torsion as a Turing machine, we should have that the corresponding cellular automaton $\phi(f^{-1}g^{-1}) \circ \phi(g) \circ \phi(f)$ is. But on a configuration where the direction bit points backwards in time, this cellular automata in fact simulates the Turing machine $g \circ f \circ g^{-1} \circ f^{-1}$ which is of infinite order.
\end{remark}	

Using that $\EL(\ZZd,n,k)$ is finitely generated and considering the restriction of the map of Lemma~\ref{lemma_blurphism} to $\EL(\ZZd,n,k)$ we obtain that:

\begin{corollary}\label{cor_somealphabet}
For some alphabet $A$ with at least two symbols, there is a finitely generated subgroup $G \leq \Aut(A^\ZZ)$ with undecidable torsion problem.
\end{corollary}

In Lemma 7 of~\cite{Sa15}, it is shown that for every finite set $A$, and every uncountable sofic subshift $X$, we have that $\End(A^{\ZZ})$ embeds into $\End(X)$. In particular $\Aut(A^{\ZZ})$ embeds into $\Aut(X)$. Letting $A$ be the alphabet of Corollary~\ref{cor_somealphabet} we obtain Corollary~\ref{cor:autgroup_undecidabletorsion}, that is, that the automorphism group of every uncountable sofic $\ZZ$-subshift contains a finitely generated subgroup $G$ with undecidable torsion problem.

\subsection{The torsion problem of finite-state machines}

The torsion problem of $\OB(\ZZ^d,n,k)$ is not of much interest, as an element of this group is periodic if and only if its shift component is $0^d$. The group $\FG(\ZZd,n,k)$ for $n \geq 2$, however, is quite interesting. Namely, the decidability of the torsion problem turns out to be dimension-sensitive as stated in Theorem~\ref{thm:decidability_torsion_FSA}. In this section we prove that theorem in two parts, namely, the one-dimensional part is a corollary of Theorem~\ref{thm:decidability_torsion1d_FSA_inchapter} whereas the multidimensional result is given in Theorem~\ref{thm:decidability_torsion_FSA_inchapter}. 

We first show that for $d =1$ the torsion problem is decidable. In fact, we obtain this from a more general result. We say the finiteness problem of a group presentation is decidable if there exists a Turing machine which on input a finite set of words in the presentation accepts if and only if the elements of the group represented by those words span a finite subgroup. This generalizes the torsion problem which can be regarded as the special case where the set is a singleton.

The proof we presented in the appendix of the conference version \cite{BaKaSa16} contained errors, here we present a new proof which is based on a strong form of the pigeonhole principle, namely the Ramsey theorem. A dynamical proof of the same result is given in \cite{Sa16}.

\begin{restatable}{theorem}{OneDTorsion}\label{thm:decidability_torsion1d_FSA_inchapter}
Let $X$ be a sofic $\ZZ$-subshift. Then the finiteness problem of $\FG(X,n,k)$ is decidable for every $n,k \geq 1$.
\end{restatable}

\begin{proof}
Let $F$ be a finite set of Turing machines represented by their local rules and $n \in \NN$. By Lemma~\ref{lemma_computability_obvious_properties} given $u,v \in F^*$ we can compute the relation $\sim$ where $u \sim v$ if and only if $\underline{u} = \underline{v}$. A semialgorithm which accepts if and only if $\langle F \rangle$ is finite computes the set $U_n = (\bigcup_{m \leq n} F^n)/\sim$ and accepts if $U_n = U_{n+1}$ for some $n$. In other words, the finiteness problem is recursively enumerable, it suffices thus to show it is co-recursively enumerable.

For the other direction, we may assume that the movement of the machines in $F$ only depends on the current symbol under them, and the machines move by at most one cell on each step. To see this, take a higher power presentation of $X$ by a sufficiently large power $p$ (see Section 1.4 of~\cite{LiMa95}), and record the position in a cell modulo $p$ in the state, so that the size of $Q$ changes to $kp$.

Now, to each word $u \sqsubset X$ associate
\[ \phi(u) \colon  \{\leftarrow, \rightarrow\} \times Q \to 2^{\{\leftarrow, \rightarrow\} \times Q}, \]
where the interpretation of $(d',q') \in \phi(u)(q,d)$ is that, when entering the word $u$ in state $q$ by moving in direction $d$, it is possible to apply a finite sequence of finite-state machines in $F$ so that eventually the machine exits $u$ by moving in direction $d'$ in state $q'$. For example by $(\rightarrow,q') \in \phi(u)(\rightarrow,q)$ we mean that there exist $f_1,f_2,\ldots,f_k \in F$ such that
\[ (f_k \circ \cdots \circ f_1)(q,x.uy) = (q',xu.y) \]
where $x$ and $y$ are tails such that $xuy \in X$, and the head does not leave $u$ during the ``intermediate steps'', i.e.\ for all $1 \leq j < k$ we have $(f_j \circ \cdots \circ f_1)(q,xu'.u''y)$ where $u = u'u''$ and $|u'|, |u''| > 0$.

Recall that sofic $\ZZ$-subshifts have a regular language~\cite{LiMa95} and thus have an associated syntactic monoid. Associate to each nonempty $u$ its class $\psi(u)$ in the syntactic monoid of $L(X)$. Write $I(u) = (\phi(u),\psi(u))$. Then $I$ is a semigroup homomorphism where defined in the sense that if $uv \sqsubset X$ then
\[ I(uv) = (\phi(uv),\psi(uv)) = (\phi(u)\phi(v),\psi(u)\psi(v)) = I(u)I(v), \]
where maps $\phi(u) \in \{\leftarrow, \rightarrow\} \times Q \to 2^{\{\leftarrow, \rightarrow\} \times Q}$ are given a semigroup structure by observing that the possible ways one can exit $uv$ after entering it from the left or right are entirely determined by the corresponding information for $u$ and $v$.

Now, by the Ramsey theorem, there exists $N$ such that in every word $w \sqsubset X$ of length at least $N$, there is a subword $uv$ such that $u$ and $v$ are nonempty, and $I(u) = I(v) = I(uv)$. To see this, let $w$ be such a word and color the 2-subset $\{i,j\} \subset N$ by $I(w_{[i,j]})$. If $N$ is large enough, there is a monochromatic subset of size $3$, which corresponds to the word $uv$. In particular, by the previous paragraph it follows that
\[ I(u) = I(uv) = I(u)I(v) = I(u)I(u) = I(u^2) \cdots = I(u^k) \]
for all $k$.\footnote{Alternatively, one can apply the well-known fact that every finite semigroup has an idempotent, but the argument of the present paragraph makes it more explicit why one can pick the idempotent $u$ in the language of $X$ and so that it has the traversability property.}

If $F$ generates an infinite group, then there are arbitrarily long words $w$ that the head can walk over, under a suitable application of elements of $F$. Otherwise, every long enough word blocks movement, so $F$ must generate a finite group. More precisely, we have arbitrarily long words $w \in X$ which are \emph{traversable}, meaning
\[ (\rightarrow,q') \in \phi(w)(\rightarrow,q) \mbox{ or } (\leftarrow,q') \in \phi(w)(\leftarrow,q) \]
for some $q, q' \in Q$.

If $w$ has this property, then all its subwords do. It follows that if $F$ generates an infinite group then there is a traversable word $u$ such that $I(u) = I(uu)$. Once we find such a word $u$, we have $u^k \sqsubset X$ for all $k$ and that each $u^k$ is traversable, thus $F$ must be infinite.\end{proof}

Before tackling the problem in the multidimensional case, we recall the \define{snake tiling problem} introduced and shown to be undecidable in \cite{Ka03}. In this problem, a set $T$ of square tiles with colored edges which have an associated direction arrow is given, and the goal is to find a partial tiling of $\ZZ^2$ -- that is, some positions can be left without tiles -- such that if among two adjacent tiles the arrow of one points to the other then they share the same color on the adjacent edge. Furthermore we require that at least one infinite path appears in the partial tiling while following the direction associated to the tiles.

For the next proof we are going to use a slightly modified version of the snake tiling problem, which is also undecidable~\cite{Ka03}. Instead of using Wang tiles with just an outgoing direction we are going to use tiles which have both a left and right direction. Formally, let $T$ be a  finite set of tiles with colored edges and functions $\texttt{left},\texttt{right}\colon T \to D$ where $D = \{(1,0),(-1,0),(0,1),(0,-1)\}$ which satisfy $\texttt{left}(t) \neq \texttt{right}(t)$ for each $t \in T$. We are going to ask for a partial tiling $\tau\colon \ZZ^2 \to T \cup \{\epsilon\}$ such that there exists a function $p\colon \ZZ \to \ZZ^2$ such that $\tau(p(n)) \in T$, $p(n+1) - p(n) = \texttt{right}(\tau(p(n)))$ and $p(n)-p(n+1) = \texttt{left}(\tau(p(n+1)))$ for all $n \in \ZZ$ and all tiles match their non-$\epsilon$ neighbors along the arrows. If a partial tiling $\tau$ with such a path exists, we say the instance $(T,\texttt{left},\texttt{right})$ of the problem \define{admits a snake}.

One way to think about this version is that right arrows give instructions on how to walk to $+\infty$ in this path, while left arrows point to $-\infty$.

\begin{restatable}{theorem}{TwoDTorsion}\label{thm:decidability_torsion_FSA_inchapter}
For all $n \geq 2, k \geq 1, d \geq 2$, there is a finitely generated subgroup of $\FG(\ZZd,n,k)$ whose torsion problem in undecidable.
\end{restatable}

\begin{proof}For the rest of the proof we assume $d = 2$. In the general case, the result follows from the obvious fact that $\FG(\ZZ^2,n,k)$ embeds into $\FG(\ZZd,n,k)$.

First, let us explain why, if we let the alphabet $\Sigma$ be arbitrary and take the local rule of an element $f \in \FG(\ZZ^2,|\Sigma|,2)$, it is undecidable whether $T_f$ is a torsion element in that group. We then explain how to do this construction in the case of a fixed alphabet and number of states. Finally, we show that the torsion problem still remains undecidable in a finitely generated subgroup of $\FG(\ZZ^2,n,k)$ for every $n \geq 2$ and $k \geq 1$.

Consider an instance of our modified snake tiling problem $(T,\texttt{left},\texttt{right})$ where $T$ is the set of tiles. We choose $|\Sigma| > |T|$ and associate the first $|T|$ symbols in $\Sigma$ to the tiles in $T$ and the rest to the empty tile. We construct a local rule $f$ which gives a non-torsion element $T_f$ if and only if $(T,\texttt{left},\texttt{right})$ admits a snake.

In this instance we ask for two states, we will call them \define{direction bits}  $Q = \{R,L\}$ standing for right and left. The machine $T_f$ has radius $F = \{ (i,j) \in \ZZ^2 \mid |i|+|j| \leq 1 \}$ and acts as follows:

\begin{itemize}
	\item Let $t$ be the tile at $(0,0)$. If $t = \epsilon$, do nothing.
	\item Otherwise:
	\begin{itemize}
		\item If the state is $L$. Check the tile in the direction $\texttt{left}(t)$. If it matches correctly with $t$ move the head to that position, otherwise switch the state to $R$.
		\item If the state is $R$. Check the tile in the direction $\texttt{right}(t)$. If it matches correctly with $t$ move the head to that position, otherwise switch the state to $L$
	\end{itemize}
\end{itemize}

The machine $T_f$ is reversible and its inverse is given by the machine which does the same but switches the roles of $R$ and $L$. If $(T,\texttt{left},\texttt{right})$ admits a snake, it suffices to consider the configuration in $\Sigma^{\ZZ^2} \times X_{q}$ that contains an infinite snake and such that the head of the machine is positioned somewhere in the snake. Clearly $T_f$ walks to infinity in that configuration without repeating positions, thus showing that $T_f$ is a non-torsion element (recall we are using the moving head model). Conversely, if $(T,\texttt{left},\texttt{right})$ does not admit a snake, then there is an uniform bound on how far $T_f$ can walk from its starting position before encountering an error or entering a cycle and henceforth $T_f$ has finite order. This concludes the proof of undecidability of the torsion problem when the alphabet is arbitrary.

Now we explain how to pass to a fixed alphabet and how to avoid using states. For this we encode the tiles into squares of size $n \times n$. By having the bottom left corner of the coding contain $\begin{smallmatrix}1&1\\1&1\end{smallmatrix}$ and having no two adjacent $1$s elsewhere, we ensure that there is a unique way to ``parse'' a given tiling into encodings of squares. Clearly, the movements of the machine $T_f$ are now inflated by $n$ and anything which is not a valid coding of a tile is interpreted as the empty tile. Also, as there is a unique way to parse correct squares, one can also use the position of the head in the bottom left corner to code the states. Say, the lower left corner of $\begin{smallmatrix}1&1\\1&1\end{smallmatrix}$ codes $L$ and the lower right corner of $(1,0)$ codes $R$. Thus the state change amounts to a shift by either $(1,0)$ or $(-1,0)$.

The inconvenient part of the previous construction is that the subgroup of machines defined by it is not necessarily finitely generated. Thus a priori it might be the case that every finitely generated subgroup of $\FG(\ZZ^2,n,k)$ still has decidable torsion problem. In order to show this does not hold we construct a finite set of machines which simulates the previous construction. For this we are going to use a specific $7 \times 7$ square coding which is shown on Figure~\ref{snake_tile_code}. This coding is composed of three zones. The \define{outer zone} consists of a ring of $1$s of side length $7$ which serves to code unambiguously the boundary of the structure. The four bottom left $1$s of this zone are used to code the states, this is obtained by forcing the head of the machine to always stay in one of these positions modulo $\ZZ^2 /7\ZZ^2$. The \define{middle zone} consists of a ring of $0$s of side length $5$ which serves to separate the three zones so no ambiguity is possible. Finally there is the \define{inner zone} consisting of a $3 \times 3$ square containing a configuration in $\{0,1\}^9$. Four of these bits $l_1,l_2,r_1,r_2$ serve to code two directions in $D = \{(1,0),(-1,0),(0,1),(0,-1)\}$. The rest of the bits are going to be specified later on.
\begin{figure}[h!]
	\centering
	\include{snake_tile_code}
	\caption{Basic coding of the construction. The outer ring of $1$s (blue) codes the boundary of the cell and the state. The middle ring of $0$s separates the zones The inner ring (green) codes the information.}
	\label{snake_tile_code}
\end{figure}

For this construction we are going to use a two bit string $s \in \{0,1\}^2$ as the set of states (which is to be coded by the head position amongst the four fixed places in the outer ring of $1$s). The first bit is the \define{direction bit}, that is, it takes the role of $L$ and $R$ for the first construction. The second bit is the \define{auxiliary bit}, whose role will become clear later on.

Let $C$ be the set of all patterns of shape as in Figure~\ref{snake_tile_code} centered in one of four fixed positions in the ring of $1$s --that is, such that the support is of the form $([-i,7-i]\times[0,7]) \cap \ZZ^2$ for some $i \in \{0,1,2,3\}$--. We consider the following finite set of machines as our generating set $S$.

\begin{enumerate}
	\item $\{T_{\vec v}\}_{v \in D}$ that walks in the direction $\vec v \in D$ independently of the configuration.
	\item $T_{\mathtt{walk}}$ that walks along the direction codified by $l_1,l_2$ or $r_1,r_2$ depending on the direction bit.
	\item $\{g_c\}_{c \in C}$ that flips the direction bit if the current pattern is $c \in C$,
	\item $\{h_c\}_{c \in C}$ that flips the auxiliary bit if the current pattern is $c \in C$,
	\item $\{g_{+,c}\}_{c \in C}$ that adds the auxiliary bit to the direction bit if the current pattern is $c \in C$, and
	\item $\{h_{+,c}\}_{c \in C}$ that adds the direction bit to the auxiliary bit if the current pattern is $c \in C$,
\end{enumerate}

The machine $T_{\mathtt{walk}}$ is the only one which needs to be carefully defined. It acts similarly to $T_f$ defined in the beginning. Formally it does the following:
\begin{itemize}
	\item If the pattern around the identity does not correspond to a $c \in C$, do nothing.
	\item Otherwise:
	\begin{itemize}
		\item If the direction bit is $0$ check the pattern centered in $7 \texttt{left}(t)$ from the actual head position. If it is a valid $c' \in C$ in the same state and its two right bits code $-\texttt{left}(t)$ then move the head by $7\texttt{left}(t)$. Otherwise flip the direction bit to $1$.
		\item If the direction bit is $1$. Check the pattern centered in $7 \texttt{right}(t)$ from the actual head position. If it is a valid $c' \in C$ in the same state and its two left bits code $-\texttt{right}(t)$ then move the head by $7\texttt{right}(t)$. Otherwise flip the direction bit to $0$.
	\end{itemize}
\end{itemize}

The machines $T_{\vec{v}}$ simply act as the shift by $\vec{v} \in D$ which is clearly reversible. As $\langle D \rangle = \ZZ^2$ we have that for every vector $\vec{u} \in \ZZ^2$ the machine $T_{\vec{u}}$ which moves the head by $\vec{u}$ belongs to $\langle S \rangle$.

Let $p^*$ be a pattern consisting of the concatenation of patterns from $c$ which are well aligned along the columns and lines of $1$s. More formally, for a finite $F \subset \ZZ^2$, $p^*$ is a pattern with support $7F +(([-i,7-i]\times[0,7]) \cap \ZZ^2)$ for some $i \in \{0,1,2,3\}$ and such that for every $\vec{v} \in F$ then $\sigma^{-7\vec{v}}(p^*)|_{([-i,7-i]\times[0,7]) \cap \ZZ^2} \in C$. We define $g_{p^*}$ and $h_{p^*}$ as the machines which flip the direction bit and the auxiliary bit respectively if they read $p^*$. We claim $g_{p^*}, h_{p^*} \in \langle S\rangle$. If $p^*$ is defined by some singleton $F = \{\vec{v}\}$ it suffices to note that $g_{p^*} = T_{-7\vec{v}}\circ g_c \circ T_{7\vec{v}}$ and $h_{p^*} = T_{-7\vec{v}}\circ h_c \circ T_{7\vec{v}}$ for the appropriate $c \in C$. Inductively, we can choose $\vec{v} \in F$ and separate $p^*$ as the disjoint union of the pattern $p^*_{F \setminus \{\vec{v}\}}$ defined by $F \setminus \{\vec{v}\}$, and the pattern $p^*_{\vec{v}}$ defined by $\vec{v}$ and thus write: $$g_{p^*} = (T_{-7\vec v} \circ g_{+,c} \circ T_{7\vec v} \circ h_{p^*_{F \setminus \{\vec{v}\}}})^2, \mbox{ and } h_{p^*} = (T_{-7\vec v} \circ h_{+,c} \circ T_{7\vec v} \circ g_{p^*_{F \setminus \{\vec{v}\}}})^2. $$

Consider an instance $(T,\texttt{left},\texttt{right})$ of the snake tiling problem. The information associated to each tile $ t\in T$ consists of a $4$-tuple of colors $(c_1,c_2,c_3,c_4)$ and the directions $\texttt{left}(t)$ and $\texttt{right}(t)$. Suppose the tiles of $T$ are defined using $N$ colors. Let $M \in \NN$ such that $M^2 > \log_{2}(N)$. We define for each $t \in T$ a macrotile $\mathcal{M}(t)$ as a fixed square array of patterns of shape as in Figure~\ref{snake_tile_code} of side length $M$ (see Figure~\ref{snake_macrotile_code}). We fix an enumeration of these patterns from the bottom left to the upper right as $\{c_{j}\}_{1 \leq j \leq M^2}$ and denote the bit $b_i$ of $c_j$ as $b_{i,j}$. We demand $\mathcal{M}(t)$ to satisfy the following properties:

\begin{itemize}
	\item For $i \in \{1,2,3,4\}$ the sequence of bits $\{b_{i,j}\}_{1 \leq j \leq M^2}$ codifies the color $c_i$.
	\item $b_{5,1} = 1$ and for all $j > 1$ the bit $b_{5,j} = 0$.
	\item The bits $l_1,l_2$ and $r_1,r_2$ of $c_1$ code $\texttt{left}(t)$ and $\texttt{right}(t)$ respectively.
	\item  If $\texttt{left}(t) =(1,0)$ then for all $2 \leq j \leq M$ we have that $l_1,l_2$ and $r_1,r_2$ of $c_j$ code $(1,0)$ and $(-1,0)$ respectively.
	\item  If $\texttt{right}(t) =(1,0)$ then for all $2 \leq j \leq M$ we have that $l_1,l_2$ and $r_1,r_2$ of $c_j$ code $(-1,0)$ and $(1,0)$ respectively.
	\item If $\texttt{left}(t) =(0,1)$ then for all $1 \leq j \leq M-1$ we have that $l_1,l_2$ and $r_1,r_2$ of $c_{1+jM}$ code $(0,1)$ and $(0,-1)$ respectively.
	\item If $\texttt{right}(t) =(0,1)$ then for all $1 \leq j \leq M-1$ we have that $l_1,l_2$ and $r_1,r_2$ of $c_{1+jM}$ code $(0,-1)$ and $(0,1)$ respectively.
\end{itemize}

As $M^2 > \log_{2}(N)$ it is possible to satisfy the first requirement. The rest are possible to satisfy as $\texttt{left}(t) \neq \texttt{right}(t)$ . An example of such a macrotile is represented in Figure~\ref{snake_macrotile_code}.

\begin{figure}[h!]
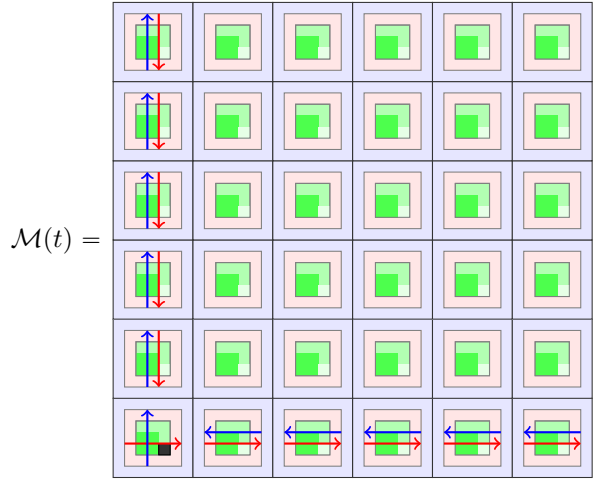

	\centering
	\include{snake_macrotile_code}
	\caption{An example of macrotile $\mathcal{M}(t)$ of side $M = 6$. The red arrows represent the function $\texttt{left}(t) = (1,0)$ while the blue arrows represent $\texttt{right}(t) = (0,1)$. The bottom left black square represents $b_{5,1} = 1$.}
	\label{snake_macrotile_code}
\end{figure}

Associate all arrays of $M \times M$ codings which do not represent some $t \in T$ to the $\epsilon$ tile. Also, let $\mathcal{M}$ be the set of all patterns given as an array of $3 \times 3$ macrotiles which represent a valid local pattern of the snake problem and such that the middle tile is not an $\epsilon$ tile and are centered in the bottom left position of the middle macrotile.

Consider the machine $T^* \in \langle S \rangle$ given by: $$T^* = (T_{\mathtt{walk}})^M \circ \prod_{p^* \in \mathcal{M}}g_{p^*} \circ \prod_{c \in C} g_c$$

We claim that $T^*$ is a torsion element if and only if $(T,\texttt{left},\texttt{right})$ does not admit a snake.

If $(T,\texttt{left},\texttt{right})$ admits a snake, it suffices to take a configuration with a snake, replace each tile and $\epsilon$ in it by a corresponding macrotile and put the head of the machine in the lower left corner of a macrotile belonging to the snake. The machine $T^*$ will first detect some pattern $c \in C$, so exactly one $g_c$ will flip the direction bit once. Then it will detect a valid pattern $p^*$ of the snake problem and thus $g_{p^*}$ will flip again the direction bit amounting to no action at all. Finally, $(T_{\mathtt{walk}})^M$ will just walk towards the lower left corner of the next macrotile. As the initial configuration codes a snake, repeating this procedure will make $T^*$ walk to infinity, and thus $T^*$ is not a torsion element.

For the converse, we need to analyze more carefully the behavior of $T^*$. First of all, if in the initial configuration the head is not over a pattern $c \in C$, then $T^*$ by definition acts trivially. Otherwise, suppose the head is over a $c \in C$. The application of $\prod_{p^* \in \mathcal{M}}g_{p^*} \circ \prod_{c \in C} g_c$ can at most do a state change, and thus the head still sees some $c' \in C$ afterwards. Also, by definition of $T_{\texttt{walk}}$, the head will always see an element of $C$ after applying $(T_{\texttt{walk}})^M$. This means that the head will always be seeing a pattern in $C$ after applying $T^*$.

There are two possible behaviors of $T^*$ starting from a pattern in $C$. If the head is not over a valid array of macrotiles in $\mathcal{M}$ then the direction bit is flipped by $g_{c}$, the second part does nothing, and $T_{\mathtt{walk}}$ is applied $M$ times. Otherwise the direction bit is flipped two times, amounting to no flip at all and $T_{\mathtt{walk}}$ is applied $M$ times.

These two behaviors translate into the following: If the head is over a valid array of macrotiles in $\mathcal{M}$ then $T^*$ can either move into another valid array (and correctly simulate the working of $T_f$ defined at first in the proof), or it can fall outside a valid array of macrotiles. It it does this, then another application of $T^*$ undoes the last $M$ steps of $T_{\mathtt{walk}}$ and changes the direction bit. Therefore the machine continues to live inside a valid array of $\mathcal{M}$ and simulate $T_f$. In this case we can use the uniform bound on the length of the snake to find a bound $N$ such that $(T^*)^N$ acts trivially over all these configurations. The only case remaining is when initially the head is not over an array in $\mathcal{M}$ and after one application of $T^*$ it stays that way. In this case, we just have that $(T^*)^2$ acts trivially over these configurations. Thus showing that $(T^*)^{2N} = id$ and thus $T^*$ is a torsion element of $\langle S \rangle$. \end{proof} 

In the special case where $k = 1$ the previous result can be expressed in dynamical terms. Namely, $\FG(\ZZd,n,1)$ is exactly the topological full group of the full $\ZZ^2$-shift on $n$ symbols. This yields Corollary~\ref{cor:fullgroup_undecidabletorsion}, that is, for every $d \geq 2$ the topological full group of the full $\ZZ^2$-shift on $n$ symbols contains a finitely generated subgroup with undecidable torsion problem.

\section{Acknowledgments}

S.\ Barbieri was partially supported by the ANR projects CODYS (ANR-18-CE40-0007), CoCoGro (ANR-16-CE40-0005) and the FONDECYT grant 11200037. V.\ Salo was partially supported by FONDECYT grant 3150552 and Academy of Finland grant 2608073211.


\bibliographystyle{abbrv}
\bibliography{bib}

\end{document}

%% file: z2_machine2.tex
\begin{tikzpicture}[scale=0.50]

\begin{scope}[scale = 0.75, shift={(-6,0)},rotate=0]
\def \c{0.3}

\draw [<->] (-5,0) to (5,0);
\draw [<->] (-6,-1) to (4,-1);
\draw [<->] (-7,-2) to (3,-2);
\draw [<->] (-4,1) to (6,1);
\draw [<->] (-3,2) to (7,2);

\draw [<->] (3,3) to (-3,-3);
\draw [<->] (3-5/3,3) to (-3-5/3,-3);
\draw [<->] (3-10/3,3) to (-3-10/3,-3);
\draw [<->] (3+5/3,3) to (-3+5/3,-3);
\draw [<->] (3+10/3,3) to (-3+10/3,-3);

\foreach \i in {-10/3,-5/3,0,5/3,10/3}{
	\foreach \j in {-2,-1,0,1,2}{
	\draw[fill = white] (\i+\j,\j) circle (\c);
}
}


\filldraw[thick,dashed,fill=blue,fill opacity=0.3] (-5/6-0.5,-0.5) -- (-5/6+0.5,0.5) -- (5/6+0.5,0.5)  -- (5/6+1.5,1.5) -- (15/6+1.5,1.5)  -- (15/6+0.5,0.5) -- (15/6-0.5,-0.5)  -- cycle;

\draw[fill = white] (-10/3,0) circle (\c);
\draw[fill = black!75] (-5/3,0)circle (\c);
\draw[fill = black!75] (0,0)circle (\c);
\draw[fill = white] (5/3,0) circle (\c);
\draw[fill = white] (10/3,0) circle (\c);

\draw[fill = white] (-10/3+1,1) circle (\c);
\draw[fill = white] (-5/3+1,1)circle (\c);
\draw[fill = black!75] (1,1)circle (\c);
\draw[fill = black!75] (5/3+1,1) circle (\c);
\draw[fill = white] (10/3+1,1) circle (\c);

\draw[fill = white] (-10/3-1,-1) circle (\c);
\draw[fill = black!75] (-5/3-1,-1)circle (\c);
\draw[fill = black!75] (-1,-1)circle (\c);
\draw[fill = black!75] (5/3-1,-1) circle (\c);
\draw[fill = white] (10/3-1,-1) circle (\c);
%
%

\draw [->, ultra thick] (-2,3) to (-0.3, 0.3);
\draw[fill = white] (-2,3) circle (0.5);
\draw node (1) at (-2,3) {$q_1$};

\end{scope}

\begin{scope}[scale = 0.75, shift={(6,0)},rotate=0]
\def \c{0.3}

\draw [<->] (-5,0) to (5,0);
\draw [<->] (-6,-1) to (4,-1);
\draw [<->] (-7,-2) to (3,-2);
\draw [<->] (-4,1) to (6,1);
\draw [<->] (-3,2) to (7,2);

\draw [<->] (3,3) to (-3,-3);
\draw [<->] (3-5/3,3) to (-3-5/3,-3);
\draw [<->] (3-10/3,3) to (-3-10/3,-3);
\draw [<->] (3+5/3,3) to (-3+5/3,-3);
\draw [<->] (3+10/3,3) to (-3+10/3,-3);

\foreach \i in {-10/3,-5/3,0,5/3,10/3}{
	\foreach \j in {-2,-1,0,1,2}{
		\draw[fill = white] (\i+\j,\j) circle (\c);
	}
}


\filldraw[thick,dashed,fill=blue,fill opacity=0.3] (-5/6-0.5,-0.5) -- (-5/6+0.5,0.5) -- (5/6+0.5,0.5)  -- (5/6+1.5,1.5) -- (15/6+1.5,1.5)  -- (15/6+0.5,0.5) -- (15/6-0.5,-0.5)  -- cycle;

\draw[fill = white] (-10/3,0) circle (\c);
\draw[fill = black!75] (-5/3,0)circle (\c);
\draw[fill = white] (0,0)circle (\c);
\draw[fill = black!75] (5/3,0) circle (\c);
\draw[fill = white] (10/3,0) circle (\c);

\draw[fill = white] (-10/3+1,1) circle (\c);
\draw[fill = white] (-5/3+1,1)circle (\c);
\draw[fill = black!75] (1,1)circle (\c);
\draw[fill = white] (5/3+1,1) circle (\c);
\draw[fill = white] (10/3+1,1) circle (\c);

\draw[fill = white] (-10/3-1,-1) circle (\c);
\draw[fill = black!75] (-5/3-1,-1)circle (\c);
\draw[fill = black!75] (-1,-1)circle (\c);
\draw[fill = black!75] (5/3-1,-1) circle (\c);
\draw[fill = white] (10/3-1,-1) circle (\c);
%
%

\draw [->, ultra thick] (-2+5/3+1,3+1) to (-0.3+5/3+1, 0.3+1);
\draw[fill = white] (-2+5/3+1,3+1)  circle (0.5);
\draw node (2) at (-2+5/3+1,3+1)  {$q_2$};
\end{scope}
\path
(1) edge [thick, bend left=30,->] node[swap]  {} (2);

\draw node at (-1.5,-3) {$f($\begin{tikzpicture}[scale = 0.2]
\draw[fill = black!75] (0,0) circle (0.3);
\draw[fill = white] (1,0)circle (0.3);
\draw[fill = black!75] (1,1) circle (0.3);
\end{tikzpicture}$,q_1) = ($\begin{tikzpicture}[scale = 0.2]
\draw[fill = white] (0,0) circle (0.3);
\draw[fill = black!75] (1,0)circle (0.3);
\draw[fill = white] (1,1) circle (0.3);
\end{tikzpicture}$,q_2,(1,1))$ \ \ \ $F_{\text{in}} = F_{\text{out}} = \{(0,0), (1,0), (1,1)\}$};
\draw node at (-1,3.5) {$T_f$};

\end{tikzpicture}

%% file: z2_machine.tex
\begin{tikzpicture}[scale=0.50]

\begin{scope}[scale = 0.75, shift={(-6,0)},rotate=0]
\def \c{0.3}

\draw [<->] (-5,0) to (5,0);
\draw [<->] (-6,-1) to (4,-1);
\draw [<->] (-7,-2) to (3,-2);
\draw [<->] (-4,1) to (6,1);
\draw [<->] (-3,2) to (7,2);

\draw [<->] (3,3) to (-3,-3);
\draw [<->] (3-5/3,3) to (-3-5/3,-3);
\draw [<->] (3-10/3,3) to (-3-10/3,-3);
\draw [<->] (3+5/3,3) to (-3+5/3,-3);
\draw [<->] (3+10/3,3) to (-3+10/3,-3);

\foreach \i in {-10/3,-5/3,0,5/3,10/3}{
	\foreach \j in {-2,-1,0,1,2}{
	\draw[fill = white] (\i+\j,\j) circle (\c);
}
}


\filldraw[thick,dashed,fill=blue,fill opacity=0.3] (-5/6-0.5,-0.5) -- (-5/6+0.5,0.5) -- (5/6+0.5,0.5)  -- (5/6+1.5,1.5) -- (15/6+1.5,1.5)  -- (15/6+0.5,0.5) -- (15/6-0.5,-0.5)  -- cycle;

\draw[fill = white] (-10/3,0) circle (\c);
\draw[fill = black!75] (-5/3,0)circle (\c);
\draw[fill = black!75] (0,0)circle (\c);
\draw[fill = white] (5/3,0) circle (\c);
\draw[fill = white] (10/3,0) circle (\c);

\draw[fill = white] (-10/3+1,1) circle (\c);
\draw[fill = white] (-5/3+1,1)circle (\c);
\draw[fill = black!75] (1,1)circle (\c);
\draw[fill = black!75] (5/3+1,1) circle (\c);
\draw[fill = white] (10/3+1,1) circle (\c);

\draw[fill = white] (-10/3-1,-1) circle (\c);
\draw[fill = black!75] (-5/3-1,-1)circle (\c);
\draw[fill = black!75] (-1,-1)circle (\c);
\draw[fill = black!75] (5/3-1,-1) circle (\c);
\draw[fill = white] (10/3-1,-1) circle (\c);
%
%

\draw [->, ultra thick] (-2,3) to (-0.3, 0.3);
\draw[fill = white] (-2,3) circle (0.5);
\draw node (1) at (-2,3) {$q_1$};

\end{scope}

\begin{scope}[scale = 0.75, shift={(6,0)},rotate=0]
\def \c{0.3}

\draw [<->] (-5,0) to (5,0);
\draw [<->] (-6,-1) to (4,-1);
\draw [<->] (-7,-2) to (3,-2);
\draw [<->] (-4,1) to (6,1);
\draw [<->] (-3,2) to (7,2);

\draw [<->] (3,3) to (-3,-3);
\draw [<->] (3-5/3,3) to (-3-5/3,-3);
\draw [<->] (3-10/3,3) to (-3-10/3,-3);
\draw [<->] (3+5/3,3) to (-3+5/3,-3);
\draw [<->] (3+10/3,3) to (-3+10/3,-3);

\filldraw[thick,dashed,fill=blue,fill opacity=0.3] (-15/6-1.5,-1.5) -- (-15/6-0.5,-0.5) -- (-5/6-0.5,-0.5)  -- (-5/6+0.5,0.5) -- (5/6+0.5,0.5)  -- (5/6-0.5,-0.5) -- (5/6-1.5,-1.5)  -- cycle;

\foreach \i in {-10/3,-5/3,0,5/3,10/3}{
	\foreach \j in {-2,-1,0,1,2}{
		\draw[fill = white] (\i+\j,\j) circle (\c);
	}
}
\draw[fill = black!75] (-10/3-1,-1)circle (\c);
\draw[fill = black!75] (-1,-1) circle (\c);

\draw[fill = black!75] (-5/3,0)circle (\c);

\draw[fill = black!75] (-10/3-2,-2) circle (\c);
\draw[fill = black!75] (-5/3-2,-2)circle (\c);
\draw[fill = black!75] (-2,-2)circle (\c);

%
%

\draw [->, ultra thick] (-2,3) to (-0.3, 0.3);
\draw[fill = white] (-2,3) circle (0.5);
\draw node (2) at (-2,3) {$q_2$};
\end{scope}
\path
(1) edge [thick, bend left=30,->] node[swap]  {} (2);

\draw node at (-1.5,-3) {$f($\begin{tikzpicture}[scale = 0.2]
\draw[fill = black!75] (0,0) circle (0.3);
\draw[fill = white] (1,0)circle (0.3);
\draw[fill = black!75] (1,1) circle (0.3);
\end{tikzpicture}$,q_1) = ($\begin{tikzpicture}[scale = 0.2]
\draw[fill = white] (0,0) circle (0.3);
\draw[fill = black!75] (1,0)circle (0.3);
\draw[fill = white] (1,1) circle (0.3);
\end{tikzpicture}$,q_2,(1,1))$ \ \ \ $F_{\text{in}} = F_{\text{out}} = \{(0,0), (1,0), (1,1)\}$};
\draw node at (-1.5,3.2) {$T_f$};

\end{tikzpicture}

%% file: ville_head_counter.tex
\begin{tikzpicture}[scale = 0.75]

\foreach \i in {-15,...,14}{
	\draw[ fill = white] (\i/2,0) rectangle (\i/2+0.5,0.5);
	\node at (\i/2+0.25,0.25) {$a$};
	\draw[ fill = white] (\i/2,-0.5) rectangle (\i/2+0.5,0);
	\node at (\i/2+0.25,-0.25) {$0$};
}
\foreach \i in {-15,-12,-11,-9,-8,-7,-3,-2,-1,3,5,6,7,8,10,11,14}{
	\draw[ fill = white] (\i/2,0) rectangle (\i/2+0.5,0.5);
	\node at (\i/2+0.25,0.25) {$b$};
}
\foreach \i in {-9,...,9}{
	\draw[ fill = white] (\i/2,-0.5) rectangle (\i/2+0.5,0);
	\node at (\i/2+0.25,-0.25) {$1$};
}
\draw[decorate,decoration={brace,amplitude=2pt,mirror,raise=2pt},yshift=0pt]
(7,0.65) -- (-6.5,0.65) node [black,midway, yshift=10pt] {\footnotesize
	$|u|$};
\filldraw (-6.5,-0.5) circle (0.1);
\draw[decorate,decoration={brace,amplitude=2pt,mirror,raise=2pt},yshift=0pt]
(-6.5,-0.65) -- (-4.5,-0.65) node [black,midway, yshift=-10pt] {\footnotesize
	$r$};
\filldraw (-6.5,-0.5) circle (0.1);
\draw[decorate,decoration={brace,amplitude=2pt,mirror,raise=2pt},yshift=0pt]
(-4.5,-0.65) -- (-2.5,-0.65) node [black,midway, yshift=-10pt] {\footnotesize
	$r$};
\draw[decorate,decoration={brace,amplitude=2pt,mirror,raise=2pt},yshift=0pt]
(5,-0.65) -- (7,-0.65) node [black,midway, yshift=-10pt] {\footnotesize
	$r$};
\draw[decorate,decoration={brace,amplitude=2pt,mirror,raise=2pt},yshift=0pt]
(3,-0.65) -- (5,-0.65) node [black,midway, yshift=-10pt] {\footnotesize
	$r$};

\draw[<->,thick] (-8,0.5) to (8,0.5);
\draw[<->,thick] (-8,-0.5) to (8,-0.5);

\draw[fill = blue, opacity = 0.2] (-2.5,-0.5) rectangle (3,0);
\draw[fill = green, opacity = 0.2] (-6.5,0) rectangle (7,0.5);
\draw[fill = red, opacity = 0.2] (-4.5,-0.5) rectangle (-2.5,0);
\draw[fill = red, opacity = 0.2] (3,-0.5) rectangle (5,0);

\draw[thick, ->]
(0,-0.65) -- (0,-1.65) node [black,midway, xshift= 10pt] {\footnotesize
	$\phi_{\sigma^2}$};

\begin{scope}[shift = {(0,-2.5)}]

\foreach \i in {-15,...,14}{
	\draw[ fill = white] (\i/2,0) rectangle (\i/2+0.5,0.5);
	\node at (\i/2+0.25,0.25) {$a$};
	\draw[ fill = white] (\i/2,-0.5) rectangle (\i/2+0.5,0);
	\node at (\i/2+0.25,-0.25) {$0$};
}
\foreach \i in {-15,-12,-11,-9,-8,-7,-3,-2,-1,3,5,6,7,8,10,11,14}{
	\draw[ fill = white] (\i/2,0) rectangle (\i/2+0.5,0.5);
	\node at (\i/2+0.25,0.25) {$b$};
}
\foreach \i in {-7,...,11}{
	\draw[ fill = white] (\i/2,-0.5) rectangle (\i/2+0.5,0);
	\node at (\i/2+0.25,-0.25) {$1$};
}
\draw[<->,thick] (-8,0.5) to (8,0.5);
\draw[<->,thick] (-8,-0.5) to (8,-0.5);

\draw[fill = blue, opacity = 0.2] (-2.5,-0.5) rectangle (3,0);
\draw[fill = green, opacity = 0.2] (-6.5,0) rectangle (7,0.5);
\draw[fill = red, opacity = 0.2] (-3.5,-0.5) rectangle (-2.5,0);
\draw[fill = red, opacity = 0.2] (3,-0.5) rectangle (6,0);

\end{scope}

\end{tikzpicture}

%% file: conveyor_zones.tex
\begin{tikzpicture}[scale = 0.75]

\foreach \i in {-15,...,14}{
	\draw[ fill = white] (\i/2,0) rectangle (\i/2+0.5,0.5);
	\node at (\i/2+0.25,0.25) {$0$};
	\draw[ fill = white] (\i/2,-1) rectangle (\i/2+0.5,-0.5);
	\node at (\i/2+0.25,-0.75) {$0$};
	\draw[ fill = white] (\i/2,-2) rectangle (\i/2+0.5,-1.5);
}
\foreach \i in {-15,-12,-11,-9,-8,-7,-3,-2,-1,3,5,6,7,8,10,11,14}{
	\draw[ fill = white] (\i/2,0) rectangle (\i/2+0.5,0.5);
	\node at (\i/2+0.25,0.25) {$1$};
}
\foreach \i in {-13,-12,-10,-8,-7,-6,-5,-1,0,1,3,5,6,8,10,11,13,14}{
	\draw[ fill = white] (\i/2,-1) rectangle (\i/2+0.5,-0.5);
	\node at (\i/2+0.25,-0.75) {$1$};
}

\draw[<->,thick] (-8,0) to (8,0);
\draw[<->,thick] (-8,0.5) to (8,0.5);
\draw[<->,thick] (-8,-0.5) to (8,-0.5);
\draw[<->,thick] (-8,-1) to (8,-1);
\draw[<->,thick] (-8,-1.5) to (8,-1.5);
\draw[<->,thick] (-8,-2) to (8,-2);

\node at (-15/2+0.25,-1.75) {$\rightarrow$};
\node at (-14/2+0.25,-1.75) {$\rightarrow$};
\node at (-13/2+0.25,-1.75) {$\leftarrow$};
\draw[ dashed, thick ] (-12/2, -2.5) to (-12/2, 1);
\node at (-12/2+0.25,-1.75) {$\rightarrow$};
\draw[ fill = white] (-11/2,-2.25) rectangle (-11/2+0.5,-1.75);
\draw[ fill = white] (-11/2,-1.75) rectangle (-11/2+0.5,-1.25);
\node at (-11/2+0.25,-1.5) {$q$};
\node at (-11/2+0.25,-2) {$\uparrow$};
\node at (-10/2+0.25,-1.75) {$\leftarrow$};
\node at (-9/2+0.25,-1.75) {$\leftarrow$};
\draw[ dashed, thick ] (-8/2, -2.5) to (-8/2, 1);
\draw[ fill = white] (-8/2,-2.25) rectangle (-8/2+0.5,-1.75);
\draw[ fill = white] (-8/2,-1.75) rectangle (-8/2+0.5,-1.25);
\node at (-8/2+0.25,-1.5) {$q$};
\node at (-8/2+0.25,-2) {$\downarrow$};
\draw[ dashed, thick ] (-7/2, -2.5) to (-7/2, 1);
\draw[ fill = white] (-7/2,-2.25) rectangle (-7/2+0.5,-1.75);
\draw[ fill = white] (-7/2,-1.75) rectangle (-7/2+0.5,-1.25);
\node at (-7/2+0.25,-1.5) {$q$};
\node at (-7/2+0.25,-2) {$\downarrow$};
\node at (-6/2+0.25,-1.75) {$\leftarrow$};
\node at (-5/2+0.25,-1.75) {$\leftarrow$};
\node at (-4/2+0.25,-1.75) {$\leftarrow$};
\draw[ dashed, thick ] (-3/2, -2.5) to (-3/2, 1);
\node at (-3/2+0.25,-1.75) {$\rightarrow$};
\node at (-2/2+0.25,-1.75) {$\rightarrow$};
\node at (-1/2+0.25,-1.75) {$\rightarrow$};
\node at (0.25,-1.75) {$\rightarrow$};
\node at (1/2+0.25,-1.75) {$\rightarrow$};
\draw[ fill = white] (2/2,-2.25) rectangle (2/2+0.5,-1.75);
\draw[ fill = white] (2/2,-1.75) rectangle (2/2+0.5,-1.25);
\node at (2/2+0.25,-1.5) {$q$};
\node at (2/2+0.25,-2) {$\uparrow$};
\node at (3/2+0.25,-1.75) {$\leftarrow$};
\node at (4/2+0.25,-1.75) {$\leftarrow$};
\node at (5/2+0.25,-1.75) {$\leftarrow$};
\node at (6/2+0.25,-1.75) {$\leftarrow$};
\node at (7/2+0.25,-1.75) {$\leftarrow$};
\draw[ dashed, thick ] (8/2, -2.5) to (8/2, 1);
\node at (8/2+0.25,-1.75) {$\rightarrow$};
\node at (9/2+0.25,-1.75) {$\leftarrow$};
\draw[ dashed, thick ] (10/2, -2.5) to (10/2, 1);
\node at (10/2+0.25,-1.75) {$\rightarrow$};
\draw[ fill = white] (11/2,-2.25) rectangle (11/2+0.5,-1.75);
\draw[ fill = white] (11/2,-1.75) rectangle (11/2+0.5,-1.25);
\node at (11/2+0.25,-1.5) {$q$};
\node at (11/2+0.25,-2) {$\uparrow$};
\node at (12/2+0.25,-1.75) {$\leftarrow$};
\node at (13/2+0.25,-1.75) {$\leftarrow$};
\node at (14/2+0.25,-1.75) {$\leftarrow$};

\draw[fill = green, opacity = 0.2] (-6,0) rectangle (-4,0.5);
\draw[fill = green, opacity = 0.2] (-4,-1) rectangle (-1.5,-0.5);
\draw[fill = green, opacity = 0.2] (-1.5,0) rectangle (4,0.5);
\draw[fill = green, opacity = 0.2] (5,0) rectangle (7.5,0.5);

\end{tikzpicture}

%% file: conveyor_belt.tex
\begin{tikzpicture}[scale = 0.75]

\begin{scope}[shift={(-6,0)}]

\foreach \i in {-2,...,4}{
	\draw[ fill = white] (\i/2,0) rectangle (\i/2+0.5,0.5);
	\draw[ fill = white] (\i/2,-1) rectangle (\i/2+0.5,-0.5);
	\draw[ fill = white] (\i/2,-2) rectangle (\i/2+0.5,-1.5);
}

\node at (-2/2+0.25,0.25) {$u_0$};
\node at (-1/2+0.25,0.25) {$u_1$};
\node at (0.25,0.25) {$u_2$};
\node at (1/2+0.25,0.25) {$u_3$};
\node at (2/2+0.25,0.25) {$u_4$};
\node at (3/2+0.25,0.25) {$u_5$};
\node at (4/2+0.25,0.25) {$u_6$};

\node at (-2/2+0.25,-0.75) {$v_0$};
\node at (-1/2+0.25,-0.75) {$v_1$};
\node at (0.25,-0.75) {$v_2$};
\node at (1/2+0.25,-0.75) {$v_3$};
\node at (2/2+0.25,-0.75) {$v_4$};
\node at (3/2+0.25,-0.75) {$v_5$};
\node at (4/2+0.25,-0.75) {$v_6$};

\draw[ dashed, thick ] (-2/2, -2.5) to (-2/2, 1);
\node at (-2/2+0.25,-1.75) {$\rightarrow$};
\node at (-1/2+0.25,-1.75) {$\rightarrow$};
\node at (0.25,-1.75) {$\rightarrow$};
\node at (1/2+0.25,-1.75) {$\rightarrow$};
\draw[ fill = white] (2/2,-2.25) rectangle (2/2+0.5,-1.75);
\draw[ fill = white] (2/2,-1.75) rectangle (2/2+0.5,-1.25);
\node at (2/2+0.25,-1.5) {$q$};
\node at (2/2+0.25,-2) {$\uparrow$};
\node at (3/2+0.25,-1.75) {$\leftarrow$};
\node at (4/2+0.25,-1.75) {$\leftarrow$};
\draw[ dashed, thick ] (5/2, -2.5) to (5/2, 1);

\end{scope}

\begin{scope}[shift={(-1.5,-0.75)}]

\node at (1.5,0.0)  {$u_6$};
\node at (1.3514533018536288,0.6508256086763372)  {$u_5$};
\node at (0.9352347027881004,1.1727472237020446)  {$u_4$};
\node at (0.3337814009344717,1.4623918682727355)  {$u_3$};
\node at (-0.3337814009344715,1.4623918682727355)  {$u_2$};
\node at (-0.9352347027881003,1.1727472237020449)  {$u_1$};
\node at (-1.3514533018536286,0.6508256086763373)  {$u_0$};
\node at (-1.5,0.0) {$v_0$};
\node at (-1.3514533018536288,-0.650825608676337) {$v_1$};
\node at (-0.9352347027881005,-1.1727472237020446) {$v_2$};
\node at (-0.3337814009344719,-1.4623918682727355) {$v_3$};
\node at (0.33378140093447134,-1.4623918682727355) {$v_4$};
\node at (0.9352347027881001,-1.1727472237020449) {$v_5$};
\node at (1.3514533018536286,-0.6508256086763375) {$v_6$};

\draw[->, thick] (0,0) to (0.7481877622304803,0.9381977789616357);
\draw [fill = white] (0,0) circle (0.3);
\node at (0,0) {$q$};

\end{scope}

\begin{scope}[shift={(3,-0.75)}]


\node at (1.5,0.0)  {$w_4$};
\node at (1.3514533018536288,0.6508256086763372)  {$w_3$};
\node at (0.9352347027881004,1.1727472237020446)  {$w_2$};
\node at (0.3337814009344717,1.4623918682727355)  {$w_1$};
\node at (-0.3337814009344715,1.4623918682727355)  {$w_0$};
\node at (-0.9352347027881003,1.1727472237020449)  {$u_1$};
\node at (-1.3514533018536286,0.6508256086763373)  {$u_0$};
\node at (-1.5,0.0) {$v_0$};
\node at (-1.3514533018536288,-0.650825608676337) {$v_1$};
\node at (-0.9352347027881005,-1.1727472237020446) {$v_2$};
\node at (-0.3337814009344719,-1.4623918682727355) {$v_3$};
\node at (0.33378140093447134,-1.4623918682727355) {$v_4$};
\node at (0.9352347027881001,-1.1727472237020449) {$v_5$};
\node at (1.3514533018536286,-0.6508256086763375) {$w_5$};

\draw[->, thick] (0,0) to (0.74818776223048,-0.9381977789616358);
\draw [fill = white] (0,0) circle (0.3);
\node at (0,0) {$r$};

\end{scope}

\begin{scope}[shift={(6,0)}]

\foreach \i in {-2,...,4}{
	\draw[ fill = white] (\i/2,0) rectangle (\i/2+0.5,0.5);
	\draw[ fill = white] (\i/2,-1) rectangle (\i/2+0.5,-0.5);
	\draw[ fill = white] (\i/2,-2) rectangle (\i/2+0.5,-1.5);
}

\node at (-2/2+0.25,0.25) {$u_0$};
\node at (-1/2+0.25,0.25) {$u_1$};
\node at (0.25,0.25) {$w_0$};
\node at (1/2+0.25,0.25) {$w_1$};
\node at (2/2+0.25,0.25) {$w_2$};
\node at (3/2+0.25,0.25) {$w_3$};
\node at (4/2+0.25,0.25) {$w_4$};

\node at (-2/2+0.25,-0.75) {$v_0$};
\node at (-1/2+0.25,-0.75) {$v_1$};
\node at (0.25,-0.75) {$v_2$};
\node at (1/2+0.25,-0.75) {$v_3$};
\node at (2/2+0.25,-0.75) {$v_4$};
\node at (3/2+0.25,-0.75) {$v_5$};
\node at (4/2+0.25,-0.75) {$w_5$};

\draw[ dashed, thick ] (-2/2, -2.5) to (-2/2, 1);
\node at (-2/2+0.25,-1.75) {$\rightarrow$};
\node at (-1/2+0.25,-1.75) {$\rightarrow$};
\node at (0.25,-1.75) {$\rightarrow$};
\node at (1/2+0.25,-1.75) {$\rightarrow$};
\draw[ fill = white] (3/2,-2.25) rectangle (3/2+0.5,-1.75);
\draw[ fill = white] (3/2,-1.75) rectangle (3/2+0.5,-1.25);
\node at (3/2+0.25,-1.5) {$r$};
\node at (3/2+0.25,-2) {$\downarrow$};
\node at (2/2+0.25,-1.75) {$\rightarrow$};
\node at (4/2+0.25,-1.75) {$\leftarrow$};
\draw[ dashed, thick ] (5/2, -2.5) to (5/2, 1);

\end{scope}

\draw[, ->, thick] (0.25,0) to (1.25,0);
\node at (0.75,0.25) {$T$};

\node at (0.8,-3) {$f_T(u_2u_3.u_4u_5u_6v_6,q) = (w_0w_1.w_2w_3w_4w_5,r,4)$};

\end{tikzpicture}

%% file: snake_tile_code.tex
\begin{tikzpicture}[scale=0.50]

\draw[fill = blue!10] (0,0) rectangle (7,7);
\draw[fill = blue!30] (0,0) rectangle (4,1);
\draw[fill = red!10] (1,1) rectangle (6,6);
\draw[fill = green!30] (2,2) rectangle (5,5);
\foreach \i in {0,...,5 }{
	\node at (0.5,\i+0.5) {$1$};
	\node at (6.5,1+\i+0.5) {$1$};
	\node at (\i+1.5,0.5) {$1$};
	\node at (\i+0.5,6.5) {$1$};
}

\foreach \i in {1,...,4 }{
	\node at (1.5,\i+0.5) {$0$};
	\node at (5.5,1+\i+0.5) {$0$};
	\node at (\i+1.5,1.5) {$0$};
	\node at (\i+0.5,5.5) {$0$};
}
\node at (2.5,2.5) {$l_1$};
\node at (3.5,2.5) {$l_2$};
\node at (2.5,3.5) {$r_1$};
\node at (3.5,3.5) {$r_2$};
\node at (2.5,4.5) {$b_1$};
\node at (3.5,4.5) {$b_2$};
\node at (4.5,4.5) {$b_3$};
\node at (4.5,3.5) {$b_4$};
\node at (4.5,2.5) {$b_5$};
\draw (0,0) grid (7,7);
\end{tikzpicture}

%% file: snake_macrotile_code.tex
\begin{tikzpicture}[scale=0.15]

\def\baldosita{
\draw[fill = blue!10, draw = black!80] (0,0) rectangle (7,7);
\draw[fill = red!10, draw = black!50] (1,1) rectangle (6,6);
\draw[fill = green!30, draw = black!50] (2,2) rectangle (5,5);
\draw[fill = green!70, draw = none] (2,2) rectangle (4,4);
\draw[fill = green!10, draw = none] (4,2) rectangle (5,3);
\draw[fill = none, draw = black!50] (2,2) rectangle (5,5);
}

\node at (-5,21) {$\mathcal{M}(t) =$};
\foreach \i in {0,1,2,3,4,5}{
	
	\foreach \j in {0,1,2,3,4,5}{
		
		\begin{scope}[shift = {(7*\i,7*\j)}]
		\baldosita;
		\end{scope}
	}
}

\draw[fill = black!80] (4,2) rectangle (5,3);

\draw[thick, ->, red] (1,3) to (6,3);
\draw[thick, ->, blue] (3,1) to (3,6);
\foreach \i in {1,2,3,4,5}{
		\begin{scope}[shift = {(7*\i,0)}]
		\draw[thick, ->, red] (1,3) to (6,3);
		\draw[thick, <-, blue] (1,4) to (6,4);
		\end{scope}
		\begin{scope}[shift = {(0,7*\i)}]
		\draw[thick, ->, blue] (3,1) to (3,6);
		\draw[thick, <-, red] (4,1) to (4,6);
		\end{scope}
}

\end{tikzpicture}